\documentclass[11pt,a4paper,reqno]{amsart}
\usepackage{a4wide}
\usepackage{amsaddr}
\usepackage{amsmath}
\usepackage{amsfonts, amssymb, amsthm}
\usepackage{paralist}
\usepackage[colorlinks=true]{hyperref}
\hypersetup{urlcolor=blue, citecolor=red}
\usepackage{graphics,graphicx,color}

\makeatletter
\renewcommand{\email}[2][]{%
	\ifx\emails\@empty\relax\else{\g@addto@macro\emails{,\space}}\fi%
	\@ifnotempty{#1}{\g@addto@macro\emails{\textrm{(#1)}\space}}%
	\g@addto@macro\emails{#2}%
}
\makeatother

\theoremstyle{definition}

\numberwithin{equation}{section}

\newcommand{\R}{\mathbb{R}}

\newtheorem{Thm}{Theorem}[section]
\newtheorem{Lem}{Lemma}[section]
\newtheorem{Cor}[Thm]{Corollary}
\newtheorem{Prop}{Proposition}[section]

\theoremstyle{definition}

\newtheorem{Rem}{Remark}[section]

\newtheorem*{Not}{Notations}
\newtheorem{Assum}{Assumption}[section]
\newtheorem{Claim}{Claim}

\allowdisplaybreaks
\usepackage[normalem]{ulem}
\usepackage{color}
\begin{document}
\title[ Asymptotic behavior for the Brezis-Nirenberg problem]
{Asymptotic behavior for the Brezis-Nirenberg problem. The subcritical perturbation case}
\author{Jinkai Gao \ and \ Shiwang Ma$^{ \rm *}$}\email{jinkaigao@mail.nankai.edu.cn, shiwangm@nankai.edu.cn}
\address{School of Mathematical Sciences and LPMC, Nankai University\\ 
	Tianjin 300071, China}

\thanks{%This work is partly supported by the National Natural Science Foundation of China (No. 12301128).\\
  {$^{\rm *}$ Corresponding author.}
}


%\emph{\bf 2020 Mathematics Subject Classification:}  35J60, 35B09, 35B33, 35B40, 35A02.

%\subjclass[2010]{Primary 35J60, 35Q55; Secondary 35B25, 35B40, 35R09, 35J91}

\date{}

\begin{abstract}
In this paper, we are concerned with the well-known  Brezis-Nirenberg problem
\begin{equation*}
\begin{cases}
-\Delta u= u^{2^*-1}+\varepsilon u^{q-1},\quad
u>0, &{\text{in}~\Omega},\\
\quad \  \ u=0, &{\text{on}~\partial \Omega},
\end{cases}
\end{equation*}
where $\Omega\subset \mathbb R^N$ with $N\ge 3$ is a bounded domain,  $q\in(2,2^*)$ and $2^*=\frac{2N}{N-2}$ denotes the critical Sobolev exponent. It is well-known (H. Br\'{e}zis and L. Nirenberg, 
\newblock {\em Comm. Pure Appl. Math.}, 36(4):437--477, 1983) that the above problem admits a positive least energy solution for all $\varepsilon >0$ and $q>\max\{2,\frac{4}{N-2}\}$. In the present paper, we first analyze the asymptotic behavior of the positive least energy solution as $\varepsilon\to 0$ and establish a sharp asymptotic characterisation of the profile and blow-up rate of the least energy solution. Then, we prove the uniqueness and nondegeneracy of the least energy solution under some mild assumptions on domain $\Omega$. The main results in this paper can be viewed as a generalization of the results for $q=2$ previously established in the literature. But the situation is  quite different from the case $q=2$, and the blow-up rate not only heavily depends on  the space dimension $N$ and the geometry of the domain $\Omega$, but also depends on the exponent $q\in(\max\{2,\frac{4}{N-2}\}, 2^*)$ in a non-trivial way.  
\end{abstract}

\maketitle
%\tableofcontents

%\newpage

%%%%%%%%%%%%%%%%%%%%%%%%%%%%%%%%%%%%%%%%%%%%%%%%%%%%%%%%%%%%%%%%%%%%%%%%%%%%%%%%%%%%%%%%%%%%%%%%%%%%%%%%%%%%%%%%%%%%%%%%%%%%

\section{Introduction and main results}
\setcounter{equation}{0}
In this paper, we consider the following Brezis-Nirenberg problem
\begin{equation}\label{p-varepsion}\tag{$P_{\varepsilon}$}
\begin{cases}
-\Delta u= u^{2^*-1}+\varepsilon u^{q-1}, \quad % &{\text{in}~\Omega},\\
 u>0, &{\text{in}~\Omega},\\
\quad \  \ u=0, &{\text{on}~\partial \Omega},
\end{cases}
\end{equation}
where $N\geq 3$, $q \in [2,2^*)$, $2^*:=\frac{2N}{N-2}$, $\varepsilon>0$ is a parameter, $\Omega$ is a smooth bounded domain in $\R^N$.

In 1983, in their celebrated paper \cite{Brezis1983CPAM},  Brezis and Nirenberg  proved that in the case that $q=2$,  if $N\geq 4$, problem (\ref{p-varepsion}) admits a positive least energy solution for all $\varepsilon\in (0,\lambda_{1})$, where $\lambda_{1}$ is the first eigenvalue of $-\Delta$ with zero Dirichlet boundary condition; if $N=3$, there exists $\lambda_{*}\in (0,\lambda_{1})$ such that (\ref{p-varepsion}) has a positive least energy solution for all $\varepsilon\in (\lambda_{*},\lambda_{1})$ and no least energy solution exists for $\varepsilon \in (0,\lambda_{*})$. 
However, in the case that $q\in (2,2^*)$, the situation is quite different in nature. Brezis and Nirenberg  \cite{Brezis1983CPAM}  proved that if  one of the following conditions holds
\begin{enumerate}%[label=\arabic*.]
    \item $N\geq 3$, $q\in(\max\{2,\frac{4}{N-2}\}, 2^*)$ and every $\varepsilon>0$;
   % \item $N=3$, $q\in (4,6)$ and $\varepsilon>0$;
    \item $N=3$, $q\in (2,4]$ and $\varepsilon$ large enough,
\end{enumerate}
then the problem \eqref{p-varepsion} has a positive least energy solution.

Once the solvability of problem (\ref{p-varepsion}) has be established, it is interesting to consider the asymptotic behavior of solutions. It is worth noting that as $\varepsilon\to0$, the formal limit for equation (\ref{p-varepsion}) is the critical Emden-Fowler equation 
\begin{equation}\label{p-0}\tag{$P_{0}$}
\begin{cases}
-\Delta u= u^{2^*-1},\quad u>0, &{\text{in}~\Omega},\\
\quad \  \ u=0, &{\text{on}~\partial \Omega},
\end{cases}
\end{equation}
and it is well known \cite{Pohozaev1965} that problem \eqref{p-0} admits no non-trivial solutions, if $\Omega$ is star-shaped. While, Bahri and Coron \cite{Bahri1988CPAM} gave an existence result, if the domain has a non-trivial topology, see also Dancer \cite{dancer1988}, Ding \cite{weiyueding} and Passaseo \cite{passaseo1989multiplicity,passaseo1994} for other existence results of (\ref{p-0}). In particularly, when $\Omega=\mathbb{R}^{N}$, 
we can give all the  solutions of (\ref{p-0}) (see \cite{AUBIN, Talenti}) by
\begin{equation}\label{Aubin-Talenti bubble}
    \alpha_{N}U_{\lambda,x}(y):=(N(N-2))^{\frac{N-2}{4}}\left(\frac{\lambda}{1+\lambda^{2}|y-x|^{2}}\right)^{\frac{N-2}{2}},
\end{equation}
which called Aubin-Talenti bubble at $x\in\R^{N}$ with height $\lambda\in\mathbb{R}^{+}$ and $\alpha_{N}=(N(N-2))^{\frac{N-2}{4}}$. We also recall the following minimization problem corresponding to problem (\ref{p-0})
\begin{equation}\label{definition of S}
    S(\Omega):=\inf\limits_{u\in\mathcal{D}^{1,2}_{0}(\Omega)\setminus\{0\}}\frac{\int_{\Omega}|\nabla u|^{2}dx}{\left(\int_{\Omega}|u|^{2^*}dx\right)^{\frac{2}{2^*}}}.
\end{equation}
It is known that \cite{willem2012minimax}  $S(\Omega)=S(\mathbb{R}^{N})$ ( we will use $S$ for simplification) and $S(\Omega)$ is never achieved except when $\Omega=\mathbb R^{N}$. %Moreover, the best Sobolev constant $S(\R^{N})$ 
% \begin{equation}
%     S=\pi N(N-2)\left(\frac{\Gamma(N/2)}{\Gamma(N)}\right)^{2/N},
% \end{equation}
% where $\Gamma (x)$ is the Gamma function (see \ref{defintion of Gamma function and Beta function}).

Before recalling some more known results, it is useful to introduce some notations. Let $G(x,\cdot)$ denote the Green's function of the negative Laplacian on $\Omega$, i.e.
\begin{equation}
\begin{cases}
-\Delta G(x,\cdot)= \delta_x, &{\text{in}~\Omega}, \\
\quad \  \ G(x,\cdot)=0, &{\text{on}~\partial\Omega},
\end{cases}
\end{equation}
where $\delta_x$ is the Dirac function at $x\in\Omega$. For $G(x,y)$, we have the following form
\begin{equation*}
G(x,y)=S(x,y)-H(x,y), ~(x,y)\in \Omega\times \Omega,
\end{equation*}
where (if $N\geq 3$)
\begin{equation}
S(x,y)=\frac{1}{(N-2)\omega_{N}|x-y|^{N-2}},
\end{equation}
is the singular part, which is also the fundamental solution to negative Laplace equation, $\omega_{N}$ is a measure of the unit sphere in $\R^N$ and $H(x,y)$ is the regular part of $G(x,y)$ satisfying 
\begin{equation}
 \begin{cases}
-\Delta H(x,\cdot)=0, &{\text{in}~\Omega}, \\
\quad \  \ H(x,\cdot)=S(x,\cdot), &{\text{on}~\partial\Omega}.
\end{cases}   
\end{equation}
For any $x\in \Omega$, we denote the leading term of $H$ as
\begin{equation}\label{definition of Robin function}
    R(x):=H(x,x),
\end{equation}
 which is called the Robin function of $\Omega$ at the point $x$.

Now, we recall some known results on the  asymptotic behavior of solutions for \eqref{p-varepsion}. When $q=2$, Han \cite{Han1991} and Rey \cite{Rey1989ProofOT} proved independently that the one-peak solutions of (\ref{p-varepsion}) blow-up and concentrate at a critical point of Robin function, which has be conjectured by Brezis and Peletier \cite{Brezispletier1989} previously. Futhermore, Takahashi \cite{FunkcialajEkvacioj2004,Takahashivariablecoefficients2006} and Frank-K\"onig-Kova\v{r}\'{i}k \cite{Frank2019Energythree-dimensional,Frank2019Energyhigher-dimensional} give a precise energy expansion on the least energy and prove that blow-up point is a minimum point of the Robin function.  For  the asymptotic behavior of ground state solutions of the Schr\"odinger equation and Choquard equation  in the whole space $R^N$ and related topics,  we refer the readers to \cite{Ma-1,Ma-2} and reference therein.

Motivated by the results mentioned above, it is natural to consider the case that $q\in (2,2^*)$ and our first result states as follows.
 \begin{Thm}\label{main theorem-1}
Assume $N\geq 3$, $q\in (\max\{2,\frac{4}{N-2}\},2^*)$, $\Omega$ is a smooth bounded star-shaped domain  in $\mathbb{R}^{N}$. If $u_{\varepsilon}$ is a least energy solution of (\ref{p-varepsion}) and $x_{\varepsilon}$ is the maximum point of $u_{\varepsilon}$, then as $\varepsilon\to 0$ (up to a subsequence)
\end{Thm}
\begin{enumerate}
    \item $\{u_{\varepsilon}\}_{\varepsilon>0}$ is a blow-up, minimizing sequence of the best Sobolev constant $S$, i.e. 
    \begin{equation}
     \|u_{\varepsilon}\|_{L^{\infty}(\Omega)}\to\infty\text{~~and~~}\frac{\int_{\Omega}|\nabla u_{\varepsilon}|^{2}dx}{(\int_{\Omega}|u_{\varepsilon}|^{2^*}dx)^{\frac{2}{2^*}}}\to S.
    \end{equation}
     \item Let $u_{\varepsilon}(x)=0$, if $x\in\R^{N}\setminus\Omega$, then
    \begin{equation}
        \frac{1}{\|u_{\varepsilon}\|_{L^{\infty}(\Omega)}}u_{\varepsilon}\left(\frac{x}{\|u_{\varepsilon}\|_{L^{\infty}(\Omega)}^{(2^*-2)/2}}+x_{\varepsilon}\right)\to U\text{~~in~~}\mathcal{D}^{1,2}(\R^{N})\text{~and~}C^{2}_{loc}(\R^{N}),
    \end{equation}
    where 
    $$U=\left(\frac{N(N-2)}{N(N-2)+|x|^{2}}\right)^{\frac{N-2}{2}}.$$
    \item The best upper bound of $u_{\varepsilon}$ is
    \begin{equation}
        u_{\varepsilon}(x)\leq C\frac{\|u_{\varepsilon}\|_{L^{\infty}(\Omega)}}{(N(N-2)+\|u_{\varepsilon}\|_{L^{\infty}(\Omega)}^{2^*-2}|x-x_{\varepsilon}|^{2})^{\frac{N-2}{2}}}.
    \end{equation}
    \item The maximum point $x_{\varepsilon}$ of $u_{\varepsilon}$  concentrates in an interior point $x_0$ of $\Omega$, which a critical point of Robin function $R(x)$, i.e.
    \begin{equation}\label{The location of blow-up point in main theorem 1}
       x_{\varepsilon}\to x_{0}\in\Omega\text{~~and~~} \nabla R(x_{0})=0.
    \end{equation}  
    Moreover, if $x_{0}$ is a nondegenerate critical point of $R(x)$ and $q\in[\frac{5}{2},4)$ if $N=4$, then
    \begin{equation}\label{contentrate speed}
        |x_{\varepsilon}-x_{0}|=O(\varepsilon^{\frac{2}{(N-2)q-4}}).
    \end{equation}
    \item We have the following convergence 
    \begin{equation}
        u_{\varepsilon}\to 0\quad \text{~in~} C^{1}(\Omega\setminus\{x_{0}\}),
    \end{equation}
    and
       \begin{equation}
|\nabla u_{\varepsilon}|^{2}\stackrel{*}{\rightharpoonup} S^{\frac{N}{2}}\delta_{x_{0}} \text{~as~} \varepsilon\to 0, 
    \end{equation}
in the sense of Radon measures of the compact space $\bar{\Omega}$, where $\delta_{x_{0}}$ is the Dirac measure supported by $x_{0}$. Moreover, we have
 \begin{equation}
\|u_{\varepsilon}\|_{L^{\infty}(\Omega)}u_{\varepsilon}(x)\to\frac{1}{N} \alpha_{N}^{2^*}\omega_{N}G(x,x_{0})\text{~in~}C^{1,\alpha}(\omega),    
\end{equation}
where $\omega_{N}$ denotes the measure of unit sphere in $\R^{N}$ and $\omega$ is a neighbourhood  of $\partial\Omega$ not containing $x_{0}$.
    \item The exact blow-up rate is
    \begin{equation}\label{exact blow-up rate}
     \varepsilon \|u_{\varepsilon}\|_{L^{\infty}(\Omega)}^{q+2-2^*}\to \alpha_{N,q}R(x_{0}),
    \end{equation}
where 
\begin{equation}
\alpha_{N,q}=\frac{2q}{2^*-q}\frac{\alpha_{N}^{2^*}\omega_{N}}{N^{2}}\frac{\Gamma(\frac{N-2}{2}q)}{\Gamma(\frac{N}{2})\Gamma(\frac{N-2}{2}q-\frac{N}{2})}.
\end{equation}
\item Let $S_{\varepsilon}$ be the least energy of $u_{\varepsilon}$, then 
\begin{equation}
   \frac{1}{N}S^{\frac{N}{2}}-S_{\varepsilon}\sim \varepsilon^{\frac{2N-4}{(N-2)q-4}}.
\end{equation}
\end{enumerate}
\begin{Rem}
\begin{enumerate}
    \item  From the conclusion in Theorem \ref{main theorem-1}, we know that if $u_{\varepsilon}$ is a sequence of blow-up solutions with only one-peak, then the concentrate point $x_{0}$ is a critical point of Robin function $R(x)$. Conversely, by using the Lyapounov-Schmidt procedure, Rey \cite{Rey1990} and Molle \cite{Molle2003} proved that any isolated,
 nondegeneracy critical point  $x_{0}$ of Robin function $R(x)$ generate a family of solutions which concentrate at $x_{0}$. 
 \item From the blow-up rate obtained in Theorem \ref{main theorem-1}, it's easy to see that the restriction on $(N,q)$ in Theorem \ref{main theorem-1} is almost optimal. In addition, the restriction on $\Omega$ in Theorem \ref{main theorem-1} is used to prove that the limit equation \eqref{p-0} has no solution, and then  $u_{\varepsilon}$ must blow-up. In fact, there also exists some non-starshaped domain, such that \eqref{p-0}  has no solution, see \cite{RODRIGUEZ1992243} for more details.
\end{enumerate}

\end{Rem}

On the other hand, it's also a very profound topic to study the uniqueness and nondegeneracy of the least energy solution, which is more complicate and is known depending on the shape of domain $\Omega$. In the following, we will state some known results on uniqueness and nondegeneracy. 

For the uniqueness problem: when $q=2$, if $\Omega$ is a ball in $\R^{N}$, it was proved that the solution is unique in the case $N\geq4$ and for $\frac{\lambda_{1}}{4}<\varepsilon<\lambda_{1}$ if $N=3$, see \cite{Liqua1992UniquenessOP,Srikanth1993UniquenessOS,Adimurthi}. On the other hand, Cerqueti \cite{Cerqueti1999AUR} proved the uniqueness for $N\geq 5 $ and $\varepsilon$ small enough, under the assumption that domain $\Omega$ is both symmetric and convex with respect to $N$ orthogonal direction. As for non-convex domain, Glangetas \cite{Glangetas1993UniquenessOP} established another uniqueness result for $N\geq 5$ with an assumption that the blow-up point is a nondegenerate critical point of the Robin function.  When $q\in (2,2^*)$, if $\Omega$ is a ball, the uniqueness has be almost resolved by Erbe and Tang \cite{Erbe1997UniquenessTF} and Chen and Zou \cite{Chen2012OnTB}, for every $\varepsilon>0$ if $N\geq 6$; for almost every $\varepsilon>0$ if $N=4,5$, and for almost every $\varepsilon>\mu_{0}$ if $N=3$, where $\mu_{0}$ is a sufficiently large constant. 

For the nondegeneracy problem: When $q=2$, it has been proved by Cerqueti \cite{Cerqueti1999AUR}, Grossi \cite{Grossi2005ANR} and Takahashi \cite{Takahashi2008nondegeneracy} under the assumption that $\Omega$ is a convex and symmetric domain or blow-up point is a nondegeneracy critical point of the Robin function.

Inspired by the above results, a natural question arises, is the least energy solution of (\ref{p-varepsion}) with $q\in (2,2^*)$ unique and nondegenerate ? In the rest of the paper, we focuses on this issue. Our second and third results state as follows.

\begin{Thm}\label{main theorem 2}
Let $\Omega$ be a smooth bounded star-shaped domain of $\R^{N}$, $u_{\varepsilon}$ and $v_{\varepsilon}$ are two least energy solutions of (\ref{p-varepsion}) which concentrate at the same point $x_{0}$. If  $(N,q,x_{0},\Omega)$ satisfy one of the following conditions:
\begin{enumerate}%[label=$(\arabic*)$]
    \item $N\geq 5$, $q\in (2,2^*)$ and $\Omega$ is convex in the $x_{i}$ directions and symmetric with respect to the hyperplanes $\{x_{i}=0\}$ for $i=1,\cdots, N$;
    \item $N\geq 4$, $q\in (2,2^*)$, $q\geq \frac{N+2}{N-2}$ and $x_{0}$ is a nondegenerate critical point of Robin function $R(x)$.
\end{enumerate}
Then, there exists a constant $\varepsilon_{0}>0$ such that for any $\varepsilon<\varepsilon_{0}$
\begin{equation}
    u_{\varepsilon}\equiv v_{\varepsilon}\quad \text{~in~}\Omega.
\end{equation}
\end{Thm}

\begin{Thm}\label{asymptotic nondegeneracy-domain and Robin function}
    Let $\Omega$ be a smooth bounded star-shaped domain of $\R^{N}$, $N\geq 5$, $q\in(2,2^*)$, $u_{\varepsilon}$ is a least energy solution of (\ref{p-varepsion}) and $x_{0}$ is the concentration point of $u_{\varepsilon}$. If domain $\Omega$ is convex in the $x_{i}$ directions and symmetric with respect to the hyperplanes $\{x_{i}=0\}$ for $i=1,\cdots, N$ or concentration point $x_{0}$ is a non-degenerate critical point of Robin function $R(x)$. Then, the least energy solution is nondegenerate, i.e. the linear problem
    \begin{equation}\label{nondegenerate-1}
        \begin{cases}
            -\Delta v=(2^*-1)u_{\varepsilon}^{2^*-2}v+\varepsilon (q-1)u_{\varepsilon}^{q-2}v, &\text{~~in~~}\Omega,\\
            \quad \  \ v=0,&\text{~~on~~}\partial\Omega,
        \end{cases}
    \end{equation}
    admits only the trivial solution $v\equiv0$.
\end{Thm}
\begin{Rem}
\begin{enumerate}
    \item The proof of Theorem $\ref{main theorem 2}$ relies on a blow-up technique (see \cite{Grossi2000ADE,Cerqueti1999AUR}) and local Poho\v{z}aev identity (see \cite{Cao2015UniquenessOP,cao2021Trans,Deng2015OnTP}).
    \item One would like to point out that the assumption on $\Omega$ in the first case of Theorem \ref{main theorem 2} is very strong. For one thing, from the results in \cite{Caffarelli1985Convexity,Cardaliaguet2002convexity,Grossi2010Nonexistence}, we know that the multiple blow-up can't occur and $0$ is a unique blow-up point, thus the restriction on least energy solution or one-peak solution in Theorem \ref{main theorem 2} can be removed, see also \cite{Cerqueti2001Localestimates} for a similar result. For another thing, from the results in \cite{Caffarelli1985Convexity,Cardaliaguet2002convexity,GROSSInondegeneracy}, it's easy to see that $0$ is a unique and nondegenerate critical point of Robin function, thus the first case in Theorem \ref{main theorem 2} can be regarded as a special case of the second case, while the restriction on $q$ is different, when $N=5$.  
    \item To prove Theorem $\ref{asymptotic nondegeneracy-domain and Robin function}$, we adopt the method used in \cite{Grossi2005ANR,Takahashi2008nondegeneracy} which does not depend on the decomposition of the least energy solution (see Lemma $\ref{decomposition of u-varepsilon}$). This method can deal with both convex and non-convex domain, but only for $N\geq 5$ and $q\in(2,2^*)$. In fact, if $N\geq 4$, $q\in(2,2^*)$, $q\geq \frac{N+2}{N-2}$ and $x_{0}$ is a nondegeneracy critical point of Robin function, we can prove the nondegeneracy by a similar proof as Theorem $\ref{main theorem 2}$ and we omit it.
    %\item The restriction on dimension and exponent $q$ in Theorem $\ref{main theorem 2}$ and Theorem $\ref{asymptotic nondegeneracy-domain and Robin function}$ is due to a technique reason and may be improved.
\end{enumerate}
\end{Rem}

\begin{Rem}
Compared with the previous work, there are some features of this paper:
\begin{enumerate}
    \item The blow-up rate obtained in Theorem \ref{main theorem-1} extends the earlier results of the special case $q=2$ in \cite{Brezispletier1989,Frank2021BlowupOS, Han1991,Rey1989ProofOT,Rey1990}. Thus, we can have a clearer understanding of the dependence of blow-up rate on spatial dimension $N$ and subcritical exponent $q$.
    \item The estimate of the least energy can't calculate directly as in \cite{Frank2019Energyhigher-dimensional,FunkcialajEkvacioj2004,Takahashivariablecoefficients2006}. Since, when $q=2$, the least energy solution can be constructed by a constrained minimization problem, while it's not hold for the case $q\in (2,2^*)$, instead a mountain pass lemma has been used.
    \item The restriction on $q$ in Theorem $\ref{main theorem 2}$ and Theorem $\ref{asymptotic nondegeneracy-domain and Robin function}$ is important in some sense. In fact, from the results in Theorem \ref{main theorem-1}, we can conclude that if $\{u_{\varepsilon}\}$ is a sequence of least energy solution with $\|u_{\varepsilon}\|_{L^{\infty}(\Omega)}=u_{\varepsilon}(x_{\varepsilon})$ and $x_{\varepsilon}\to x_{0}$. Then $x_{0}\in\Omega$, $\lim\limits_{\varepsilon\to0}\varepsilon\|u_{\varepsilon}\|^{q+2-2^*}_{L^{\infty}(\Omega)}:=k_{0}^{2^*-q-2}<\infty$ and $(k_{0},x_{0})$ is a critical point of the function 
    \begin{equation}
        F(k,x)=\frac{1}{2}R(x)k^{2}-\frac{1}{\alpha_{N,q}(2^*-q)}k^{2^*-q}.
    \end{equation}
    If $q\geq \frac{N+2}{N-2}$, then $\mathcal{D}^{2}_{k}F(k,x_{0})$ is strictly positive definite for every $k$, thus $F(k,x_{0})$ is convex in the variable $k$ and has a unique critical point $k_{0}$. Hence, the asymptotic behavior is uniquely determined in term of $\Omega,N,q$ and $x_{0}$.
   \item Although the main idea of this paper comes from previous literature,  it is more complicated to deal with the case $q\in(2,2^*)$, compared with the case $q=2$, and we need some refined argument, which is related to the choice of $q$.
\end{enumerate}
\end{Rem}
\begin{Rem}
\begin{enumerate}
    \item  The main results in this paper are also hold for one-peak solutions. But, it's still open that whether or not the blow-up point is a minimum point of Robin function, as showed in \cite{FunkcialajEkvacioj2004,Takahashivariablecoefficients2006} when $q=2$, where the restriction on the least energy solution is essential.
    \item Recently,  in Frank, K{\"o}nig and Kovař{\'i}k \cite{Frank2021BlowupOS}, and K{\"o}nig and Laurain \cite{konig2023multibubbleblowupanalysisbrezisnirenberg}, 
    %Frank and K\"onig \cite{Frank2021BlowupOS, konig2023multibubbleblowupanalysisbrezisnirenberg} 
     a complete picture of blow-up phenomena is given for $N=3$ and $q=2$. While, it is not fully understood that whether or not the solution exists, when $N=3,q\in(2,4]$ and $\varepsilon$ not large enough. In particular, when $\Omega$ is a ball,  the existence and uniqueness of the solutions has been suggested in \cite{Brezis1983CPAM} by numerical computations,
    \begin{enumerate}
    \item If $2<q<4$, there is some $\varepsilon_{0}>0$ such that
    \begin{enumerate}
        \item for $\varepsilon>\varepsilon_{0}$, there are two solutions;
        \item for $\varepsilon=\varepsilon_{0}$, there is a unique solution;
        \item for $\varepsilon<\varepsilon_{0}$, there is no solutions.
    \end{enumerate}
    \item If $q=4$, there is some $\varepsilon_{0}>0$ such that
    \begin{enumerate}
        \item for $\varepsilon>\varepsilon_{0}$, there is a unique solution;
        \item for $\varepsilon\leq \varepsilon_{0}$, there is no solution.
    \end{enumerate}
\end{enumerate}
and the case $q\in (2,4)$ has be confirmed by Atkinson-Peletier \cite{AtkinsonPeletier} afterwards. On the hand, when $q=4$, even if the existence and uniqueness of the solution suggested above hold, the leading order of the speed at which blow-up solution concentrate as $\varepsilon\searrow \varepsilon_{0}$ is hard to obtain and the leading order of $\|u_{\varepsilon}\|_{\infty}$ can no longer be captured by the right hand side of (\ref{exact blow-up rate}), that is, the blow-up rate obtained in Theorem \ref{main theorem-1} is no longer applicable. We will study this in a forthcoming work.
\end{enumerate}
\end{Rem}
Our paper is organized as follows. In section \ref{Asymptotic behavior}, we show the asymptotic behavior of the least energy solution by using the method of blow-up analysis. In section \ref{Asymptotic uniqueness} and section \ref{Asymptotic nondegeneracy}, we establish the uniqueness and nondegeneracy by using the local Poho\v{z}aev identity and the results obtained in section \ref{Asymptotic behavior}. Finally,  we give some known facts in the Appendix.

\begin{Not}
Throughout this paper, we use $\|u\|_{H^{1}_{0}(\Omega)}:=\|\nabla u\|_{L^{2}(\Omega)}$ to denote the norm in $H^{1}_{0}(\Omega)$  and $\langle\cdot,\cdot\rangle$ to mean the corresponding inner product. The homogeneous Sobolev space $\mathcal{D}^{1,2}(\R^{N})$ is defined as the completion of $C_{c}^{\infty}(\R^{N})$ with respect to the norm $\|\nabla u\|_{L^{2}(\R^{N})}$. In addition, $C$ denotes positive constant possibly different from line to line, $A=o(B)$ means $A/B\to 0$, $A=O(B)$ means that $|A/B|\leq C$ and we write $A\sim B$, if $A=O(B)$ and $B=O(A)$.
\end{Not}

%%%%%%%%%%%%%%%%%%%%%%%%%%%%%%%%%%%%%%%%%%%%%%%%%%%%%%%%%%%%%%%%%%%%%%%%%%%%%%%%%%%%%%%%%%%%%%%%%%%%%%%%%%%%%%%%%%%%%%%%%%

\section{Asymptotic behavior}\label{Asymptotic behavior} % as \texorpdfstring{$\varepsilon\to 0$}{}}
In this section, we will give a refined asymptotic behavior of the least energy solution and the main results in Theorem \ref{main theorem-1} are obtained. In the following, we always assume that: 
\begin{Assum}\label{assumption 1}
    $N\geq 3$, $q\in (\max\{2,\frac{4}{N-2}\},2^*)$, $\Omega$ is a  smooth bounded star-shaped domain in $\mathbb{R}^{N}$. 
\end{Assum}
%, moreover, by translation, we can assume that $0\in\Omega$ and $D:=\text{dist}(0,\partial\Omega)=\max_{x\in\Omega}\text{dist}(0,\partial\Omega)$
\subsection{Asymptotic behavior of \texorpdfstring{$u_\varepsilon$}{}}

Let  $u_{\varepsilon}$ be the positive least energy solution of (\ref{p-varepsion}) obtained in \cite{Brezis1983CPAM} with energy $0<S_{\varepsilon}<\frac{1}{N}S^{\frac{N}{2}}$, then $u_{\varepsilon}$ can be characterized as
\begin{equation}
    S_{\varepsilon}=\inf_{u\in N_{\varepsilon}}I_{\varepsilon}(u),
\end{equation}
where $N_{\varepsilon}$ is the Nehari manifold correspond to (\ref{p-varepsion})
\begin{equation}\label{Nehari identity for p-varepsion}
    N_{\varepsilon}:=\left\{u\in H^{1}_{0}(\Omega)\setminus\{0\},u\geq 0:\int_{\Omega}|\nabla u|^{2}dx=\int_{\Omega}|u|^{2^*}dx+\varepsilon\int_{\Omega}|u|^{q}dx\right\}.
\end{equation}
In addition, the least energy solution $u_{\varepsilon}$ is also known as a mountain pass solutions and
\begin{equation}
S_{\varepsilon}=\inf_{\gamma\in\Gamma}\max_{t\in[0,1]}I_{\varepsilon}(\gamma(t)),
\end{equation}
where
\begin{equation}
    \Gamma:=\left\{\gamma\in C([0,1],H^{1}_{0}(\Omega)):\gamma(0)=0, I_{\varepsilon}(\gamma(1))<0\right\},
\end{equation}
and $I_{\varepsilon}$ is the energy function of (\ref{p-varepsion}) defined by
\begin{equation}
    I_{\varepsilon}(u):=\frac{1}{2}\int_{\Omega}|\nabla u|^{2}dx-\frac{1}{2^*}\int_{\Omega}|u|^{2^*}dx-\frac{\varepsilon}{q}\int_{\Omega}|u|^{q}dx.
\end{equation}
On the other hand, by Corollary \ref{Pohozaev identity}, we have the following Poho\v{z}aev identity for (\ref{p-varepsion})
\begin{equation}\label{Pohazaev-p-varepsion}
\frac{1}{2N}\int_{\partial\Omega}|\nabla u|^{2}(x-y)\cdot n d S_{x}=\left(\frac{1}{q}-\frac{1}{2^*}\right)\varepsilon\int_{\Omega}|u|^{q}dx,
\end{equation}
for any $y\in\R^{N}$. 

Now, we show some properties of $u_{\varepsilon}$ as $\varepsilon\to 0$.
\begin{Lem}\label{limit of u-varepsilon}
 $\{u_{\varepsilon}\}_{\varepsilon>0}$ is bounded in $H^{1}_{0}(\Omega)$ and as $\varepsilon\to 0$
 \begin{equation}\label{limit of u-varepsilon proof 1}
    \begin{cases}
 u_{\varepsilon}\rightharpoonup 0 &\text{~weakly in~} H^{1}_{0}(\Omega),\\
    u_{\varepsilon}\to 0 &\text{ ~strongly in~} L^{p}(\Omega), \text{~for any~} p\in [1,2^*),\\
    u_{\varepsilon}\to 0 &\text{~almost everywhere in~} \R^{N},\\
 \end{cases}  
 \end{equation}
but $u_{\varepsilon}$ does not converge to $0$ strongly in $H^{1}_{0}(\Omega)$.
 \end{Lem}
\begin{proof}
Note that $u_{\varepsilon}\in N_{\varepsilon}$, where $N_{\varepsilon}$ is the Nehari manifold defined in (\ref{Nehari identity for p-varepsion}), then we have
\begin{equation}
   S_{\varepsilon}=I_{\varepsilon}(u_{\varepsilon})=\left(\frac{1}{2}-\frac{1}{q}\right)\int_{\Omega}|\nabla u_{{\varepsilon}}|^2dx+\left(\frac{1}{q}-\frac{1}{2^*}\right)\int_{\Omega}|u_{\varepsilon}|^{2^*}dx< \frac{1}{N}S^{\frac{N}{2}}.
\end{equation}
Thus by the Assumption \ref{assumption 1}, we can obtain that $\{u_{\varepsilon}\}_{\varepsilon>0}$ is bounded in $H^{1}_{0}(\Omega)$.  If we assume that $u_{\varepsilon}\rightharpoonup u_{0}$ weakly in $H^{1}_{0}(\Omega)$, then $u_{0}$ is a solution of (\ref{p-0}). Since $\Omega$ is starshaped, thus $u_{0}=0$ (see \cite{Pohozaev1965}). Moreover, by Rellich compact embedding theorem, the remaining results of (\ref{limit of u-varepsilon proof 1}) hold. Next, we claim that $u_{\varepsilon}$ does not converge to $0$ strongly in $H^{1}_{0}(\Omega)$, otherwise $u_{\varepsilon}\to 0 $ and $S_{\varepsilon}\to 0$ as $\varepsilon\to0$. To derive a contraction, we only need to show that there exist a constant $C>0$ such that $S_{\varepsilon}>C$ for any $\varepsilon\leq 1$. By Sobolev embedding, we know that there exist constant $C_{1}(N)$ and $C_{2}(N,q)$ such that for any $\varepsilon\leq 1$, 
\begin{equation}
\begin{aligned}
  I_{\varepsilon}(u)&=\frac{1}{2}\|\nabla u\|_{2}^{2}-\frac{1}{2^*}\|u\|_{2^*}^{2^*}-\frac{\varepsilon}{q}\|u\|_{q}^{q} \\
  &\geq \left(\frac{1}{2}-C_{1}\|\nabla u\|_{2}^{2^*-2}-\varepsilon C_{2}\|\nabla u\|_{2}^{q-2}\right)\|\nabla u\|_{2}^{2}.
\end{aligned}
\end{equation}
Then, there exist a small enough constant $\delta_{0}(N,q)>0$ such that for any $\varepsilon\leq 1$ and $\|\nabla u\|_{2}=\delta_{0}$ we have
\begin{equation}
    I_{\varepsilon}(u)\geq \frac{1}{4}\delta_{0},
\end{equation}
which together with the mountain pass character of $u_{\varepsilon}$, we have $S_{\varepsilon}\geq \frac{1}{4}\delta_{0}$ for any $\varepsilon\leq 1$. Hence, we complete the proof.
\end{proof}

In addition, we can obtain the asymptotic behavior of the least energy $S_{\varepsilon}$ as $\varepsilon\to 0$.

\begin{Lem}\label{limit of S-varepsilon}
$S_{\varepsilon}$ is strictly decreasing for $\varepsilon$. Furthermore, $S_{\varepsilon}\to \frac{1}{N}S^{\frac{N}{2}}$, as $\varepsilon\to 0$.
\end{Lem}

\begin{proof}
 Assume that $0<\varepsilon_{1}<\varepsilon_{2}$ and $u_{\varepsilon_{1}}, u_{\varepsilon_{2}}$ are the corresponding least energy solution. It is easy to verify that for any $u\in H^{1}_{0}(\Omega)\setminus\{0\}$ and $\varepsilon>0$, the fiber map $I_{\varepsilon}(tu)$ exist a unique maximum point $t_{\varepsilon,u}$ and $t_{\varepsilon,u}u\in N_{\varepsilon}$, where $N_{\varepsilon}$ is the Nehari manifold defined in (\ref{Nehari identity for p-varepsion}). Hence, we can obtain that $t_{\varepsilon, u}$ is strictly decreasing on $\varepsilon$ and $t_{\varepsilon_{2},u_{\varepsilon_{1}}}<t_{\varepsilon_{1},u_{\varepsilon_{1}}}=1$. Now, we have
 \begin{equation}\label{limit of S-varepsilon proof 1}
     \begin{aligned}
         S_{\varepsilon_{2}}\leq \max_{t>0}I_{\varepsilon_{2}}(tu_{\varepsilon_{1}})&= I_{\varepsilon_{2}}(t_{{\varepsilon}_{2},u_{\varepsilon_{1}}}u_{{\varepsilon_{1}}})\\
         &=\left(\frac{1}{2}-\frac{1}{q}\right)t_{\varepsilon_{2,u_{{\varepsilon_{1}}}}}^{2}\int_{\Omega}|\nabla u_{{\varepsilon}_{1}}|^2dx+\left(\frac{1}{q}-\frac{1}{2^*}\right)t_{\varepsilon_{2,u_{{\varepsilon_{1}}}}}^{2^*}\int_{\Omega}|u_{{\varepsilon}_{1}}|^{2^*}dx\\
         &<\left(\frac{1}{2}-\frac{1}{q}\right)\int_{\Omega}|\nabla u_{{\varepsilon}_{1}}|^2dx+\left(\frac{1}{q}-\frac{1}{2^*}\right)\int_{\Omega}|u_{{\varepsilon}_{1}}|^{2^*}dx\\
         &=I_{\varepsilon_{1}}(u_{\varepsilon_{1}})=S_{\varepsilon_{1}}.    
     \end{aligned}
 \end{equation}
 Next, we show $S_{\varepsilon}\to S$, as $\varepsilon\to 0$. By Lemma \ref{limit of u-varepsilon} and (\ref{definition of S}), we have
 \begin{equation}\label{limit of S-varepsilon proof 2}
     \|\nabla u_{\varepsilon}\|_{2}^{2}=\|u_{\varepsilon}\|_{2^*}^{2^*}+\varepsilon\|u_{\varepsilon}\|_{q}^{q}\leq S^{-\frac{2^*}{2}}\|\nabla u_{\varepsilon}\|^{2^*}_{2}+o(1).
 \end{equation}
Let $\|\nabla u_{\varepsilon}\|_{2}^{2}\to l>0$, as $\varepsilon\to 0$. Then by (\ref{limit of S-varepsilon proof 2}), we have $l\geq S^{\frac{N}{2}}$. Moreover, we have $S_{\varepsilon}=I_{\varepsilon}(u_{\varepsilon})=(\frac{1}{2}-\frac{1}{2^*})\|\nabla u_{\varepsilon}\|_{2}^{2}+o(1)$, thus 
\begin{equation}
  \frac{1}{N}S^{\frac{N}{2}}+o(1)\leq S_{\varepsilon}<\frac{1}{N}S^{\frac{N}{2}}. 
\end{equation}
Let $\varepsilon\to0$ and we complete the proof.
\end{proof}
The following Lemma shows that $\{u_{\varepsilon}\}_{\varepsilon>0}$ is a blow-up sequence as $\varepsilon\to 0$.

\begin{Lem}
 $\|u_{\varepsilon}\|_{L^{\infty}(\Omega)}\to \infty$ as $\varepsilon\to 0$. 
\end{Lem}
\begin{proof}
If $\|u_{\varepsilon}\|_{L^{\infty}(\Omega)}\leq C$, then by Lemma \ref{limit of u-varepsilon} and Lebesgue dominated convergence theorem, we have $S_{\varepsilon}\to 0$ as $\varepsilon\to 0$, which make a contraction with Lemma \ref{limit of S-varepsilon}.    
\end{proof}

Let $x_{\varepsilon}\in\Omega$, $\mu_{\varepsilon}\in\R^{+}$ such that 
\begin{equation}
    u_{\varepsilon}(x_{\varepsilon})=\max_{x\in\Omega}u_{\varepsilon}(x)=\mu_{\varepsilon}^{\frac{N-2}{2}}.
\end{equation}
We assume that $x_{\varepsilon}\to x_{0}\in \bar{\Omega}$ and we define a family of rescaled functions
\begin{equation}
    v_{\varepsilon}(x)=\mu_{\varepsilon}^{-\frac{N-2}{2}}u_{\varepsilon}(\mu_{\varepsilon}^{-1}x+x_{\varepsilon}),
\end{equation}
then $v_{\varepsilon}$ satisfy that
\begin{equation}\label{p-varepsion-*}\tag{$P_{\varepsilon}^*$}
\begin{cases}
-\Delta v_{\varepsilon}= v_{\varepsilon}^{2^*-1}+\varepsilon\mu_{\varepsilon}^{-\frac{N-2}{2}(2^*-q)} v_{\varepsilon}^{q-1},\quad v_{\varepsilon}>0, &{\text{in}~\Omega_{\varepsilon}},\\
\quad \  \ v_{\varepsilon}=0, &{\text{on}~\partial \Omega_{\varepsilon}},
\end{cases}
\end{equation}
where $\Omega_{\varepsilon}:=\{x\in\R^{N}:\mu_{\varepsilon}^{-1}x+x_{\varepsilon}\in\Omega\}$.

It is standard to verify the following Lemma.
\begin{Lem}\label{properties of v-varepsilon}
The following properties hold true
\begin{enumerate}%[label=$(\arabic*)$]
    \item $0\leq v_{\varepsilon}(x)\leq 1$, for $x\in \Omega_{\varepsilon}$ and  $v_{\varepsilon}(0)=\max\limits_{x\in\Omega_{\varepsilon}}v_{\varepsilon}(x)=1$.
    \item $\|\nabla v_{\varepsilon}\|_{L^{2}(\Omega_{\varepsilon})}^{2}=\|\nabla u_{\varepsilon}\|_{L^{2}(\Omega)}^{2}$, $\|v_{\varepsilon}\|_{L^{p}(\Omega_{\varepsilon})}^{p}=\mu_{\varepsilon}^{\frac{N-2}{2}(2^*-p)}\|u_{\varepsilon}\|_{L^{p}(\Omega)}^{p}$, \textrm{for any} $p\in[2,2^*].$
\end{enumerate}
\end{Lem}
Hence, $v_{\varepsilon}$ is also a least energy solution of (\ref{p-varepsion-*}) with the same energy $S_{\varepsilon}$. The following proposition shows that the rescaled least energy solutions $v_{\varepsilon}$ converge to the standard bubble function $U$ as $\varepsilon\to 0$.

\begin{Prop}\label{limit of v-varepsilon}
Let $v_{\varepsilon}(x):=0$, for any $x\in \R^{N}\setminus \Omega_{\varepsilon}$. Then, up to a sub-sequence, one has $\mu_{\varepsilon}d(x_{\varepsilon},\partial\Omega)\to\infty$ as $\varepsilon\to 0$ and there exists a function $U\in \mathcal{D}^{1,2}(\R^{N})$ such that
\begin{equation}
v_{\varepsilon}\to U \text{~in~} \mathcal{D}^{1,2}(\R^{N})\text{~and~}v_{\varepsilon}\to U  \text{~in~}C^{2}_{loc}(\R^{N}).
\end{equation}
Moreover, $U$ satisfies that
\begin{equation}\label{limit of v-varepsilon 2}
    \begin{cases}
        -\Delta U=U^{2^*-1},\quad 0\leq U\leq 1, \quad x\in \R^{N},\\
       \ U(0)=1,
    \end{cases}
\end{equation}
Hence by $(\ref{Aubin-Talenti bubble})$, one has
\begin{equation}
U(x)=\left(\frac{N(N-2)}{N(N-2)+|x|^{2}}\right)^{\frac{N-2}{2}}. 
\end{equation}
\end{Prop}
\begin{proof}
By Lemma \ref{properties of v-varepsilon}, we know that $\{v_{\varepsilon}\}$ is bounded in $\mathcal{D}^{1,2}(\R^{N})$, thus there exist a function $U\in \mathcal{D}^{1,2}(\R^{N})$ such that $v_{\varepsilon}\rightharpoonup U$ weakly in $\mathcal{D}^{1,2}(\R^{N})$ as $\varepsilon\to 0$. Next by standard elliptic regular theory, we have $v_{\varepsilon} \to U$ in $C^{2}_{loc}(\R^{N})$ and hence $U\not\equiv 0$. Moreover, if $\mu_{\varepsilon}d(x_{\varepsilon},\partial\Omega)\to a<\infty$, then after translation and rotation $U$ satisfies
\begin{equation}\label{limit of v-varepsilon proof 1}
    \begin{cases}
        -\Delta U=U^{2^*-1},\quad U>0, \quad x\in \R^{N}_{+},\\
        \ U(0)=\max\limits_{y\in\R^{N}_{+}}U(x)=1,
    \end{cases}
\end{equation}
which make a contraction with  \cite[Theorem 2]{Dancer1992SomeNO}, hence $\mu_{\varepsilon}d(x_{\varepsilon},\partial\Omega)\to\infty$ and $U$ satisfies equation (\ref{limit of v-varepsilon 2}). Finally, by Lions concentration-compactness lemma \cite{Lions1984TheCP1,Lions1984TheCP2} and Lemma \ref{limit of S-varepsilon} it's standard to show that $v_{\varepsilon}\to U$ in $\mathcal{D}^{1,2}(\R^{N})$ as $\varepsilon\to 0$.
\end{proof}

\begin{Cor}
Let $u_{\varepsilon}$ be the least energy solution to the problem (\ref{p-varepsion}), then 
\begin{equation}
\frac{\int_{\Omega}|\nabla u_{\varepsilon}|^{2}dx}{(\int_{\Omega}|u_{\varepsilon}|^{2^*}dx)^{\frac{2}{2^*}}}\to S \text{~as~} \varepsilon\to 0. 
\end{equation}
\end{Cor}

By Kelvin transformation and Moser iteration, we can obtain a upper bound for $v_{\varepsilon}$.
\begin{Prop}\label{uniform estimate of v-varepsilon}
There exist a constant $C>0$ such that for any 
$x\in\Omega_{\varepsilon}$
\begin{equation}
    v_{\varepsilon}(x)\leq CU(x).
\end{equation}
\end{Prop}
\begin{proof}
Let $w_{\varepsilon}$ be the Kelvin transform of $v_{\varepsilon}$
\begin{equation}
    w_{\varepsilon}(x):=|x|^{2-N}v_{\varepsilon}\left(\frac{x}{|x|^{2}}\right),
\end{equation}
then
\begin{equation}
\begin{cases}
-\Delta w_{\varepsilon}=w_{\varepsilon}^{2^*-1}+\varepsilon\mu_{\varepsilon}^{-\frac{N-2}{2}(2^*-q)}|x|^{-(N-2)(2^*-q)}w_{\varepsilon}^{q-1}, \text{~in~} \Omega_{\varepsilon}^*,\\
\quad \ \ w_{\varepsilon}\geq 0,
\end{cases}    
\end{equation}
where $\Omega_{\varepsilon}^*$ is the image of $\Omega_{\varepsilon}$ under the Kelvin transform, which is the whole $\R^{N}$ except a small region near the origin and there exist a constant $a>0$ such that $\Omega_{\varepsilon}^*\subset B^{c}(0,\frac{1}{a\mu_{\varepsilon}})$. Then it just need to show that there exist a constant $C>0$ such that
\begin{equation}\label{uniform estimate of v-varepsilon 8}
    w_{\varepsilon}\leq C \text{~for any~} x\in \Omega_{\varepsilon}^*.
\end{equation}
From the definition of $w_{\varepsilon}$ and  the fact that $0\leq v_{\varepsilon}\leq 1$, we have 
\begin{equation}
    w_{\varepsilon}(x)\leq |x|^{2-N}, \text{~for any~} x\in \Omega_{\varepsilon}^*,
\end{equation}
so we only need to bound $w_{\varepsilon}$ near the origin. We shall show this by Moser iteration (see Lemma \ref{moser iteration}) with $a(x)=0$ and $b(x)=w_{\varepsilon}^{2^*-2}(x)+\varepsilon\mu_{\varepsilon}^{-\frac{N-2}{2}(2^*-q)}|x|^{-(N-2)(2^*-q)}w_{\varepsilon}^{q-2}(x)$. 

Let $v_{\varepsilon}(x):=0$, when $x\in \R^{N}\setminus \Omega_{\varepsilon}$ and $w_{\varepsilon}(x):=0$, when $x\in \R^{N}\setminus \Omega_{\varepsilon}^{*}$. Note that Kelvin transform is linear and preserve the $\mathcal{D}^{1,2}(\R^{N})$ norm, then by $v_{\varepsilon}\to U$ in $L^{2^*}(\R^N)$, we have 
\begin{equation}\label{decay estimate-proof-5}
  w_{\varepsilon}(x)\to w(x):=\left(\frac{N(N-2)}{N(N-2)|x|^{2}+1}\right)^{\frac{N-2}{2}} \text{~in~} L^{2^*}(\R^N) \quad \text{~as~}\varepsilon\to 0.
\end{equation}
Let $\alpha_{\varepsilon}:=\varepsilon\mu_{\varepsilon}^{-\frac{N-2}{2}(2^*-q)}$ and $\gamma:=(N-2)(2^*-q)$, then by H\"older inequality we have  
\begin{equation}\label{decay estimate-proof-6}
\begin{aligned}
\int_{B(0,4)}\left|\frac{\alpha_{\varepsilon}}{|x|^{\gamma}}w_{\varepsilon}(x)^{q-2}\right|^{\frac{N}{2}}dx &=\int_{\frac{1}{a\mu_{\varepsilon}}\leq |x|\leq 4}\left|\frac{\alpha_{\varepsilon}}{|x|^{\gamma}}w_{\varepsilon}(x)^{q-2}\right|^{\frac{N}{2}}dx\\
&\leq \alpha_{\varepsilon}^{\frac{N}{2}}(\int_{\frac{1}{a\mu_{\varepsilon}}\leq |x|\leq 4}|x|^{-2N}dx)^{\frac{4-(q-2)(N-2)}{4}}(\int_{\frac{1}{a\mu_{\varepsilon}}\leq |x|\leq 4}w_{\varepsilon}^{2^*}dx)^{\frac{(q-2)(N-2)}{4}}\\
&\leq C\alpha_{\varepsilon}^{\frac{N}{2}}\mu_{\varepsilon}^{\frac{4N-N(q-2)(N-2)}{4}}=C\varepsilon^{\frac{N}{2}}\to 0\quad \text{~as~} \varepsilon\to 0,
\end{aligned}
\end{equation}
which together with (\ref{decay estimate-proof-5}) and Lemma \ref{moser iteration}, we can obtain that for any $p>1$ there exist $\varepsilon_{0}>0$ and a constant $C_{p}$ such that for any $\varepsilon<\varepsilon_{0}(p)$
\begin{equation}\label{decay estimate-proof-7}
    \|w_{\varepsilon}^{p}\|_{H^{{1}}(B_{1})}\leq C_{p}.
\end{equation}
Next, for $p>2^*$, we let $s=\frac{2Np}{p(N-2)(2^*-q)+2N(q-2)}$, then $s>\frac{N}{2}$, $(q-2)s<p$ and
\begin{equation}
\begin{aligned}
\int_{B(0,4)}\left|\frac{\alpha_{\varepsilon}}{|x|^{\gamma}}w_{\varepsilon}(x)^{q-2}\right|^{s}dx &=\int_{\frac{1}{a\mu_{\varepsilon}}\leq |x|\leq 4}\left|\frac{\alpha_{\varepsilon}}{|x|^{\gamma}}w_{\varepsilon}(x)^{q-2}\right|^{s}dx\\
&\leq \alpha_{\varepsilon}^{s}(\int_{\frac{1}{a\mu_{\varepsilon}}\leq |x|\leq 4}|x|^{-\gamma s\frac{p}{p-(q-2)s}}dx)^{\frac{p-(q-2)s}{p}}(\int_{\frac{1}{a\mu_{\varepsilon}}\leq |x|\leq 4}w_{\varepsilon}^{p}dx)^{\frac{(q-2)s}{p}}\\
&\leq C\varepsilon^{s}\to 0\quad \text{~as~} \varepsilon\to 0,
\end{aligned}
\end{equation}
which together with (\ref{decay estimate-proof-7}), we can obtain that there exist $s>\frac{N}{2}$ and $\varepsilon_{1}>0$, such that $b\in L^{s}(B_{4})$ for any $\varepsilon<\varepsilon_{1}$.
Then by Lemma \ref{moser iteration} we have
\begin{equation}
\|w_{\varepsilon}\|_{L^{\infty}(B_{1})}\leq C(N,s,\|b\|_{L^{s}(B_{4})})\|w_{\varepsilon}\|_{L^{2^{*}}(B_{4})},   
\end{equation}
which together with (\ref{decay estimate-proof-5}), we obtain that $(\ref{uniform estimate of v-varepsilon 8})$ hold.
\end{proof}
\begin{Cor}\label{upper estimate of u-varepsilon}
    $u_{\varepsilon}(x)\leq C\left(\frac{N(N-2)\mu_{\varepsilon}}{N(N-2)+\mu_{\varepsilon}^{2}|x-x_{\varepsilon}|^{2}}\right)^{\frac{N-2}{2}}\leq CU_{\mu_{\varepsilon},x_{\varepsilon}}(x)$.
\end{Cor}

\subsection{Blow-up rate and location of blow-up point}
In this subsection, we will consider the location of Blow-up point and give the exact blow-up rate. First, we can show that the blow-up point concentrates in the interior of domain $\Omega$. 
\begin{Prop}
$x_{\varepsilon}\to x_{0}\in\Omega$ as $\varepsilon\to 0$. 
\end{Prop}
\begin{proof}
To the contrary, we assume $d_{\varepsilon}:=\frac{1}{4}\text{dist}(x_{\varepsilon},\partial\Omega)\to 0$ as $\varepsilon\to 0$, then we shall use the local Poho\v{z}aev identity (Lemma \ref{Local Pohozaev identiy}) on the sphere $\partial B(x_{\varepsilon},2d_{\varepsilon})$ to reach a contradiction. Recall that
\begin{equation}\label{boundary estimate proof 1}
    -\int_{\partial B(x_{\varepsilon},2d_{\varepsilon})}\frac{\partial u_{\varepsilon}}{\partial x_{i}}\frac{\partial u_{\varepsilon}}{\partial n}+\frac{1}{2}\int_{\partial B(x_{\varepsilon},2d_{\varepsilon})}|\nabla u_{\varepsilon}|^{2}n_{i}=\frac{1}{2^*}\int_{\partial B(x_{\varepsilon},2d_{\varepsilon})}u_{\varepsilon}^{2^*}n_{i}+\frac{\varepsilon}{q}\int_{\partial B(x_{\varepsilon},2d_{\varepsilon})}u_{\varepsilon}^{q}n_{i},
\end{equation}
 where $i=1,\cdots ,N$ and $n$ is the unit outward normal of $\partial B(x_{\varepsilon},2d_{\varepsilon})$. For the sake of brevity, we let $L^{\varepsilon}_{i}$ and $R^{\varepsilon}_{i}$ be the left and right hand sides of the local Poho\v{z}aev identity.
 Firstly, we claim that for any $x\in\partial B(x_{\varepsilon},2d_{\varepsilon})$, we have
 \begin{equation}\label{boundary estimate proof 2}
 u_{\varepsilon}(x)=\mu_{\varepsilon}^{-\frac{N-2}{2}}C_{N}G(x,x_{\varepsilon})+o(\mu_{\varepsilon}^{-\frac{N-2}{2}}d_{\varepsilon}^{-(N-2)}),    
 \end{equation}
and
\begin{equation}\label{boundary estimate proof 3}
    \nabla u_{\varepsilon}=\mu_{\varepsilon}^{-\frac{N-2}{2}}C_{N}\nabla_{x}G(x,x_{\varepsilon})+o(\mu_{\varepsilon}^{-\frac{N-2}{2}}d_{\varepsilon}^{-(N-1)}),   
\end{equation}
where
\begin{equation}\label{boundary estimate proof 4}
    C_{N}:=\int_{\R^{N}}\left(\frac{N(N-2)}{N(N-2)+|x|^{2}}\right)^{\frac{N+2}{2}}dx=\frac{\alpha_{N}^{2^*}\omega_{N}}{N},
\end{equation}
and $o$ nation is understood as
\begin{equation}\label{boundary estimate proof 5}
    \lim_{\varepsilon\to 0}\sup_{x\in\partial B(x_{\varepsilon},2d_{\varepsilon})}d_{\varepsilon}^{k}\mu_{\varepsilon}^{\frac{N-2}{2}}|o(d_{\varepsilon}^{-k}\mu_{\varepsilon}^{-\frac{N-2}{2}})|=0 \text{~for~} k=N-1 \text{~or~} N-2.
\end{equation}

By (\ref{boundary estimate proof 1}), (\ref{boundary estimate proof 3}) and Lemma \ref{estimate of Green function}, we obtain that
\begin{equation}\label{boundary estimate proof 6}
\begin{aligned}
L^{\varepsilon}_{i}&=-\mu_{\varepsilon}^{-(N-2)}C_{N}^{2}\left(\int_{\partial B(x_{\varepsilon},2d_{\varepsilon})}\frac{\partial G(x,x_{\varepsilon})}{\partial x_{i}}\frac{\partial G(x,x_{\varepsilon})}{\partial n}-\frac{1}{2}\int_{\partial B(x_{\varepsilon},2d_{\varepsilon})}|\nabla G(x,x_{\varepsilon})|^{2}n_{i}\right)\\
&\quad+o(\mu_{\varepsilon}^{-(N-2)}d_{\varepsilon}^{-(N-1)}).
\end{aligned}
\end{equation}
Next, for any $x\in\partial B(x_{\varepsilon},2d_{\varepsilon}) $, by Corollary \ref{upper estimate of u-varepsilon}, we have
\begin{equation}
    u_{\varepsilon}(x)\leq C\mu_{\varepsilon}^{-\frac{N-2}{2}}d_{\varepsilon}^{-(N-2)},
\end{equation}
then 
\begin{equation}
\left|\frac{1}{2^*}\int_{\partial B(x_{\varepsilon},2d_{\varepsilon})}u_{\varepsilon}^{2^*}n_{i}\right|\leq C \mu_{\varepsilon}^{-N}d_{\varepsilon}^{-(N+1)},    
\end{equation}
and
\begin{equation}
\left|\frac{\varepsilon}{q}\int_{\partial B(x_{\varepsilon},2d_{\varepsilon})}u_{\varepsilon}^{q}n_{i}\right|\leq C\varepsilon\mu_{\varepsilon}^{-\frac{N-2}{2}q}d_{\varepsilon}^{(N-1)-(N-2)q}.    
\end{equation}
Hence
\begin{equation}\label{boundary estimate proof 10}
 R^{\varepsilon}_{i}=O\left(\mu_{\varepsilon}^{-N}d_{\varepsilon}^{-(N+1)}+\varepsilon\mu_{\varepsilon}^{-\frac{N-2}{2}q}d_{\varepsilon}^{(N-1)-(N-2)q}\right).  
\end{equation}
By (\ref{boundary estimate proof 6}) and (\ref{boundary estimate proof 10}), we obtain that
\begin{equation}
\begin{aligned}
 &\int_{\partial B(x_{\varepsilon},2d_{\varepsilon})}\frac{\partial G(x,x_{\varepsilon})}{\partial x_{i}}\frac{\partial G(x,x_{\varepsilon})}{\partial n}-\frac{1}{2}\int_{\partial B(x_{\varepsilon},2d_{\varepsilon})}|\nabla G(x,x_{\varepsilon})|^{2}n_{i}\\
&=o(d_{\varepsilon}^{-(N-1)})+O\left(\mu_{\varepsilon}^{-2}d_{\varepsilon}^{-(N+1)}+\varepsilon\mu_{\varepsilon}^{(N-2)-\frac{N-2}{2}q}d_{\varepsilon}^{(N-1)-(N-2)q}\right).    
\end{aligned}
\end{equation}
This together with Lemma \ref{Pohozaev identity of Green functon } gives
\begin{equation}
    \frac{\partial H(x,x_{\varepsilon})}{\partial x_{i}}\bigg|_{x=x_{\varepsilon}}=o(d_{\varepsilon}^{-(N-1)})+O\left(\mu_{\varepsilon}^{-2}d_{\varepsilon}^{-(N+1)}+\varepsilon\mu_{\varepsilon}^{-\frac{N-2}{2}(q-2)}d_{\varepsilon}^{(N-1)-(N-2)q}\right). 
\end{equation}
But this is impossible in view of Lemma \ref{estimate of robin function up to doundary}. 

Finally, we need only to show the claim hold. We first derive (\ref{boundary estimate proof 2}). By Green representation formula \ref{Green's representation formula}, we have
\begin{equation}\label{boundary estimate proof 13}
    u_{\varepsilon}(x)=G(x,x_{\varepsilon})\left(\int_{\Omega}u_{\varepsilon}^{2^*-1}(y)+\varepsilon u_{\varepsilon}^{q-1}(y)dy\right)+\int_{\Omega}(G(x,y)-G(x,x_{\varepsilon}))(u_{\varepsilon}^{2^*-1}(y)+\varepsilon u_{\varepsilon}^{q-1}(y))dy.
\end{equation}
By Proposition \ref{limit of v-varepsilon}, we deduce that
\begin{equation}\label{boundary estimate proof 14}
  \lim_{\varepsilon\to 0}\mu_{\varepsilon}^{\frac{N-2}{2}}\int_{\Omega}u_{\varepsilon}^{2^*-1}(y)+\varepsilon u_{\varepsilon}^{q-1}(y)dy=C_{N},
\end{equation}
which together with Lemma \ref{estimate of Green function}, we obtain
\begin{equation}\label{boundary estimate proof 15}
   G(x,x_{\varepsilon})\left(\int_{\Omega}u_{\varepsilon}^{2^*-1}(y)+\varepsilon u_{\varepsilon}^{q-1}(y)dy\right)=\mu_{\varepsilon}^{-\frac{N-2}{2}}C_{N}G(x,x_{\varepsilon})+o(\mu_{\varepsilon}^{-\frac{N-2}{2}}d_{\varepsilon}^{-(N-2)}), 
\end{equation}
for any $x\in \partial B(x_{\varepsilon},x_{2d_{\varepsilon}})$. To estimate the second integral in the right hand side of (\ref{boundary estimate proof 13}), we split it into three parts as follows
\begin{equation}
\begin{aligned}
&\int_{\Omega}(G(x,y)-G(x,x_{\varepsilon}))(u_{\varepsilon}^{2^*-1}(y)+\varepsilon u_{\varepsilon}^{q-1}(y))dy \\
&=\left(\int_{B(x_{\varepsilon},d_{\varepsilon})}+\int_
{B(x_{\varepsilon},4d_{\varepsilon})\setminus B(x_{\varepsilon},d_{\varepsilon})}+\int_{\Omega\setminus B(x_{\varepsilon},4d_{\varepsilon})}\right)(G(x,y)-G(x,x_{\varepsilon}))(u_{\varepsilon}^{2^*-1}(y)+\varepsilon u_{\varepsilon}^{q-1}(y))dy \\
&:=I_{1}(x)+I_{2}(x)+I_{3}(x).
\end{aligned}
\end{equation}
\textbf{Estimate of $I_{1}$}. For any $x\in\partial B(x_{\varepsilon},2d_{\varepsilon})$ and $y\in B(x_{\varepsilon},d_{\varepsilon})$, we have $|x-y|\geq d_{\varepsilon}$. Thus $|\nabla_{y}G(x,y)|\leq Cd_{\varepsilon}^{-(N-1)}$ and $|G(x,y)-G(x,x_{\varepsilon})|\leq Cd_{\varepsilon}^{-(N-1)}|y-x_{\varepsilon}|$, which together with Corollary \ref{upper estimate of u-varepsilon} shows that
\begin{equation}
\begin{aligned}
&\int_{B(x_{\varepsilon},d_{\varepsilon})}(G(x,y)-G(x,x_{\varepsilon}))u_{\varepsilon}^{q-1}dy\\
&\leq C d_{\varepsilon}^{-(N-1)} \int_{B(x_{\varepsilon},d_{\varepsilon})}|y-x_{\varepsilon}|u_{\varepsilon}^{q-1} dy\\
&\leq C d_{\varepsilon}^{-(N-1)}\mu_{\varepsilon}^{\frac{N-2}{2}(q-1)-(N+1)} \int_{B(0,\mu_{\varepsilon}d_{\varepsilon})}|y|\left(\frac{N(N+2)}{N(N+2)+|y|^{2}}\right)^{\frac{(N-2)(q-1)}{2}} dy.
\end{aligned}
\end{equation}
Note that
\begin{equation}
\int_{B(0,\mu_{\varepsilon}d_{\varepsilon})}|y|\left(\frac{N(N+2)}{N(N+2)+|y|^{2}}\right)^{\frac{(N-2)(q-1)}{2}} dy\leq \begin{cases}
        C, &\text{~if~} q>\frac{2N-1}{N-2},\\
        C\ln(\mu_{\varepsilon}d_{\varepsilon}), &\text{~if~} q=\frac{2N-1}{N-2},\\
        C(\mu_{\varepsilon}d_{\varepsilon})^{(N+1)-(N-2)(q-1)}, &\text{~if~} q<\frac{2N-1}{N-2},\\
    \end{cases}
\end{equation}
Hence
\begin{equation}
\begin{aligned}
\int_{B(x_{\varepsilon},d_{\varepsilon})}(G(x,y)-G(x,x_{\varepsilon}))u_{\varepsilon}^{q-1}dy=\begin{cases}
    O(d_{\varepsilon}^{-(N-1)}\mu_{\varepsilon}^{\frac{N-2}{2}(q-1)-(N+1)}), &\text{~if~} q>\frac{2N-1}{N-2},\\
     O(d_{\varepsilon}^{-(N-2)}\mu_{\varepsilon}^{\frac{N-2}{2}(q-1)-N}), &\text{~if~} q=\frac{2N-1}{N-2},\\
     O(d_{\varepsilon}^{2-(N-2)(q-1)}\mu_{\varepsilon}^{-\frac{N-2}{2}(q-1)}), &\text{~if~} q<\frac{2N-1}{N-2},\\
\end{cases}
\end{aligned}
\end{equation}
and
\begin{equation}\label{boundary estimate proof 20}
    |I_{1}(x)|=o(\mu_{\varepsilon}^{-\frac{N-2}{2}}d_{\varepsilon}^{-(N-2)}).
\end{equation}
\textbf{Estimate of $I_{2}$}. For any $y\in B(x_{\varepsilon},4d_{\varepsilon})\setminus B(x_{\varepsilon},d_{\varepsilon})$ and $x\in\partial B(x_{\varepsilon},2d_{\varepsilon})$, we have that
\begin{equation}
    |G(x,y)-G(x,x_{\varepsilon})|\leq |G(x,y)|+|G(x,x_{\varepsilon})|\leq \frac{C}{|x-y|^{N-2}}+\frac{C}{d_{\varepsilon}^{N-2}}\leq  \frac{C}{|x-y|^{N-2}}.
\end{equation}
Hence, we obtain 
\begin{equation}
\begin{aligned}
 \int_{B(x_{\varepsilon},4d_{\varepsilon})\setminus B(x_{\varepsilon},d_{\varepsilon})}u_{\varepsilon}^{q-1}dy&\leq C\int_{B(x_{\varepsilon},4d_{\varepsilon})\setminus B(x_{\varepsilon},d_{\varepsilon})}\frac{1}{|x-y|^{N-2}}u_{\varepsilon}^{q-1}dy\\
 &\leq C\int_{B(x_{\varepsilon},4d_{\varepsilon})\setminus B(x_{\varepsilon},d_{\varepsilon})}\frac{1}{|x-y|^{N-2}}\left(\frac{\mu_{\varepsilon}}{1+\mu_{\varepsilon}^{2}|y-x_{\varepsilon}|^{2}}\right)^{\frac{N-2}{2}(q-1)}dy\\
 &\leq C\mu_{\varepsilon}^{-\frac{N-2}{2}(q-1)}d_{\varepsilon}^{2-(N-2)(q-1)}.
\end{aligned}
\end{equation}
Thus
\begin{equation}\label{boundary estimate proof 23}
    |I_{2}(x)|=o(\mu_{\varepsilon}^{-\frac{N-2}{2}}d_{\varepsilon}^{-(N-2)}).
\end{equation}
\textbf{Estimate of $I_{3}$}. For any $x\in\partial B(x_{\varepsilon},2d_{\varepsilon})$ and $y\in \Omega\setminus B(x_{\varepsilon},4d_{\varepsilon})$, by Lemma \ref{estimate of Green function}, we have 
\begin{equation}
    |G(x,y)-G(x,x_{\varepsilon})|\leq Cd_{\varepsilon}^{-(N-2)}.
\end{equation}
Therefore 
\begin{equation}
    \begin{aligned}
        \int_{\Omega\setminus B(x_{\varepsilon},4d_{\varepsilon})}u_{\varepsilon}^{q-1}dy&\leq Cd_{\varepsilon}^{-(N-2)}\int_{\Omega\setminus B(x_{\varepsilon},4d_{\varepsilon})}u_{\varepsilon}^{q-1}dy\\
        &\leq Cd_{\varepsilon}^{-(N-2)}\mu_{\varepsilon}^{\frac{N-2}{2}(q-1)-N}\int_{B(0,\mu_{\varepsilon}R)\setminus B(0,\mu_{\varepsilon}d_{\varepsilon})}\left(\frac{1}{1+|y|^{2}}\right)^{\frac{N-2}{2}(q-1)}dy,
    \end{aligned}
\end{equation}
for some $R>0$. Note that
\begin{equation}
\int_{B(0,\mu_{\varepsilon}R)\setminus B(0,\mu_{\varepsilon}d_{\varepsilon})}\left(\frac{1}{1+|y|^{2}}\right)^{\frac{N-2}{2}(q-1)}dy=\begin{cases}
    O((\mu_{\varepsilon}d_{\varepsilon})^{N-(N-2)(q-1)}),&\text{~if~} q>\frac{2N-2}{N-2},\\
    O(\ln(\frac{1}{d_{\varepsilon}})), &\text{~if~} q=\frac{2N-2}{N-2},\\
    O(\mu_{\varepsilon}^{N-(N-2)(q-1)}), &\text{~if~} q<\frac{2N-2}{N-2}.
\end{cases}    
\end{equation}
we have

\begin{equation}
  \int_{\Omega\setminus B(x_{\varepsilon},4d_{\varepsilon})}u_{\varepsilon}^{q-1}dy=\begin{cases}
  O(d_{\varepsilon}^{2-(N-2)(q-1)}\mu_{\varepsilon}^{-\frac{N-2}{2}(q-1)}), &\text{~if~} q>\frac{2N-2}{N-2},\\
    O(d_{\varepsilon}^{-(N-1)}\mu_{\varepsilon}^{\frac{N-2}{2}(q-1)-N}), &\text{~if~} q=\frac{2N-2}{N-2},\\
    O(d_{\varepsilon}^{-(N-2)}\mu_{\varepsilon}^{-\frac{N-2}{2}(q-1)}), &\text{~if~} q<\frac{2N-2}{N-2},
    \end{cases}
\end{equation}    
Thus
\begin{equation}\label{boundary estimate proof 28}
    |I_{3}(x)|=o(\mu_{\varepsilon}^{-\frac{N-2}{2}}d_{\varepsilon}^{-(N-2)}),
\end{equation}
which together with (\ref{boundary estimate proof 20}) and (\ref{boundary estimate proof 23}), we obtain (\ref{boundary estimate proof 2}). Similarly, we can obtain (\ref{boundary estimate proof 3}), by using the Green representation formula and the following estimate
\begin{equation}
|\nabla_{x}G(x,y)-\nabla_{x}G(x,x_{\varepsilon})|\leq
    \begin{cases}
        Cd_{\varepsilon}^{-n}|y-x_{\varepsilon}|, &\text{~if~} y\in B(x_{\varepsilon},d_{\varepsilon}),\\
        C|x-y|^{-(N-1)}, &\text{~if~} y\in B(x_{\varepsilon},4d_{\varepsilon}),\\
        Cd_{\varepsilon}^{-(N-1)}, &\text{~if~} y\in \Omega\setminus B(x_{\varepsilon},4d_{\varepsilon}),\\
    \end{cases}
\end{equation}
which is hold for any $x\in\partial B(x_{\varepsilon},2d_{\varepsilon})$. Hence, we complete the proof.
\end{proof}

\begin{Lem}\label{upper blow-up estimate }
    $\varepsilon \leq C \mu_{\varepsilon}^{-\frac{N-2}{2}(q+2-2^*)}$.
\end{Lem}
\begin{proof}
We will show this Lemma by using the Poho\v{z}aev identity (\ref{Pohazaev-p-varepsion}). First of all, we shall find an upper bound of right-hand of (\ref{Pohazaev-p-varepsion}) by using Lemma \ref{gradient estimate}. Assume $\omega^{'}\subset\subset\omega$ be two  neighborhoods of $\partial\Omega$ and don't contain $x_{0}$. For any $x\in \omega$, by Corollary \ref{upper estimate of u-varepsilon} we have
\begin{equation}
\begin{aligned}
\|u_{\varepsilon}^{2^*-1}(x)+\varepsilon u_{\varepsilon}^{q-1}(x)\|_{L^{\infty}(\omega)}&\leq C(U_{\mu_{\varepsilon},x_{\varepsilon}}^{2^*-1}(x)+\varepsilon U_{\mu_{\varepsilon},x_{\varepsilon}}^{q-1}(x))\\
&\leq C(\mu_{\varepsilon}^{-\frac{N+2}{2}}+\varepsilon\mu_{\varepsilon}^{-\frac{(N-2)(q-1)}{2}})\\
&\leq C \mu_{\varepsilon}^{-\frac{N-2}{2}},
\end{aligned}
\end{equation}
where $C$ is constant depending only on $\omega,N,q$. On the other hand, by Corollary \ref{upper estimate of u-varepsilon} we have
\begin{equation}
 \int_{\Omega}u_{\varepsilon}^{2^*-1}(x)dx\leq C   \mu_{\varepsilon}^{-\frac{N-2}{2}}\int_{\Omega_{\varepsilon}}\left(\frac{1}{1+|y|^2}\right)^{\frac{N+2}{2}}dy\leq C\mu_{\varepsilon}^{-\frac{N-2}{2}},
\end{equation}
and
\begin{equation}\label{blow-up rate-proof-3}
\begin{aligned}
\int_{\Omega} u_{\varepsilon}^{q-1}(x)dx&\leq C\mu_{\varepsilon}^{-\frac{N-2}{2}(2^*+1-q)}\int_{\Omega_{\varepsilon}}\left(\frac{1}{1+|y|^2}\right)^{\frac{(N-2)(q-1)}{2}}dy\\
&\leq C\mu_{\varepsilon}^{-\frac{N-2}{2}(2^*+1-q)} \int_{B(0,c\mu_{\varepsilon})}\left(\frac{1}{1+|y|^2}\right)^{\frac{(N-2)(q-1)}{2}}dy\\
&\leq \begin{cases}
    C\mu_{\varepsilon}^{-\frac{N-2}{2}(2^*+1-q)}, &\text{~if~}(N-2)(q-1)>N,\\
    C\mu_{\varepsilon}^{-\frac{N-2}{2}(2^*+1-q)}\ln\mu_{\varepsilon}, &\text{~if~}(N-2)(q-1)=N,\\
    C\mu_{\varepsilon}^{-\frac{N-2}{2}(q-1)}, &\text{~if~}(N-2)(q-1)<N,\\
\end{cases}\\
&\leq C\mu_{\varepsilon}^{-\frac{N-2}{2}},
\end{aligned}
\end{equation}
where constant $c>0$ and $C>0$ not depending on $\varepsilon$. Hence
\begin{equation}
\int_{\Omega}u_{\varepsilon}^{2^*-1}(x)+\varepsilon u_{\varepsilon}^{q-1}(x)dx\leq C  \mu_{\varepsilon}^{-\frac{N-2}{2}}.  
\end{equation}
Then by Lemma \ref{gradient estimate}, we obtain that
\begin{equation}\label{blow-up rate-proof-5}
\begin{aligned}
 |\nabla u_{\varepsilon}|_{C^{0,\alpha}(\omega^{\prime})}&\leq C (\|u_{\varepsilon}^{2^*-1}(x)+\varepsilon u_{\varepsilon}^{q-1}(x)\|_{L^{\infty}(\omega)}+\|u_{\varepsilon}^{2^*-1}(x)+\varepsilon u_{\varepsilon}^{q-1}(x)\|_{L^{1}(\Omega)}) \\
 &\leq C\mu_{\varepsilon}^{-\frac{N-2}{2}},
\end{aligned}
\end{equation}
which together with Poho\v{z}aev identity (\ref{Pohazaev-p-varepsion}) show that
\begin{equation}\label{blow-up rate-proof-6}
    \varepsilon \int_{\Omega}u_{\varepsilon}^{q}dx\leq C\mu_{\varepsilon}^{2-N}.
\end{equation}
Next, we shall find a lower bound of right-hand side of (\ref{Pohazaev-p-varepsion}). By Proposition \ref{limit of v-varepsilon}, we have
\begin{equation}\label{blow-up rate-proof-7}
    \begin{aligned}
       \varepsilon \int_{\Omega}u_{\varepsilon}^{q}dx&=\varepsilon\mu_{\varepsilon}^{-\frac{N-2}{2}(2^*-q)}\int_{\Omega_{\varepsilon}}v_{\varepsilon}^{q}dx\\
       &\geq \varepsilon\mu_{\varepsilon}^{-\frac{N-2}{2}(2^*-q)}\int_{B(0,1)}v_{\varepsilon}^{q}dx\\
       &\geq C\varepsilon\mu_{\varepsilon}^{-\frac{N-2}{2}(2^*-q)},
    \end{aligned}
\end{equation}
This together with (\ref{blow-up rate-proof-6}) complete the proof.
\end{proof}

\begin{Prop}\label{u-varepsilon-times-u-infty}
As $\varepsilon\to 0$, there holds
\begin{equation}
\|u_{\varepsilon}\|_{L^{\infty}(\Omega)}u_{\varepsilon}(x)\to \alpha_{N}^{2^*}\frac{1}{N}w_{N}G(x,x_{0}) \text{~in~}C^{1,\alpha}(\omega),    
\end{equation}
where $\alpha_{N}=(N(N-2))^{\frac{N-2}{4}}$, $\omega_{N}$ denotes the measure of unit sphere in $\R^{N}$ and $\omega$ is a neighborhood of $\partial \Omega$ not containing $x_{0}$.
\end{Prop}
\begin{proof}
 Let $\tilde{u}_{\varepsilon}(x)=\|u_{\varepsilon}\|_{L^{\infty}(\Omega)}u_{\varepsilon}(x)=\mu_{\varepsilon}^{\frac{N-2}{2}}u_{\varepsilon}(x)$. Then $\tilde u_{\varepsilon}$ satisfy that
 \begin{equation}
     \begin{cases}
         -\Delta\tilde{u}_{\varepsilon}(x)=\mu_{\varepsilon}^{\frac{N-2}{2}}u_{\varepsilon}^{2^*-1}(x)+\varepsilon\mu_{\varepsilon}^{\frac{N-2}{2}}u_{\varepsilon}^{q-1}(x), &\text{~in~}\Omega,\\
         \quad \ \ \tilde{u}_{\varepsilon}(x)=0, &\text{~on~}\partial\Omega.
     \end{cases}
 \end{equation}
First of all, by Proposition \ref{limit of v-varepsilon} and (\ref{blow-up rate-proof-3}), we have
\begin{equation}
\begin{aligned}
    \int_{\Omega}\mu_{\varepsilon}^{\frac{N-2}{2}}u_{\varepsilon}^{2^*-1}(x)+\varepsilon\mu_{\varepsilon}^{\frac{N-2}{2}}u_{\varepsilon}^{q-1}(x)dx&=\int_{\Omega_{\varepsilon}}v_{\varepsilon}^{2^*-1}(x)dx+\int_{\Omega}\varepsilon\mu_{\varepsilon}^{\frac{N-2}{2}}u_{\varepsilon}^{q-1}(x)dx\\
    &\to \alpha_{N}^{2^*}\frac{1}{N}\omega_{N}, \quad\text{~as~} \varepsilon\to 0.
\end{aligned}
\end{equation}
On the other hand, for any $x\neq x_{0}$
\begin{equation}
\begin{aligned}
\mu_{\varepsilon}^{\frac{N-2}{2}}u_{\varepsilon}^{2^*-1}(x)+\varepsilon\mu_{\varepsilon}^{\frac{N-2}{2}}u_{\varepsilon}^{q-1}(x)&\leq C\left(\mu_{\varepsilon}^{\frac{N-2}{2}}U_{\mu_{\varepsilon},x_{\varepsilon}}^{2^*-1}(x)+\varepsilon\mu_{\varepsilon}^{\frac{N-2}{2}}U_{\mu_{\varepsilon},x_{\varepsilon}}^{q-1}(x)\right)\\
    &\leq C \left(\mu_{\varepsilon}^{-2}+\varepsilon\mu_{\varepsilon}^{-\frac{N-2}{2}(q-2)}\right)|x-x_{0}|^{-N-2}\\
    &\to 0 \text{~as~}\varepsilon \to 0.    
\end{aligned}
\end{equation}
Hence
\begin{equation}
    -\Delta\tilde{u}_{\varepsilon}(x)\stackrel{*}{\rightharpoonup}\alpha_{N}^{2^*}\frac{1}{N}\omega_{N}\delta_{x_{0}} \text{~in the sense of Radon measure~},
\end{equation}
where $\delta_{x_{0}}$ is the Dirac function supported by $x_{0}$. Moreover, similar to the proof (\ref{blow-up rate-proof-5}), by Lemma \ref{gradient estimate} and Arezel\`a-Ascoli's theorem, there exist a function $P(x)$ such that
\begin{equation}
    \tilde{u}_{\varepsilon}(x)\to P(x) \text{~~in~~} C^{1,\alpha}(\omega),
\end{equation}
 where $\alpha\in (0,1)$ and $\omega$ is a neighborhood of $\partial\Omega$ not containing $x_{0}$. By the Green representation formula (\ref{Green's representation formula}), we obtain that
 \begin{equation}
     \tilde{u}_{\varepsilon}(x)=\int_{\Omega}G(x,y)(-\Delta \tilde{u}_{\varepsilon})(y)dy,
 \end{equation}
Hence 
 \begin{equation}
P(x)=\alpha_{N}^{2^*}\frac{1}{   N}\omega_{N}G(x,x_{0}).     
 \end{equation}
\end{proof}

\begin{Prop}\label{blow-up rate}
There holds the following exact blow-up rate
  \begin{equation}
    \lim_{\varepsilon\to 0}\varepsilon\|u_{\varepsilon}\|_{L^{\infty}(\Omega)}^{q+2-2^*}=\alpha_{N,q}R(x_{0}),
\end{equation}
where 
\begin{equation}\label{definition of alpha-N-q}
\alpha_{N,q}=\frac{2q}{2^*-q}\frac{\alpha_{N}^{2^*}\omega_{N}}{N^{2}}\frac{\Gamma(\frac{N-2}{2}q)}{\Gamma(\frac{N}{2})\Gamma(\frac{N-2}{2}q-\frac{N}{2})}.
\end{equation}
\end{Prop}
%\frac{2q\alpha_{N}^{\frac{2N-8}{N-2}}\omega_{N}(N-2)^{3}}{2N-(N-2)q}
\begin{proof}
Rewrite the Poho\v{z}aev identity (\ref{Pohazaev-p-varepsion}) as
\begin{equation}\label{limit behaviour proof 1}
    \frac{1}{2N}\int_{\partial\Omega}(x-x_{0},n)\left(\frac{\partial}{\partial  n}(\|u_{\varepsilon}\|_{L^{\infty}(\Omega)}u_{\varepsilon})\right)^{2}dx=(\frac{1}{q}-\frac{1}{2^*})\varepsilon\|u_{\varepsilon}\|_{L^{\infty}(\Omega)}^{2}\int_{\Omega}u_{\varepsilon}^{q}dx.
\end{equation}
By Proposition \ref{u-varepsilon-times-u-infty}, the left-hand side of (\ref{limit behaviour proof 1}), we obtain that 
\begin{equation}
    \frac{1}{2N}\int_{\partial\Omega}(x-x_{0},n)\left(\frac{\partial}{\partial  n}(\|u_{\varepsilon}\|_{L^{\infty}(\Omega)}u_{\varepsilon})\right)^{2}dx\to \frac{\alpha_{N}^{22^*}w_{N}^{2}}{2N^{3}}\int_{\partial\Omega}(x-x_{0},n)\left(\frac{\partial}{\partial  n}G(x,x_{0})\right)^{2}dx.
\end{equation}
By Proposition \ref{limit of v-varepsilon}, we have the right-hand side of (\ref{limit behaviour proof 1})
\begin{equation}
 (\frac{1}{q}-\frac{1}{2^*})\varepsilon\|u_{\varepsilon}\|_{L^{\infty}(\Omega)}^{2}\int_{\Omega}u_{\varepsilon}^{q}dx=\lim_{\varepsilon\to 0}(\frac{1}{q}-\frac{1}{2^*}) \varepsilon \|U\|_{L^{q}(\R^{N})}^{q}\|u_{\varepsilon}\|^{q+2-2^*}_{L^{\infty}(\Omega)}.  
\end{equation}
Furthermore, by Lemma \ref{Green function propertity 1}, we have
\begin{equation}
    \lim_{\varepsilon\to 0}\varepsilon\|u_{\varepsilon}\|_{L^{\infty}(\Omega)}^{q+2-2^*}=\sigma_{N,q}\int_{\partial\Omega}(x-x_{0},n)\left(\frac{\partial}{\partial  n}G(x,x_{0})\right)^{2}dS_{x}=\sigma_{N,q}(N-2)R(x_{0}),
\end{equation}
where 
\begin{equation}
    \sigma_{N,q}=\frac{\alpha_{N}^{22^*}w_{N}^{2}}{2N^{3}}/(\frac{1}{q}-\frac{1}{2^*}) \|U\|_{L^{q}(\R^{N})}^{q}=\frac{\alpha_{N}^{2^*}\omega_{N}2q}{N^{2}(2N-(N-2)q)}\frac{\Gamma(\frac{N-2}{2}q)}{\Gamma(\frac{N}{2})\Gamma(\frac{N-2}{2}q-\frac{N}{2})}.
\end{equation}
which together with the definition of $\alpha_{N}$ concludes the proof.
\end{proof}

Finally, we can show that $x_{0}$ is a critical point of Robin function $R(x)$.
\begin{Prop}
   $\nabla R (x_{0})=0.$
\end{Prop}
\begin{proof}
In virtue of the Poho\v{z}aev identity, multiplying the (\ref{p-varepsion}) by $\frac{\partial u_{\varepsilon}}{\partial x_{j}}$ and integrating, we have
\begin{equation}\label{boundary estimate=0}
    \int_{\partial\Omega}\left(\frac{\partial u_{\varepsilon}}{\partial n}\right)^{2}n(x)dS_{x}=0,
\end{equation}
where $n=n(x)$ denotes the unit outward normal of $\partial\Omega$ at $x$. Then by Proposition \ref{u-varepsilon-times-u-infty} and Lemma \ref{Green function propertity 1}, we have
 \begin{equation}
     0=\int_{\partial\Omega}\left(\frac{\partial G(x,x_{0})}{\partial n}\right)^{2}n(x)dS_{x}=\nabla R(x_{0}).
 \end{equation}
\end{proof}

\subsection{Asymptotic behavior of \texorpdfstring{$S_\varepsilon$}{}}
In this subsection, we consider the asymptotic behavior on the least energy $S_{\varepsilon}$ and derive a precise energy expansion.

First of all, we give an orthogonal decomposition for least energy solution $u_{\varepsilon}$. Before that, we define a subspace of $H^{1}_{0}(\Omega)$. For any $a\in\Omega$ and $\lambda\in \R^{+}$, we define
\begin{equation}\label{definition of space E}
\begin{aligned}
E_{\lambda,a}:=\left\{w\in H^{1}_{0}(\Omega)\right.:0&=\int_{\Omega}\nabla w\cdot\nabla PU_{\lambda,a}dx=\int_{\Omega}\nabla w\cdot\nabla \left(\frac{\partial}{\partial \lambda}PU_{\lambda,a}\right)dx\\
&=\left.\int_{\Omega}\nabla w\cdot\nabla \left(\frac{\partial}{\partial a_{i}}PU_{\lambda,a}\right)dx, \text{~for~}i=1,\cdots,N\right\},    
\end{aligned}
\end{equation}
where $PU_{\lambda,a}$ is the  projection of $U_{\lambda,a}$, see Appendix \ref{Projection of bubbles} for more details.

\begin{Lem}\label{decomposition of u-varepsilon}
For $\varepsilon$ small enough, we have the following decomposition for $u_{\varepsilon}$
\begin{equation}
    u_{\varepsilon}=\alpha_{\varepsilon}PU_{\lambda_{\varepsilon},a_{\varepsilon}}+w_{\varepsilon},
\end{equation}
where $\alpha_{\varepsilon}\in\R^{+}$, $\lambda_{\varepsilon}\in\R^{+}$, $a_{\varepsilon}\in\Omega$ and $w_{\varepsilon}\in E_{\lambda_{\varepsilon},a_{\varepsilon}}$ satisfying 
\begin{equation}
    \alpha_{\varepsilon}\to\alpha_{N}, \lambda_{\varepsilon}\to \infty, a_{\varepsilon}\to x_{0}\in {\Omega}, \lambda_{\varepsilon}d_{\varepsilon}\to\infty,\|w_{\varepsilon}\|_{H^{1}_{0}(\Omega)}\to 0,\quad \text{~as~}\varepsilon\to 0,
\end{equation}
where $d_{\varepsilon}=\text{dist}(a_{\varepsilon},\partial\Omega)$ is the distance between $a_{\varepsilon}$ and the boundary $\partial\Omega$. Moreover, we have 
\begin{equation}\label{decomposition of u-varepsilon 3}
     \varepsilon=O\left(\frac{1}{\lambda_{\varepsilon}^{\frac{N-2}{2}(q+2-2^*)}}\right)\text{~~and~~}|\nabla u_{\varepsilon}|^{2}\stackrel{*}{\rightharpoonup}S^{\frac{N}{2}}\delta_{x_{0}}, \text{~in the sense of Radon measure~},
\end{equation}
where $\delta_{x_{0}}$ is the Dirac function at $x_{0}$ and $E_{\lambda_{\varepsilon},a_{\varepsilon}}$ defined in $(\ref{definition of space E})$.
\end{Lem}
\begin{proof}
From Proposition \ref{limit of v-varepsilon}, we know that if we let $u_{\varepsilon}(x):=0$, when $x\in \R^{N}\setminus \Omega$, then there exist $\tilde{\mu}_{\varepsilon}>0$, $x_{\varepsilon}\in\Omega$ and $\delta_{\varepsilon}\in \mathcal{D}^{1,2}(\R^{N})$ such that 
\begin{equation}\label{decomposition of u-varepsilon proof 1}
u_{\varepsilon}=\alpha_{N}U_{\tilde{\mu}_{\varepsilon},x_{\varepsilon}}+\delta_{\varepsilon}, \text{~in~}\mathcal{D}^{1,2}(\R^{N}),
\end{equation}
with
\begin{equation}\label{decomposition of u-varepsilon proof 2}
   \tilde{\mu}_{\varepsilon}=(N(N-2))^{-\frac{1}{2}}\mu_{\varepsilon}\to\infty,\quad x_{\varepsilon}\to x_{0}\in{\Omega},  \quad\|\delta_{\varepsilon}\|_{\mathcal{D}^{1,2}(\R^{N})}\to 0,
\end{equation}
and
\begin{equation}\label{decomposition of u-varepsilon proof 3}
   \tilde{\mu}_{\varepsilon}d(x_{\varepsilon}, \partial\Omega)\to \infty. 
\end{equation}
Then by Lemma \ref{estimate of U-lambda-a and psi-lambda-a 3}, we have
\begin{equation}\label{decomposition of u-varepsilon proof 4}
    \frac{u_{\varepsilon}}{a_{N}}\to PU_{\tilde{\mu}_{\varepsilon},x_{\varepsilon}}, \text{~in~} H^{1}_{0}(\Omega),
\end{equation}
which together with  \cite[Proposition 7]{Bahri1988CPAM} show that there exist $\alpha_{\varepsilon}\in \R^{+}$, $\lambda_{\varepsilon}\in \R^{+}$, $a_{\varepsilon}\in \Omega$ and $w_{\varepsilon}\in E_{\lambda_{\varepsilon},a_{\varepsilon}}$ such that
\begin{equation}\label{decomposition of u-varepsilon proof 5}
u_{\varepsilon}=\alpha_{\varepsilon}PU_{\lambda_{\varepsilon},a_{\varepsilon}}+w_{\varepsilon}, \text{~in~}{H}^{1}_{0}(\Omega).  
\end{equation}
with $\alpha_{\varepsilon}$ bing  bounded and $\|w_{\varepsilon}\|_{H^{1}_{0}(\Omega)}\to 0$ as $\varepsilon\to 0$. Thus by (\ref{decomposition of u-varepsilon proof 1}), (\ref{decomposition of u-varepsilon proof 5}) and Lemma \ref{estimate of U-lambda-a and psi-lambda-a 3}, we deduce that
\begin{equation}
    \alpha_{\varepsilon}\to \alpha_{N},
\end{equation}
and
\begin{equation}\label{decomposition of u-varepsilon proof 6}
    \alpha_{N}\alpha_{\varepsilon}\int_{\Omega}\nabla U_{\tilde{\mu}_{\varepsilon},x_{\varepsilon}}\nabla PU_{\lambda_{\varepsilon},a_{\varepsilon}}=S^{\frac{N}{2}}+o(1).
\end{equation}
Thus  by using Green formula, Lemma \ref{estimate of U-lambda-a and psi-lambda-a 3} and Lemma \ref{estimate of two different bubbles}, we have
\begin{equation}
\begin{aligned}
\int_{\Omega}\nabla U_{\tilde{\mu}_{\varepsilon},x_{\varepsilon}}\nabla PU_{\lambda_{\varepsilon},a_{\varepsilon}}&=N(N-2)\int_{\Omega}U_{\tilde{\mu}_{\varepsilon},x_{\varepsilon}}^{2^*-1}PU_{\lambda_{\varepsilon},a_{\varepsilon}}  \\
&\leq N(N-2)\int_{\Omega}U_{\tilde{\mu}_{\varepsilon},x_{\varepsilon}}^{2^*-1}U_{\lambda_{\varepsilon},a_{\varepsilon}}  \\
&=O\left(\frac{\lambda_{\varepsilon}}{\tilde{\mu}_{\varepsilon}}+\frac{\tilde{\mu}_{\varepsilon}}{\lambda_{\varepsilon}}+\tilde{\mu}_{\varepsilon}\lambda_{\varepsilon}|x_{\varepsilon
}-a_{\varepsilon}|\right)^{-\frac{N-2}{2}}.
\end{aligned}
\end{equation}
Therefore, by (\ref{decomposition of u-varepsilon proof 6}), (\ref{decomposition of u-varepsilon proof 3}) and Proposition \ref{blow-up rate}, we obtain that
\begin{equation}
  \varepsilon=O\left(\frac{1}{\lambda_{\varepsilon}^{\frac{N-2}{2}(q+2-2^*)}}\right), \quad \lambda_{\varepsilon}\to \infty,\quad a_{\varepsilon}\to x_{0}\in{\Omega}, \quad\lambda_{\varepsilon}d_{\varepsilon}\to\infty,  \quad \text{~as~}\varepsilon\to 0.
\end{equation}
Finally, the second term in (\ref{decomposition of u-varepsilon 3}) can also be obtained by decomposition (\ref{decomposition of u-varepsilon proof 5}) and Lemma \ref{estimate of U-lambda-a and psi-lambda-a 3}. Hence, we complete the proof.
\end{proof}

Next, we consider the $H^{1}_{0}(\Omega)$ estimate for the perturbation part $w_{\varepsilon}$.

\begin{Lem}\label{perturbation estimate of w-varepsilon}
For small enough $\varepsilon$, we have the following $H^{1}_{0}(\Omega)$ estimate for the perturbation part $w_{\varepsilon}$

 \begin{equation}\label{perturbation estimate of w-varepsilon 1}
 \|w_{\varepsilon}\|_{H^{1}_{0}(\Omega)}=\begin{cases}
     O\left(\frac{1}{\lambda_{\varepsilon}}\right), &\text{~~if~~} N=3,\\

         O\left(\frac{1}{\lambda_{\varepsilon}}\right), &\text{~~if~~} N=4 \text{~and~} q\in (2,\frac{5}{2}],\\
     O\left(\frac{1}{\lambda_{\varepsilon}^{2}}\right), &\text{~~if~~} N=4 \text{~and~} q\in (\frac{5}{2},4),\\
     
      O\left(\frac{1}{\lambda_{\varepsilon}^{5/2}}\right), &\text{~~if~~} N=5 \text{~and~} q\in (2,\frac{13}{6}],\\
      O\left(\frac{1}{\lambda_{\varepsilon}^{3}}\right), &\text{~~if~~} N=5 \text{~and~} q\in (\frac{13}{6},\frac{10}{3}),\\
O\left(\frac{(\ln\lambda_{\varepsilon})^{2/3}}{\lambda_{\varepsilon}^{4}}\right), &\text{~~if~~} N=6,\\
      O\left(\frac{1}{\lambda_{\varepsilon}^{(N+2)/2}}\right), &\text{~~if~~} N\geq 7.\\
 \end{cases}   
\end{equation}
\end{Lem}
\begin{proof}
Multiplying (\ref{p-varepsion}) by $w_{\varepsilon}$ and integrating by parts, note the $w_{\varepsilon}\in E_{\lambda_{\varepsilon},a_{\varepsilon}}$, then we have
\begin{equation}\label{perturbation estimate of w-varepsilon proof 1}
    \int_{\Omega}|\nabla w_{\varepsilon}|^{2}dx=\int_{\Omega}(u_{\varepsilon}^{2^*-1}+\varepsilon u_{\varepsilon}^{q-1})w_{\varepsilon}dx.
\end{equation}
Hence, by decomposition $u_{\varepsilon}=\alpha_{\varepsilon}PU_{\lambda_{\varepsilon},a_{\varepsilon}}+w_{\varepsilon}$ and inequality $(a+b)^{t}\leq C(a^{t}+b^{t})$ for some $C>0$ with $a, b>0, t>1$, we obtain that
\begin{equation}\label{perturbation estimate of w-varepsilon proof 2}
\begin{aligned}
  \int_{\Omega}|\nabla w_{\varepsilon}|^{2}dx&\leq C_{1}\int_{\Omega}PU_{\lambda_{\varepsilon},a_{\varepsilon}}^{2^*-1}w_{\varepsilon}dx+C_{2}\int_{\Omega}w_{\varepsilon}^{2^*}dx+C_{3}\varepsilon\int_{\Omega}PU_{\lambda_{\varepsilon},a_{\varepsilon}}^{q-1}w_{\varepsilon}dx+C_{4}\varepsilon\int_{\Omega}w_{\varepsilon}^{q}dx \\ 
  &:=I_{1}+I_{2}+I_{3}+I_{4}.
\end{aligned}
\end{equation}

Now, we estimate the four terms on the right hand side separately. First, since $w_{\varepsilon}\in E_{\lambda_{\varepsilon},a_{\varepsilon}}$ and $-\Delta PU_{\lambda_{\varepsilon},a_{\varepsilon}}=N(N-2) U_{\lambda_{\varepsilon},a_{\varepsilon}}^{2^*-1}$, we deduce
\begin{equation}\label{perturbation estimate of w-varepsilon proof 3}
\begin{aligned}
I_{1}&=C_{1}\int_{\Omega}(PU_{\lambda_{\varepsilon},a_{\varepsilon}}^{2^*-1}-U_{\lambda_{\varepsilon},a_{\varepsilon}}^{2^*-1})w_{\varepsilon}dx \\
&=C_{1}\int_{B(a_{\varepsilon},d_{\varepsilon})}(PU_{\lambda_{\varepsilon},a_{\varepsilon}}^{2^*-1}-U_{\lambda_{\varepsilon},a_{\varepsilon}}^{2^*-1})w_{\varepsilon}dx+O\left(\int_{\Omega\setminus B(a_{\varepsilon},d_{\varepsilon})}U_{\lambda_{\varepsilon},a_{\varepsilon}}^{2^*-1}|w_{\varepsilon}|dx\right).
\end{aligned}
\end{equation}
From Lemma \ref{estimate of U-lambda-a and psi-lambda-a 1}-Lemma \ref{estimate of U-lambda-a and psi-lambda-a 3}, H\"older inequality and Sobolev inequaity, we have
\begin{equation}\label{perturbation estimate of w-varepsilon proof 4}
\begin{aligned}
O\left(\int_{\Omega\setminus B(a_{\varepsilon},d_{\varepsilon})}U_{\lambda_{\varepsilon},a_{\varepsilon}}^{2^*-1}|w_{\varepsilon}|dx\right)&=O\left(\left(\int_{\Omega\setminus B(a_{\varepsilon},d_{\varepsilon})}U_{\lambda_{\varepsilon},a_{\varepsilon}}^{2^*}dx\right)^{\frac{2^*-1}{2^*}}\|w_{\varepsilon}\|_{H^{1}_{0}}\right)\\
&=O\left(\frac{1}{(\lambda_{\varepsilon}d_{\varepsilon})^{\frac{N+2}{2}}}\|w_{\varepsilon}\|_{H^{1}_{0}}\right),
\end{aligned}
\end{equation}
and
\begin{equation}\label{perturbation estimate of w-varepsilon proof 5}
    \begin{aligned}
     \int_{B(a_{\varepsilon},d_{\varepsilon})}(PU_{\lambda_{\varepsilon},a_{\varepsilon}}^{2^*-1}-U_{\lambda_{\varepsilon},a_{\varepsilon}}^{2^*-1})w_{\varepsilon}dx&=O\left(\int_{B(a_{\varepsilon},d_{\varepsilon})}\psi_{\lambda_{\varepsilon},a_{\varepsilon}}U_{\lambda_{\varepsilon},a_{\varepsilon}}^{2^*-2}|w_{\varepsilon}|dx\right)\\
     &=O\left(\|\psi_{\lambda_{\varepsilon},a_{\varepsilon}}\|_{L^{\infty}}\left(\int_{B(a_{\varepsilon},d_{\varepsilon})}U_{\lambda_{\varepsilon},a_{\varepsilon}}^{\frac{(2^*-2)2^*}{(2^*-1)}}dx\right)^{\frac{2^*-1}{2^*}}\|w_{\varepsilon}\|_{H^{1}_{0}}\right)\\
     &=\begin{cases}
         O\left(\frac{1}{(\lambda_{\varepsilon}d_{\varepsilon})^{N-2}}\|w_{\varepsilon}\|_{H^{1}_{0}}\right), &\text{~~if~~} N\leq 5,\\
          O\left(\frac{(\ln(\lambda_{\varepsilon}d_{\varepsilon}))^{\frac{2}3}}{(\lambda_{\varepsilon}d_{\varepsilon})^{N-2}}\|w_{\varepsilon}\|_{H^{1}_{0}}\right), &\text{~~if~~} N=6,\\
          O\left(\frac{1}{(\lambda_{\varepsilon}d_{\varepsilon})^{\frac{N+2}{2}}}\|w_{\varepsilon}\|_{H^{1}_{0}}\right), &\text{~~if~~} N\geq 7.
     \end{cases}
    \end{aligned}
\end{equation}
Thus by Cauchy inequality with $\varepsilon$, we have
\begin{equation}\label{perturbation estimate of w-varepsilon proof 6}
    I_{1}\leq \begin{cases}
         O\left(\frac{1}{(\lambda_{\varepsilon}d_{\varepsilon})^{2(N-2)}}\right), &\text{~~if~~} N\leq 5,\\
          O\left(\frac{(\ln(\lambda_{\varepsilon}d_{\varepsilon}))^{\frac{4}3}}{(\lambda_{\varepsilon}d_{\varepsilon})^{2(N-2)}}\right), &\text{~~if~~} N=6,\\
          O\left(\frac{1}{(\lambda_{\varepsilon}d_{\varepsilon})^{{N+2}}}\right), &\text{~~if~~} N\geq 7,
    \end{cases}+\frac{1}{8}\int_{\Omega}|\nabla w_{\varepsilon}|^{2}.
\end{equation}

Next, for sufficiently small $\varepsilon>0$, we have
\begin{equation}\label{perturbation estimate of w-varepsilon proof 7}
    I_{2}\leq C\left(\int_{\Omega}|\nabla w_{\varepsilon}|^{2}dx\right)^{\frac{2^*}{2}}\leq \frac{1}{8}\int_{\Omega}|\nabla w_{\varepsilon}|^{2},
\end{equation}
and
\begin{equation}\label{perturbation estimate of w-varepsilon proof 8}
    I_{4}\leq C\varepsilon\left(\int_{\Omega}|\nabla w_{\varepsilon}|^{2}dx\right)^{\frac{q}{2}}\leq \frac{1}{8}\int_{\Omega}|\nabla w_{\varepsilon}|^{2}.
\end{equation}

Finally, we estimate $I_{3}$. First, from H\"older inequality and Sobolev embedding, we obtain that
\begin{equation}\label{perturbation estimate of w-varepsilon proof 9}
\begin{aligned}
  \int_{\Omega}(PU_{\lambda_{\varepsilon},a_{\varepsilon}})^{q-1}w_{\varepsilon}dx&\leq \|PU_{\lambda_{\varepsilon},a_{\varepsilon}}^{q-1}\|_{L^{p}}\|w_{\varepsilon}\|_{L^{p_{1}}}\leq C \|U_{\lambda_{\varepsilon},a_{\varepsilon}}^{q-1}\|_{L^{p}}\|w_{\varepsilon}\|_{H^{1}_{0}}.
\end{aligned}
\end{equation}
where $\frac{1}{p}+\frac{1}{p_{1}}=1$ and $p\in [\frac{2N}{N+2},+\infty]$. Moreover, there exists constants $C,R>0$ such that
\begin{equation}\label{perturbation estimate of w-varepsilon proof 10}
\begin{aligned}
 \left(\int_{\Omega}U_{\lambda_{\varepsilon},a_{\varepsilon}}^{(q-1)p}dx\right)^{\frac{1}{p}}&=\lambda_{\varepsilon}^{\frac{(N-2)(q-1)}{2}-\frac{N}{p}}\left(\int_{\lambda_{\varepsilon}(\Omega-a_{\varepsilon})}\left(\frac{1}{1+|y|^{2}}\right)^{\frac{(N-2)(q-1)p}{2}}dy\right)^{\frac{1}{p}}\\
 &\leq C\lambda_{\varepsilon}^{\frac{(N-2)(q-1)}{2}-\frac{N}{p}}\left(\int_{0}^{\lambda_{\varepsilon}R}\frac{r^{N-1}}{(1+r^2)^{\frac{(N-2)(q-1)p}{2}}}dr\right)^{\frac{1}{p}}.
\end{aligned}
\end{equation}
Next, we will choose the appropriate $p$ used in \eqref{perturbation estimate of w-varepsilon proof 9} to estimate term $I_{3}$, which  depends on the selection of $(N,q)$. First, when $(N,q)$ satisfy the following condition
\begin{equation}\label{perturbation estimate of w-varepsilon proof 11}
    \begin{cases}
        q\in (4,6),&\quad \text{~~if~~} N= 3,\\
        q\in (\frac{5}{2},4), &\quad \text{~~if~~} N= 4,\\
        q\in (\frac{13}{6},\frac{10}{3}), &\quad \text{~~if~~} N= 5,\\
        q\in (2,2^*), &\quad \text{~~if~~} N\geq 6,\\
    \end{cases}
\end{equation}
we take $p=\frac{2N}{N+2}$. Then $(N-2)(q-1)p>N$ and by \eqref{perturbation estimate of w-varepsilon proof 9} and \eqref{perturbation estimate of w-varepsilon proof 10}, we can obtain that
\begin{equation}\label{perturbation estimate of w-varepsilon proof 12}
  \int_{\Omega}(PU_{\lambda_{\varepsilon},a_{\varepsilon}})^{q-1}w_{\varepsilon}dx\leq C\frac{1}{\lambda_{\varepsilon}^{N-\frac{N-2}{2}q}}\|w_{\varepsilon}\|_{H^{1}_{0}}.  
\end{equation}
On the other hand, when $N=4$ and $q\in(2,\frac{5}{2}]$, we take $p=2$. Thus $(N-2)(q-1)p>N$ and by \eqref{perturbation estimate of w-varepsilon proof 9} and \eqref{perturbation estimate of w-varepsilon proof 10}
\begin{equation}\label{perturbation estimate of w-varepsilon proof 13}
   \int_{\Omega}(PU_{\lambda_{\varepsilon},a_{\varepsilon}})^{q-1}w_{\varepsilon}dx\leq C\frac{1}{\lambda_{\varepsilon}^{3-q}} \|w_{\varepsilon}\|_{H^{1}_{0}}. 
\end{equation}
Finally, when $N=5$ and $q\in(2,\frac{13}{6}]$, we take $p=\frac{5}{3}$. Similarly, we can conclude that
\begin{equation}\label{perturbation estimate of w-varepsilon proof 14}
  \int_{\Omega}(PU_{\lambda_{\varepsilon},a_{\varepsilon}})^{q-1}w_{\varepsilon}dx\leq C\frac{1}{\lambda_{\varepsilon}^{\frac{9-3q}{2}}}\|w_{\varepsilon}\|_{H^{1}_{0}}.  
\end{equation}
Now, combining the above argument with Cauchy inequality, the term $I_{3}$ can be estimated as
\begin{equation}\label{perturbation estimate of w-varepsilon proof 15}
\begin{aligned}
     I_{3}=\begin{cases}
        O\left(\frac{\varepsilon^{2}}{\lambda_{\varepsilon}^{6-q}}\right), &\text{~~if~~} N=3,\\

         O\left(\frac{\varepsilon^{2}}{\lambda_{\varepsilon}^{6-2q}}\right), &\text{~~if~~} N=4 \text{~and~} q\in (2,\frac{5}{2}],\\
     O\left(\frac{\varepsilon^{2}}{\lambda_{\varepsilon}^{8-2q}}\right), &\text{~~if~~} N=4, \text{~and~} q\in (\frac{5}{2},4),\\
     
      O\left(\frac{\varepsilon^{2}}{\lambda_{\varepsilon}^{9-3q}}\right), &\text{~~if~~} N=5 \text{~and~} q\in (2,\frac{13}{6}],\\
      O\left(\frac{\varepsilon^{2}}{\lambda_{\varepsilon}^{10-3q}}\right), &\text{~~if~~} N=5 \text{~and~} q\in (\frac{13}{6},\frac{10}{3}),\\
O\left(\frac{\varepsilon^{2}}{\lambda_{\varepsilon}^{12-4q}}\right), &\text{~~if~~} N=6,\\
      O\left(\frac{\varepsilon^{2}}{\lambda_{\varepsilon}^{2N-(N-2)q}}\right), &\text{~~if~~} N\geq 7,\\ 
     \end{cases}+\frac{1}{8}\int_{\Omega}|\nabla w_{\varepsilon}|^{2}.
\end{aligned}
\end{equation}

Putting the estimates (\ref{perturbation estimate of w-varepsilon proof 6}), (\ref{perturbation estimate of w-varepsilon proof 7}), (\ref{perturbation estimate of w-varepsilon proof 8}) and (\ref{perturbation estimate of w-varepsilon proof 15}) together, we deduce from (\ref{perturbation estimate of w-varepsilon proof 2}) that
\begin{equation}
 \|w_{\varepsilon}\|_{H^{1}_{0}(\Omega)}=\begin{cases}
     O\left(\frac{1}{\lambda_{\varepsilon}d_{\varepsilon}}+\frac{\varepsilon}{\lambda_{\varepsilon}^{3-\frac{q}{2}}}\right), &\text{~~if~~} N=3,\\

         O\left(\frac{1}{(\lambda_{\varepsilon}d_{\varepsilon})^{2}}+\frac{\varepsilon}{\lambda_{\varepsilon}^{3-q}}\right), &\text{~~if~~} N=4 \text{~and~} q\in (2,\frac{5}{2}],\\
     O\left(\frac{1}{(\lambda_{\varepsilon}d_{\varepsilon})^{2}}+\frac{\varepsilon}{\lambda_{\varepsilon}^{4-q}}\right), &\text{~~if~~} N=4 \text{~and~} q\in (\frac{5}{2},4),\\
     
      O\left(\frac{1}{(\lambda_{\varepsilon}d_{\varepsilon})^{3}}+\frac{\varepsilon}{\lambda_{\varepsilon}^{\frac{9}{2}-\frac{3}{2}q}}\right), &\text{~~if~~} N=5 \text{~and~} q\in (2,\frac{13}{6}],\\
      O\left(\frac{1}{(\lambda_{\varepsilon}d_{\varepsilon})^{3}}+\frac{\varepsilon}{\lambda_{\varepsilon}^{5-\frac{3}{2}q}}\right), &\text{~~if~~} N=5 \text{~and~} q\in (\frac{13}{6},\frac{10}{3}),\\
O\left(\frac{(\ln(\lambda_{\varepsilon}d_{\varepsilon}))^{\frac{2}{3}}}{(\lambda_{\varepsilon}d_{\varepsilon})^{4}}+\frac{\varepsilon}{\lambda_{\varepsilon}^{6-2q}}\right), &\text{~~if~~} N=6,\\
      O\left(\frac{1}{(\lambda_{\varepsilon}d_{\varepsilon})^{\frac{N+2}{2}}}+\frac{\varepsilon}{\lambda_{\varepsilon}^{N-\frac{N-2}{2}q}}\right), &\text{~~if~~} N\geq 7.\\
 \end{cases}   
\end{equation}
Due to the fact
\begin{equation}
        \varepsilon=O\left(\frac{1}{\lambda_{\varepsilon}^{\frac{N-2}{2}(q+2-2^*)}}\right)\text{~~and~~} a_{\varepsilon}\to x_{0}\in\Omega,
    \end{equation}
thus, (\ref{perturbation estimate of w-varepsilon 1}) follows.
\end{proof}

Finally, we can obtain a precise asymptotic expansion on least energy $S_{\varepsilon}$.
\begin{Prop}\label{energy estimate of S-varepsilon}
As $\varepsilon\to 0 $, we have
\begin{equation}
    S_{\varepsilon}=\frac{1}{N}S^{\frac{N}{2}}-\left(\frac{1}{2^*-q}-\frac{1}{2}\right)\frac{\alpha_{N}^{22^*-2}}{N^{2}}\omega_{N}^{2}R(a_{\varepsilon})\frac{1}{\lambda_{\varepsilon}^{N-2}}+o\left(\frac{1}{\lambda_{\varepsilon}^{N-2}}\right). 
\end{equation}
\end{Prop}
\begin{proof}
  Using the decomposition 
$U_{\lambda_{\varepsilon},a_{\varepsilon}}=PU_{\lambda_{\varepsilon},a_{\varepsilon}}+\psi_{\lambda_{\varepsilon},a_{\varepsilon}}$, Lemma \ref{estimate of U-lambda-a and psi-lambda-a 2} and Lemma \ref{estimate of U-lambda-a and psi-lambda-a 3} with a similar argument as  in \cite[Lemma 2.1]{FunkcialajEkvacioj2004}, we have
\begin{equation}
\int_{\Omega} |\nabla PU_{\lambda_{\varepsilon},a_{\varepsilon}}|^{2}d x=\alpha_{N}^{-2}S^{\frac{N}{2}}-(N-2)^{2}\omega_{N}^{2}R(a_{\varepsilon})\frac{1}{\lambda_{\varepsilon}^{N-2}}+o\left(\frac{1}{\lambda_{\varepsilon}^{N-2}}\right),
\end{equation}
and
\begin{equation}
    \int_{\Omega} PU_{\lambda_{\varepsilon},a_{\varepsilon}}^{2^*}dx=\alpha_{N}^{-2^*}S^{\frac{N}{2}}-2\omega_{N}^{2}R(a_{\varepsilon})\frac{1}{\lambda_{\varepsilon}^{N-2}}+o\left(\frac{1}{\lambda_{\varepsilon}^{N-2}}\right).
\end{equation}
Next, we estimate the sub-critical term. By Lemma \ref{Elementary estimate} and decomposition 
$U_{\lambda_{\varepsilon},a_{\varepsilon}}=PU_{\lambda_{\varepsilon},a_{\varepsilon}}+\psi_{\lambda_{\varepsilon},a_{\varepsilon}}$
\begin{equation}\label{estimate of sub-critical term PU}
    \int_{\Omega}U_{\lambda_{\varepsilon},a_{\varepsilon}}^{q}dx=\int_{\Omega}PU_{\lambda_{\varepsilon},a_{\varepsilon}}^{q}dx+q\int_{\Omega}PU_{\lambda_{\varepsilon},a_{\varepsilon}}^{q-1}\psi_{\lambda_{\varepsilon},a_{\varepsilon}}dx+O\left(\int_{\Omega}PU_{\lambda_{\varepsilon},a_{\varepsilon}}^{q-2}\psi^{2}_{\lambda_{\varepsilon},a_{\varepsilon}}dx\right).
\end{equation}
The right-hand side of above equation can be  estimated  as follows
\begin{equation}\label{estimate of sub-critical term PU-1}
\begin{aligned}
\int_{\Omega}U_{\lambda_{\varepsilon},a_{\varepsilon}}^{q}&=\int_{\R^{N}}U_{\lambda_{\varepsilon},a_{\varepsilon}}^{q}-\int_{\R^{N}\setminus\Omega}U_{\lambda_{\varepsilon},a_{\varepsilon}}^{q}\\
&=\frac{\omega_{N}C_{N,q}}{\lambda_{\varepsilon}^{N-\frac{N-2}{2}q}}+O\left(\frac{1}{{\lambda_{\varepsilon}^{N-\frac{N-2}{2}q}}}\int_{\lambda_{\varepsilon}d_{\varepsilon}}^{\infty}\frac{r^{N-1}}{(1+r^2)^{\frac{N-2}{2}q}}\right)\\
&=\frac{\omega_{N}C_{N,q}}{\lambda_{\varepsilon}^{N-\frac{N-2}{2}q}}+
O\left(\frac{1}{\lambda_{\varepsilon}^{\frac{(N-2)q}{2}}d_{\varepsilon}^{(N-2)q-N}}\right),     
\end{aligned}
\end{equation}
where
\begin{equation}
 C_{N,q}=\int_{0}^{\infty}\frac{r^{N-1}}{(1+r^2)^{\frac{N-2}{2}q}}dr=\frac{\Gamma(\frac{N}{2})\Gamma(\frac{N-2}{2}q-\frac{N}{2})}{2\Gamma(\frac{N-2}{2}q)}.   
\end{equation}
The second and third term of \eqref{estimate of sub-critical term PU} can be bound by
\begin{equation}\label{estimate of sub-critical term PU-2}
\begin{aligned}
  \int_{\Omega}U_{\lambda_{\varepsilon},a_{\varepsilon}}^{q-1}\psi_{\lambda_{\varepsilon},a_{\varepsilon}}dx&\leq  \left(\int_{\Omega}U_{\lambda_{\varepsilon},a_{\varepsilon}}^{q}dx\right)^{\frac{q-1}{q}}\left(\int_{\Omega}\psi_{\lambda_{\varepsilon},a_{\varepsilon}}^{q}dx\right)^{\frac{1}{q}}\\
  &\lesssim \frac{1}{\lambda_{\varepsilon}^{(N-\frac{N-2}{2}q)\frac{q-1}{q}}}\frac{1}{\lambda_{\varepsilon}^{\frac{N-2}{2}}}=o\left(\frac{1}{\lambda_{\varepsilon}^{N-\frac{N-2}{2}q}}\right),
\end{aligned}
\end{equation}  
where H\"older inequality, Lemma \ref{estimate of U-lambda-a and psi-lambda-a 3} and $q>\frac{N}{N-2}$ has been used. Thus, by \eqref{estimate of sub-critical term PU}-\eqref{estimate of sub-critical term PU-2}
\begin{equation}
 \int_{\Omega}PU_{\lambda_{\varepsilon},a_{\varepsilon}}^{q}dx=C_{N,q}\omega_{N}\frac{1}{\lambda_{\varepsilon}^{N-\frac{N-2}{2}q}}+o\left(\frac{1}{\lambda_{\varepsilon}^{N-\frac{N-2}{2}q}}\right).
\end{equation}
Finally, by Proposition \ref{blow-up rate} and Lemma \ref{decomposition of u-varepsilon}, we have the following estimate on $\varepsilon$
\begin{equation}
 \varepsilon=\frac{q}{2^*-q}\frac{\alpha_{N}^{22^*-2-q}}{N^{2}}\omega_{N}C_{N,q}^{-1}R(a_{\varepsilon})\frac{1}{\lambda_{\varepsilon}^{\frac{N-2}{2}q-2}}+o\left(\frac{1}{\lambda_{\varepsilon}^{\frac{N-2}{2}q-2}}\right).   
\end{equation}

Moreover, using the orthogonal decomposition $u_{\varepsilon}=\alpha_{\varepsilon}PU_{\lambda_{\varepsilon},a_{\varepsilon}}+w_{\varepsilon}$, Lemma \ref{perturbation estimate of w-varepsilon} and above results, it's easy to obtain
\begin{equation}
    \begin{aligned}
        \int_{\Omega} |\nabla u_{\varepsilon}|^{2}d x=\alpha_{\varepsilon}^{2}\alpha_{N}^{-2}S^{\frac{N}{2}}-\alpha_{\varepsilon}^{2}(N-2)^{2}\omega_{N}^{2}R(a_{\varepsilon})\frac{1}{\lambda_{\varepsilon}^{N-2}}+o\left(\frac{1}{\lambda_{\varepsilon}^{N-2}}\right)
    \end{aligned}
\end{equation}
\begin{equation}
    \int_{\Omega}  u_{\varepsilon}^{2^*} d x=\alpha_{\varepsilon}^{2^*}\alpha_{N}^{-2^*}S^{\frac{N}{2}}-2\alpha_{\varepsilon}^{2^*}\omega_{N}^{2}R(a_{\varepsilon})\frac{1}{\lambda_{\varepsilon}^{N-2}}+o\left(\frac{1}{\lambda_{\varepsilon}^{N-2}}\right),
\end{equation}
%\begin{equation}
%\int_{\Omega} u_{\varepsilon}^{q} d x=\alpha_{\varepsilon}^{q}C_{N,q}\omega_{N}\frac{1}%%{\lambda_{\varepsilon}^{N-\frac{N-2}{2}q}}+o\left(\frac{1}{\lambda_{\varepsilon}^{N-\frac{N-2}{2}q}}\right),
%\end{equation}
and
\begin{equation}
   \varepsilon \int_{\Omega} u_{\varepsilon}^{q} d x= \frac{q}{2^*-q}\frac{\alpha_{N}^{22^*-2-q}\alpha_{\varepsilon}^{q}}{N^{2}}\omega_{N}^{2}R(a_{\varepsilon})\frac{1}{\lambda_{\varepsilon}^{N-2}}+o\left(\frac{1}{\lambda_{\varepsilon}^{N-2}}\right).
\end{equation}

Next, we estimate $\alpha_{\varepsilon}$. Recall the following Nehari identity
\begin{equation}
    \int_{\Omega} |\nabla u_{\varepsilon}|^{2}d x=\int_{\Omega}  u_{\varepsilon}^{2^*} d x+\varepsilon \int_{\Omega} u_{\varepsilon}^{q} d x.
\end{equation}
Hence
\begin{equation}
    \alpha_{\varepsilon}=\alpha_{N}+c_{\varepsilon}\frac{1}{\lambda_{\varepsilon}^{N-2}}+o\left(\frac{1}{\lambda_{\varepsilon}^{N-2}}\right),
\end{equation}
where $c_{\varepsilon}$ satisfy that
\begin{equation}
    (2^*-2)\alpha_{N}^{-1}S^{\frac{N}{2}}c_{\varepsilon}=\left(2\alpha_{N}^{2^*}-\alpha_{N}^{2}(N-2)^{2}-\frac{q}{2^*-q}\frac{\alpha_{N}^{22^*-2}}{N^{2}}\right)\omega_{N}^{2}R(a_{\varepsilon}).
\end{equation}

Combining every thing, we have
\begin{equation}
\begin{aligned}
    \int_{\Omega} |\nabla u_{\varepsilon}|^{2}d x&=S^{\frac{N}{2}}+\left(2\alpha_{N}^{-1}S^{\frac{N}{2}}c_{\varepsilon}-\alpha_{N}^{2}(N-2)^{2}\omega_{N}^{2}R(a_{\varepsilon})\right)\frac{1}{\lambda_{\varepsilon}^{N-2}}+o\left(\frac{1}{\lambda_{\varepsilon}^{N-2}}\right) \\
    &=\left(\frac{4}{2^*-2}\alpha_{N}^{2^*}-\frac{2^*}{2^*-2}\alpha_{N}^{2}(N-2)^{2}-\frac{2q}{(2^*-2)(2^*-q)}\frac{\alpha_{N}^{22^*-2}}{N^{2}}\right)\omega_{N}^{2}R(a_{\varepsilon})\frac{1}{\lambda_{\varepsilon}^{N-2}}\\
    &\quad +S^{\frac{N}{2}}+o\left(\frac{1}{\lambda_{\varepsilon}^{N-2}}\right) \\
    &=S^{\frac{N}{2}}+\left(\frac{N}{2}-\frac{2q}{(2^*-2)(2^*-q)}\right)\frac{\alpha_{N}^{22^*-2}}{N^{2}}\omega_{N}^{2}R(a_{\varepsilon})\frac{1}{\lambda_{\varepsilon}^{N-2}}+o\left(\frac{1}{\lambda_{\varepsilon}^{N-2}}\right), 
\end{aligned}
\end{equation}

\begin{equation}
    \begin{aligned}
        \int_{\Omega} u_{\varepsilon}^{2^*}d x&=S^{\frac{N}{2}}+\left(2^{*}\alpha_{N}^{-1}S^{\frac{N}{2}}c_{\varepsilon}-2\alpha_{N}^{2^*}\omega_{N}^{2}R(a_{\varepsilon})\right)\frac{1}{\lambda_{\varepsilon}^{N-2}}+o\left(\frac{1}{\lambda_{\varepsilon}^{N-2}}\right)\\
         &=\left(\frac{4}{2^*-2}\alpha_{N}^{2^*}-\frac{2^*}{2^*-2}\alpha_{N}^{2}(N-2)^{2}-\frac{2^{*}q}{(2^*-2)(2^*-q)}\frac{\alpha_{N}^{22^*-2}}{N^{2}}\right)\omega_{N}^{2}R(a_{\varepsilon})\frac{1}{\lambda_{\varepsilon}^{N-2}}\\
    &\quad +S^{\frac{N}{2}}+o\left(\frac{1}{\lambda_{\varepsilon}^{N-2}}\right) \\
     &=S^{\frac{N}{2}}+\left(\frac{N}{2}-\frac{2^{*}q}{(2^*-2)(2^*-q)}\right)\frac{\alpha_{N}^{22^*-2}}{N^{2}}\omega_{N}^{2}R(a_{\varepsilon})\frac{1}{\lambda_{\varepsilon}^{N-2}}+o\left(\frac{1}{\lambda_{\varepsilon}^{N-2}}\right), 
    \end{aligned}
\end{equation}
and
\begin{equation}
   \varepsilon \int_{\Omega} u_{\varepsilon}^{q} d x= \frac{q}{2^*-q}\frac{\alpha_{N}^{22^*-2}}{N^{2}}\omega_{N}^{2}R(a_{\varepsilon})\frac{1}{\lambda_{\varepsilon}^{N-2}}+o\left(\frac{1}{\lambda_{\varepsilon}^{N-2}}\right).
\end{equation}
Then, the estimate on $S_{\varepsilon}$ is stated as follows
\begin{equation}
    S_{\varepsilon}=\frac{1}{N}S^{\frac{N}{2}}-\left(\frac{1}{2^*-q}-\frac{1}{2}\right)\frac{\alpha_{N}^{22^*-2}}{N^{2}}\omega_{N}^{2}R(a_{\varepsilon})\frac{1}{\lambda_{\varepsilon}^{N-2}}+o\left(\frac{1}{\lambda_{\varepsilon}^{N-2}}\right), 
\end{equation}
and by assumption $q>2^*-2$, it's easy to see that 
\begin{equation}
    \frac{1}{2^*-q}-\frac{1}{2}>0.
\end{equation}
\end{proof}

\begin{proof}[\textrm{Proof of Theorem} $\ref{main theorem-1}$]    
From the results in above subsections, it only remains to show $(\ref{contentrate speed})$ hold. In fact, it can be obtained by a similar proof as Claim $1$ and Claim $2$ in Theorem $\ref{local uniqueness-Robin function}$ below and we omit it.  
\end{proof}

\begin{Rem}
It's worth note that the orthogonal decomposition of the least energy solution $u_{\varepsilon}$ in Lemma \ref{decomposition of u-varepsilon} is very useful. In fact, we can use this decomposition to give another proof of blow-up rate and location of blow-up point, see (\cite{Rey1989ProofOT}) for the main idea.     
\end{Rem}

\section{Asymptotic uniqueness}\label{Asymptotic uniqueness}
\subsection{Convex domain}
In this subsection, we will prove the asymptotic uniqueness of the least energy solution, when $\Omega$ is a convex domain. 

\begin{Assum}\label{Assum-1}
    $\Omega$ is convex in the $x_{i}$ directions and symmetric with respect to the hyperplanes $\{x_{i}=0\}$ for $i=1,\cdots, N$.
\end{Assum}

\begin{Thm}\label{asymptotic uniqueness-1}
Let $\Omega$ be a smooth bounded star-shape domain of $\R^{N}$ satisfying assumption $\ref{Assum-1}$, $N\geq 5$ and $q\in (2,2^*)$. Suppose that $u_{\varepsilon}$ and $v_{\varepsilon}$ are two least energy solutions of (\ref{p-varepsion}), then there exists $\varepsilon_{0}>0$ such that for any $\varepsilon<\varepsilon_{0}$
\begin{equation}
    u_{\varepsilon}\equiv v_{\varepsilon}\quad \text{~in~}\Omega.
\end{equation}
\end{Thm}

\begin{proof}[\textrm{Proof of Theorem} $\ref{asymptotic uniqueness-1}$]    
By Gidas-Ni-Nirenberg theorem \cite{Gidas1979Symmetry}, $u_{\varepsilon}$ and $v_{\varepsilon}$ are symmetric with respect to the hyperplanes $\{x_{i}=0\}$ for $i=1,\cdots,N$ and $u_{\varepsilon}(0)=\|u_{\varepsilon}\|_{L^{\infty}(\Omega)},v_{\varepsilon}(0)=\|v_{\varepsilon}\|_{L^{\infty}(\Omega)}$. Next, we will show Theorem \ref{asymptotic uniqueness-1} hold by several steps.

\textbf{~Step~}1. we define 
\begin{equation}
    \tilde{w}_{\varepsilon}(x):=u_{\varepsilon}\left(\frac{x}{\|u_{\varepsilon}\|_{L^{\infty}(\Omega)}^{(2^*-2)/2}}\right)-v_{\varepsilon}\left(\frac{x}{\|u_{\varepsilon}\|_{L^{\infty}(\Omega)}^{(2^*-2)/2}}\right), \quad x\in\Omega_{\varepsilon}:=\|u_{\varepsilon}\|_{L^{\infty}(\Omega)}^{(2^*-2)/2}\Omega,
\end{equation}
and
\begin{equation}
    w_{\varepsilon}:=\frac{\tilde{w}_{\varepsilon}}{\|\tilde{w}_{\varepsilon}\|_{L^{\infty}(\Omega_{\varepsilon})}}=\frac{\tilde{w}_{\varepsilon}}{\|u_{\varepsilon}-v_{\varepsilon}\|_{L^{\infty}(\Omega)}}.
\end{equation}
Then $\|w_{\varepsilon}\|_{L^{\infty}(\Omega_{\varepsilon})}=1$ and $w_{\varepsilon}$ satisfy
\begin{equation}\label{asymptotic uniqueness-1-proof-step-1-3}
    -\Delta w_{\varepsilon}=\left(c_{2^*,\varepsilon}(x)+\frac{\varepsilon}{\|u_{\varepsilon}\|_{L^{\infty}(\Omega)}^{2^*-q}}c_{q,\varepsilon}(x)\right)w_{\varepsilon},
\end{equation}
where
\begin{equation}
    c_{2^*,\varepsilon}=(2^*-1)\int_{0}^{1}\left[t\frac{u_{\varepsilon}}{\|u_{\varepsilon}\|_{L^{\infty}(\Omega)}}\left(\frac{x}{\|u_{\varepsilon}\|_{L^{\infty}(\Omega)}^{(2^*-2)/2}}\right)+(1-t)\frac{v_{\varepsilon}}{\|u_{\varepsilon}\|_{L^{\infty}(\Omega)}}\left(\frac{x}{\|u_{\varepsilon}\|_{L^{\infty}(\Omega)}^{(2^*-2)/2}}\right)\right]^{2^*-2}dt
\end{equation}
and
\begin{equation}
    c_{q,\varepsilon}=(q-1)\int_{0}^{1}\left[t\frac{u_{\varepsilon}}{\|u_{\varepsilon}\|_{L^{\infty}(\Omega)}}\left(\frac{x}{\|u_{\varepsilon}\|_{L^{\infty}(\Omega)}^{(2^*-2)/2}}\right)+(1-t)\frac{v_{\varepsilon}}{\|u_{\varepsilon}\|_{L^{\infty}(\Omega)}}\left(\frac{x}{\|u_{\varepsilon}\|_{L^{\infty}(\Omega)}^{(2^*-2)/2}}\right)\right]^{q-2}dt.
\end{equation}
Thus by (2) and (6) in Theorem \ref{main theorem-1}, we deduce that
\begin{equation}
    c_{2^*,\varepsilon}(x)\to (2^*-1)U^{2^*-2}(x), \quad C_{q,\varepsilon}(x)\to (q-1)U^{q-2}(x),
\end{equation}
uniformly on compact sets of $\R^{N}$. Note the $\|w_{\varepsilon}\|_{L^{\infty}(\Omega)}=1$ and $\frac{\varepsilon}{\|u_{\varepsilon}\|_{L^{\infty}(\Omega)}^{2^*-q}}\to 0$, then by elliptic regularity theory in \cite{gilbarg1977elliptic}, there exist a function $w(x)\in C^{1,\alpha}_{loc}(\R^{N})$ such that up to a sub-sequence $w_{\varepsilon}\to w$ in $C^{1}_{loc}(\R^{N})$. Moreover, $w(x)$ is symmetric and satisfies
\begin{equation}\label{asymptotic uniqueness-1-proof-step-1-7}
    \begin{cases}
        -\Delta w=(2^*-1)U^{2^*-2}w, \text{~~in~~}\R^{N},\\
        \|w\|_{L^{\infty}(\R^{N})}\leq 1.     
    \end{cases}
\end{equation}

\textbf{Step} 2. In this step, we will show that $w\in\mathcal{D}^{1,2}(\R^{N})$. Firstly, by (\ref{asymptotic uniqueness-1-proof-step-1-3}) we have
\begin{equation}\label{asymptotic uniqueness-1-proof-step-2-1}
    \int_{\Omega_{\varepsilon}}|\nabla w_{\varepsilon}|^{2}dx=\int_{\Omega_{\varepsilon}}\left(c_{2^*,\varepsilon}(x)+\frac{\varepsilon}{\|u_{\varepsilon}\|_{L^{\infty}(\Omega)}^{2^*-q}}c_{q,\varepsilon}(x)\right)w_{\varepsilon}^{2}dx.
\end{equation}
In addition, by (3) and (6) in Theorem \ref{main theorem-1}, we have that there exist constants $C>0$ and $\varepsilon_{1}>0$ such $c_{q,\varepsilon}(x)\leq C$, for any $\varepsilon<\varepsilon_{1}$ and $x\in \Omega_{\varepsilon}$. This together with Poincar\'{e} inequality, we deduce that
\begin{equation}\label{asymptotic uniqueness-1-proof-step-2-2}
    \int_{\Omega_{\varepsilon}}\frac{\varepsilon}{\|u_{\varepsilon}\|_{L^{\infty}(\Omega)}^{2^*-q}}c_{q,\varepsilon}(x)w_{\varepsilon}^{2}dx\leq C\frac{\varepsilon}{\|u_{\varepsilon}\|_{L^{\infty}(\Omega)}^{2^*-q}}\int_{\Omega_{\varepsilon}}w_{\varepsilon}^{2}dx\leq \varepsilon\|u_{\varepsilon}\|_{L^{\infty}(\Omega)}^{q-2}\frac{C}{\lambda_{1}(\Omega)}\int_{\Omega_{\varepsilon}}|\nabla w_{\varepsilon}|^{2}dx,
\end{equation}
for any $\varepsilon<\varepsilon_{1}$, where $\lambda_{1}(\Omega)$ is the first eigenvalue of the Laplace operator in $H^{1}_{0}(\Omega)$. Thus, by (\ref{asymptotic uniqueness-1-proof-step-2-1}), (\ref{asymptotic uniqueness-1-proof-step-2-2}) and (6) in Theorem \ref{main theorem-1}, we obtain that
\begin{equation}\label{asymptotic uniqueness-1-proof-step-2-3}
    (1+o(1))\int_{\Omega_{\varepsilon}}|\nabla w_{\varepsilon}|^{2}\leq \int_{\Omega_{\varepsilon}}c_{2^*,\varepsilon}(x)w_{\varepsilon}^{2}(x)dx,
\end{equation}
which together with Sobolev inequality, we have
\begin{equation}\label{asymptotic uniqueness-1-proof-step-2-4}
    (S+o(1))\left(\int_{\Omega_{\varepsilon}}|w_{\varepsilon}^{2^*}|dx\right)^{\frac{2}{2^*}}\leq \int_{\Omega_{\varepsilon}}c_{2^*,\varepsilon}(x)w_{\varepsilon}^{2}(x)dx.
\end{equation}
Next, we let $0<\delta<2^*-2$, since $\|w_{\varepsilon}\|_{L^{\infty}(\Omega_{\varepsilon})}=1$, by H\"older inequality we have
\begin{equation}\label{asymptotic uniqueness-1-proof-step-2-5}
    \begin{aligned}
        \int_{\Omega_{\varepsilon}}c_{2^*,\varepsilon}(x)w_{\varepsilon}^{2}dx&\leq \int_{\Omega_{\varepsilon}}c_{2^*,\varepsilon}(x)w_{\varepsilon}^{2-\delta}dx\\
        &\leq \left(\int_{\Omega_{\varepsilon}}|w_{\varepsilon}|^{2^*}dx\right)^{\frac{(2-\delta)}{2^*}}\left(\int_{\Omega_{\varepsilon}}c_{2^*,\varepsilon}(x)^{\frac{2^*}{2^*-2+\delta}}dx\right)^{\frac{2^*-2+\delta}{2^*}}.
    \end{aligned}
\end{equation}
By (\ref{asymptotic uniqueness-1-proof-step-2-4}), (\ref{asymptotic uniqueness-1-proof-step-2-5}) and Theorem \ref{main theorem-1}, we obtain that for small enough $\varepsilon$
\begin{equation}\label{asymptotic uniqueness-1-proof-step-2-6}
\begin{aligned}
\left(\int_{\Omega_{\varepsilon}}|w_{\varepsilon}^{2^*}|dx\right)^{\frac{\delta}{2^*}}&\leq  C \left(\int_{\Omega_{\varepsilon}}c_{2^*,\varepsilon}(x)^{\frac{2^*}{2^*-2+\delta}}dx\right)^{\frac{2^*-2+\delta}{2^*}}\\
&\leq C\left (\int_{\R^{N}}\frac{dx}{(N(N-2)+|x|^{2})^{22^*/(2^*-2+\delta)}}dx\right)^{(2^*-2+\delta)/2^*}\leq C,    
\end{aligned}
\end{equation}
Finally, by (\ref{asymptotic uniqueness-1-proof-step-2-3}) and (\ref{asymptotic uniqueness-1-proof-step-2-6}) for small $\varepsilon>0$ we have
\begin{equation}\label{asymptotic uniqueness-1-proof-step-2-7}
\begin{aligned}
    \int_{\Omega_{\varepsilon}}|\nabla w_{\varepsilon}|^{2}dx &\leq C \int_{\Omega_{\varepsilon}}c_{2^*,\varepsilon}(x)w_{\varepsilon}^{2}dx\leq \left(\int_{\Omega_{\varepsilon}}c_{2^*,\varepsilon}(x)^{\frac{N}{2}}\right)^{\frac{2}{N}}\left(\int_{\Omega_{\varepsilon}}w_{\varepsilon}^{2^*}dx\right)^{\frac{2}{2^*}}\\
    &\leq C\int_{\R^{N}}\frac{dx}{(N(N-2)+|x|^{2})^{N}}\leq C,
\end{aligned} 
\end{equation}
Thus 
\begin{equation}\label{asymptotic uniqueness-1-proof-step-2-8}
    \int_{\R^{N}}|\nabla w|^{2}dx\leq  \liminf_{\varepsilon\to 0} \int_{\Omega_{\varepsilon}}|\nabla w_{\varepsilon}|^{2}dx\leq C.
\end{equation}

By Lemma \ref{Kernel of Emden-Fowler equation} together with (\ref{asymptotic uniqueness-1-proof-step-1-7}) and (\ref{asymptotic uniqueness-1-proof-step-2-8}), we know that $w$ can be written as
\begin{equation}
w(x)=\sum_{i=1}^{N}\frac{a_{i}x_{i}}{(N(N-2)+|x|^{2})^{\frac{N}{2}}}+b\frac{N(N-2)-|x|^{2}}{(N(N-2)+|x|^{2})^{\frac{N}{2}}},
\end{equation}
for some $a_{i},b\in\R$.
since $w$ is symmetric, so $a_{i}=0$ for any $i=1,\cdots,N$. Next we will show that $b=0$. Firstly, we give the following estimate for $w_{\varepsilon}$.

\textbf{Step} 3. Let $\omega_{0}$ be a neighbourhood of the origin, we will show that there exist some constant $C>0$ such for $\varepsilon$ small enough, we have
\begin{equation}\label{asymptotic uniqueness-1-proof-step-3-1}
    |w_{\varepsilon}(x)|\leq \frac{C}{|x|^{N-2}},\quad x\in\Omega_{\varepsilon}\setminus\omega_{0}.
\end{equation}
We consider the Kelvin transform of $w_{\varepsilon}$ 
\begin{equation}
    z_{\varepsilon}(x):=\frac{1}{|x|^{N-2}}w_{\varepsilon}\left(\frac{x}{|x|^{2}}\right),\quad x\in \Omega_{\varepsilon}^{*}:=\left\{x\in\R^{N}:\frac{x}{|x|^{2}}\in\Omega_{\varepsilon}\right\}.
\end{equation}
Note that $\|w_{\varepsilon}\|_{L^{\infty}(\Omega_{\varepsilon})}=1$, thus it is enough to prove that there exist $C>0$ such that for any $\varepsilon>0$ small enough
\begin{equation}
    \sup_{\Omega_{\varepsilon}^*\cap B(0,1)}z_{\varepsilon}(x)\leq C.
\end{equation}
By (\ref{asymptotic uniqueness-1-proof-step-1-3}) we get that
\begin{equation}
    \begin{cases}
        -\Delta z_{\varepsilon}(x)=a_{\varepsilon}(x)z_{\varepsilon}(x),&\quad \text{~in~} \Omega_{\varepsilon}^*,\\
        z_{\varepsilon}(x)=0,&\quad \text{~on~}\partial\Omega_{\varepsilon}^*,
    \end{cases}
\end{equation}
where
\begin{equation}
    a_{\varepsilon}(x)=\frac{1}{|x|^{4}}\left(c_{2^*,\varepsilon}\left(\frac{x}{|x|^{2}}\right)+\frac{\varepsilon}{\|u_{\varepsilon}\|_{L^{\infty}(\Omega)}^{2^*-q}}c_{q,\varepsilon}\left(\frac{x}{|x|^{2}}\right)\right).
\end{equation}
Next, we estimate $a_{\varepsilon}(x)$. By (3) and (6) in Theorem \ref{main theorem-1}, we get that for $\varepsilon>0$ small enough 
\begin{equation}
    \frac{1}{|x|^{4}}c_{2^*,\varepsilon}\left(\frac{x}{|x|^{2}}\right)\leq C\frac{1}{(N(N-2)+|x|^{2})^{2}},
\end{equation}
and
\begin{equation}
\frac{1}{|x|^{4}}c_{q,\varepsilon}\left(\frac{x}{|x|^{2}}\right)\leq C\frac{1}{|x|^{4-(N-2)(q-2)}}\frac{1}{(N(N-2)+|x|^{2})^{\frac{N-2}{2}(q-2)}},    
\end{equation}
then for $\alpha>\frac{N}{2}$, we have
\begin{equation}
    \int_{\Omega_{\varepsilon}^{*}}\frac{1}{|x|^{4\alpha}}c_{2^*,\varepsilon}^{\alpha}\left(\frac{x}{|x|^{2}}\right)\leq \int_{\R^{N}}\frac{1}{(N(N-2)+|x|^{2})^{2\alpha}}\leq C.
\end{equation}
Let $\gamma>0$ such that $\Omega\subset B(0,\gamma)$, then $\Omega_{\varepsilon}^*\subset \R^{N}\setminus B(0,\frac{1}{\gamma\|u_{\varepsilon}\|_{L^{\infty}(\Omega)}^{(2^*-2)/2}})$, then
    \begin{align}
      &\frac{\varepsilon}{\|u_{\varepsilon}\|_{L^{\infty}(\Omega)}^{2^*-q}}\left(\int_{\Omega_{\varepsilon}^{*}}\frac{1}{|x|^{4\alpha}}c_{q,\varepsilon}^{\alpha}\left(\frac{x}{|x|^{2}}\right)dx\right)^{\frac{1}{\alpha}}\\
      &\leq \frac{\varepsilon}{\|u_{\varepsilon}\|_{L^{\infty}(\Omega)}^{2^*-q}}\left(\int_{|x|\geq \frac{1}{\gamma\|u_{\varepsilon}\|_{L^{\infty}(\Omega)}^{(2^*-2)/2}}}\frac{1}{|x|^{4\alpha}}c_{q,\varepsilon}^{\alpha}\left(\frac{x}{|x|^{2}}\right)dx\right)^{\frac{1}{\alpha}}\\
      &\leq C\left(\int_{|x|\geq \frac{1}{\gamma}}\frac{1}{|x|^{4\alpha-(N-2)(q-2)\alpha}}\frac{1}{(N(N-2)\|u_{\varepsilon}\|_{L^{\infty}(\Omega)}^{2^*-2}+|x|^{2})^{\frac{N-2}{2}(q-2)\alpha}}dx\right)^{\frac{1}{\alpha}}\\
      &\quad\times \varepsilon\|u_{\varepsilon}\|_{L^{\infty}(\Omega)}^{\frac{8}{N-2}+q-2^*-\frac{2^*}{\alpha}}.
    \end{align}
By (6) in Theorem \ref{main theorem-1} we have
\begin{equation}
    \varepsilon\|u_{\varepsilon}\|_{L^{\infty}(\Omega)}^{\frac{8}{N-2}+q-2^*-\frac{2^*}{\alpha}}\leq C{\|u_{\varepsilon}\|_{L^{\infty}(\Omega)}^{\frac{8}{N-2}-\frac{2^*}{\alpha}-2}}.
\end{equation}
Hence if $N\geq 6$ and $\alpha>\frac{N}{2}$, while if $N=5$ and $\frac{5}{2}<\alpha<5$ , we have for $\varepsilon$ small enough
\begin{equation}
    a_{\varepsilon}(x)\in L^{\alpha}(\Omega_{\varepsilon}^*).
\end{equation}
for some $\alpha>\frac{N}{2}$, which together with Moser iteration (see Lemma \ref{moser iteration}), we can obtain (\ref{asymptotic uniqueness-1-proof-step-3-1}).

\textbf{Step} 4. Now we will show that $b=0$. Suppose for the contrary that $b\neq 0$. Indeed, we have
\begin{equation}\label{asymptotic uniqueness-1-proof-step-4-1}
\begin{aligned}
    -\Delta \left(\|u_{\varepsilon}\|_{L^{\infty}(\Omega)}^{2}\frac{u_{\varepsilon}(x)-v_{\varepsilon}(x)}{\|u_{\varepsilon}-v_{\varepsilon}\|_{L^{\infty}(\Omega)}}\right)&=\frac{\|u_{\varepsilon}\|_{L^{\infty}(\Omega)}^{2}}{\|u_{\varepsilon}-v_{\varepsilon}\|_{L^{\infty}(\Omega)}}[(u_{\varepsilon}^{2^*-1}-v_{\varepsilon}^{2^*-1})+\varepsilon (u_{\varepsilon}^{q-1}-v_{\varepsilon}^{q-1})]\\
    &=\frac{\|u_{\varepsilon}\|_{L^{\infty}(\Omega)}^{2}}{\|u_{\varepsilon}-v_{\varepsilon}\|_{L^{\infty}(\Omega)}}[d_{2^*,\varepsilon}(x)+\varepsilon d_{q,\varepsilon}(x)](u_{\varepsilon}-v_{\varepsilon}):=f_{\varepsilon}(x),
\end{aligned} 
\end{equation}
where
\begin{equation}
    d_{2^*,\varepsilon}(x)=(2^*-1)\int_{0}^{1}[t{u_{\varepsilon}(x)}+(1-t)v_{\varepsilon}(x)]^{2^*-2}dt,
\end{equation}
and
\begin{equation}
    d_{q,\varepsilon}(x)=(q-1)\int_{0}^{1}[t{u_{\varepsilon}(x)}+(1-t)v_{\varepsilon}(x)]^{q-2}dt.
\end{equation}
Next, we consider function $f_{\varepsilon}(x)$. By a change of variable, we have
\begin{equation}
\begin{aligned}
    \int_{\Omega}f_{\varepsilon}(x)dx&=\int_{\Omega}\frac{\|u_{\varepsilon}\|_{L^{\infty}(\Omega)}^{2}}{\|u_{\varepsilon}-v_{\varepsilon}\|_{L^{\infty}(\Omega)}}\left[d_{2^*,\varepsilon}(x)+\varepsilon d_{q,\varepsilon}(x)\right](u_{\varepsilon}-v_{\varepsilon})dx\\
    &=\int_{\Omega_{\varepsilon}}[c_{2^*,\varepsilon}(x)+\varepsilon \|u_{\varepsilon}\|_{L^{\infty}(\Omega)}^{q-2^*} c_{q,\varepsilon}(x)]w_{\varepsilon}(x).
\end{aligned}
\end{equation}
By (\ref{asymptotic uniqueness-1-proof-step-3-1}), for $\varepsilon$ small enough, we deduce that
\begin{equation}
\begin{aligned}
   \frac{\varepsilon }{\|u_{\varepsilon}\|_{L^{\infty}(\Omega)}^{2^*-q}}\int_{\Omega_{\varepsilon}}c_{q,\varepsilon}(x)w_{\varepsilon}(x)&\leq C\frac{\varepsilon }{\|u_{\varepsilon}\|_{L^{\infty}(\Omega)}^{2^*-q}} \int_{B(0,1)}w_{\varepsilon}dx+\int_{\Omega\setminus B(0,1)}\frac{1}{|x|^{N-2}}dx\\
   &\leq  C\frac{\varepsilon }{\|u_{\varepsilon}\|_{L^{\infty}(\Omega)}^{2^*-q}}+ C\frac{\varepsilon }{\|u_{\varepsilon}\|_{L^{\infty}(\Omega)}^{2^*-q}}(\|u_{\varepsilon}\|_{L^{\infty}(\Omega)}^{2^*-2}-1).
\end{aligned}   
\end{equation}
thus by (6) in Theorem \ref{main theorem-1} we have
\begin{equation}
    \lim_{\varepsilon\to 0}\frac{\varepsilon }{\|u_{\varepsilon}\|_{L^{\infty}(\Omega)}^{2^*-q}}\int_{\Omega_{\varepsilon}}c_{q,\varepsilon}(x)w_{\varepsilon}(x)=0.
\end{equation}
then by dominate convergence theorem, we deduce that
\begin{equation}
\begin{aligned}
\lim_{\varepsilon\to 0}\int_{\Omega_{\varepsilon}}c_{2^*,\varepsilon}(x)w_{\varepsilon}(x)&=\int_{\R^{N}}(2^*-1)U^{2^*-2}(x)w(x)dx=N(N+2)b\int_{\R^{N}}\frac{1-|x|^{2}}{(1+|x|^{2})^{2+\frac{N}{2}}}dx\\
&=-b(N-2)\omega_{N},
\end{aligned}
\end{equation}
Thus
\begin{equation}
    \lim_{\varepsilon\to 0}\int_{\Omega}f_{\varepsilon}(x)dx=-b(N-2)\omega_{N}.
\end{equation}
Next, we have
\begin{equation}
\begin{aligned}
     f_{\varepsilon}(x)&=\|u_{\varepsilon}\|_{L^{\infty}(\Omega)}^{2}[d_{2^*,\varepsilon}(x)+\varepsilon d_{q,\varepsilon}(x)]w_{\varepsilon}(x\|u_{\varepsilon}\|_{L^{\infty}(\Omega)}^{(2^*-2)/2})\\
     &\leq C\|u_{\varepsilon}\|^{2}_{L^{\infty}(\Omega)}\left(\frac{1}{\|u_{\varepsilon}\|_{L^{\infty}(\Omega)}^{4/(N-2)}|x|^{4}}+\frac{\varepsilon}{\|u_{\varepsilon}\|_{L^{\infty}(\Omega)}^{q-2}|x|^{(N-2)(q-2)}}\right)\frac{1}{\|u_{\varepsilon}\|_{L^{\infty}(\Omega)}^{2}|x|^{N-2}}\\
     &\to 0,
\end{aligned}
\end{equation}
uniformly in any neighbourhood $\omega$ of $\partial\Omega$ not containing the origin. Then by (\ref{asymptotic uniqueness-1-proof-step-4-1}) we have 
\begin{equation}\label{asymptotic uniqueness-1-proof-step-4-10}
    \|u_{\varepsilon}\|_{L^{\infty}(\Omega)}^{2}\frac{u_{\varepsilon}(x)-v_{\varepsilon}(x)}{\|u_{\varepsilon}-v_{\varepsilon}\|_{L^{\infty}(\Omega)}}\to -b(N-2)\omega_{N}G(x,0)\text{~~in~~} C^{1}(\omega),
\end{equation}
for any neighbourhood $\omega$ of $\partial\Omega$ not containing the origin. Now, we consider the Poho\v{z}aev identity (\ref{Pohazaev-p-varepsion}) for $u_{\varepsilon}$ and $v_{\varepsilon}$
\begin{equation}\label{asymptotic uniqueness-1-proof-step-4-11}
\frac{1}{2N}\int_{\partial\Omega}|\nabla u_{\varepsilon}|^{2}(x\cdot n)dS_{x}=(\frac{1}{q}-\frac{1}{2^*})\varepsilon\int_{\Omega}|u_{\varepsilon}|^{q}dx
\end{equation}
and
\begin{equation}\label{asymptotic uniqueness-1-proof-step-4-12}
\frac{1}{2N}\int_{\partial\Omega}|\nabla v_{\varepsilon}|^{2}(x\cdot n)dS_{x}=(\frac{1}{q}-\frac{1}{2^*})\varepsilon\int_{\Omega}|v_{\varepsilon}|^{q}dx.
\end{equation}
by (\ref{asymptotic uniqueness-1-proof-step-4-11}) and (\ref{asymptotic uniqueness-1-proof-step-4-12}), we have
\begin{equation}\label{asymptotic uniqueness-1-proof-step-4-13}
\begin{aligned}
 \frac{1}{2N}&\int_{\partial\Omega} \nabla \left(\frac{\|u_{\varepsilon}\|_{L^{\infty}(\Omega)}^{2}}{\|u_{\varepsilon}-v_{\varepsilon}\|_{L^{\infty}(\Omega)}}(u_{\varepsilon}-v_{\varepsilon})\right)\nabla\left(\|u_{\varepsilon}\|_{L^{\infty}(\Omega)}(u_{\varepsilon}+v_{\varepsilon})\right)  (x\cdot n)dS_{x}\\
 &=(\frac{1}{q}-\frac{1}{2^*})\varepsilon\int_{\Omega}\tilde{d}_{q,\varepsilon}(x)\|u_{\varepsilon}\|_{L^{\infty}(\Omega)}\|u_{\varepsilon}\|_{L^{\infty}(\Omega)}^{2}\frac{(u_{\varepsilon}-v_{\varepsilon})}{\|u_{\varepsilon}-v_{\varepsilon}\|_{L^{\infty}(\Omega)}}dx,
\end{aligned} 
\end{equation}
where
\begin{equation}
    \tilde{d}_{q,\varepsilon}(x)=q\int_{0}^{1}[t{u_{\varepsilon}(x)}+(1-t)v_{\varepsilon}(x)]^{q-1}dt.
\end{equation}
Then by (\ref{asymptotic uniqueness-1-proof-step-4-10}), (5) in Theorem \ref{main theorem-1} and Lemma \ref{Green function propertity 1}, the left-hand side (LHS) of (\ref{asymptotic uniqueness-1-proof-step-4-13}) tends to 
\begin{equation}\label{asymptotic uniqueness-1-proof-step-4-15}
   \text{LHS}\to -b(N-2)\omega_{N}^{2}\alpha_{N}^{2^*}\frac{1}{N^{2}}\int_{\partial\Omega}|\nabla G(x,0)|^{2}(x\cdot n)dS_{x}=-b(N-2)^{2}\frac{\omega_{N}^{2}\alpha_{N}^{2^*}}{N^{2}}R(0).
\end{equation}
On the other hand, by a change of variable, the right hand side (RHS) of (\ref{asymptotic uniqueness-1-proof-step-4-13}) becomes
\begin{equation}
    (\frac{1}{q}-\frac{1}{2^*})\varepsilon\|u_{\varepsilon}\|_{L^{\infty}(\Omega)}^{q+2-2^*}\int_{\Omega_{\varepsilon}}\tilde{c}_{q,\varepsilon}(x)w_{\varepsilon}(x)dx,
\end{equation}
where
\begin{equation}
    \tilde{c}_{q,\varepsilon}=q\int_{0}^{1}\left[t\frac{u_{\varepsilon}}{\|u_{\varepsilon}\|_{L^{\infty}(\Omega)}}\left(\frac{x}{\|u_{\varepsilon}\|_{L^{\infty}(\Omega)}^{(2^*-2)/2}}\right)+(1-t)\frac{v_{\varepsilon}}{\|u_{\varepsilon}\|_{L^{\infty}(\Omega)}}\left(\frac{x}{\|u_{\varepsilon}\|_{L^{\infty}(\Omega)}^{(2^*-2)/2}}\right)\right]^{q-1}dt.
\end{equation}
Then by (2) and (6) in Theorem \ref{main theorem-1} we have
\begin{equation}\label{asymptotic uniqueness-1-proof-step-4-18}
\begin{aligned}
 \text{RHS}&\to (\frac{1}{q}-\frac{1}{2^*})q\alpha_{N,q}R(0)\int_{\R^{N}}w(x)U^{q-1}(x)dx\\
 &=(\frac{1}{q}-\frac{1}{2^*})q\alpha_{N,q}R(0)N(N-2)b\int_{\R^{N}}\frac{1-|x|^{2}}{(1+|x|^{2})^{\frac{N-2}{2}(q-1)+\frac{N}{2}}}dx \\
 &=(\frac{1}{q}-\frac{1}{2^*})q\alpha_{N,q}R(0)N(N-2)\omega_{N}b\int_{0}^{1}\frac{1-t^{2}}{(1+t^{2})^{\frac{N-2}{2}(q-1)+\frac{N}{2}}}(t^{N-1}-t^{(N-2)(q-1)-3})dt\\
&=(\frac{1}{q}-\frac{1}{2^*})q\alpha_{N,q}R(0)N(N-2)\omega_{N}b(\frac{N-2}{2}(q-1)-\frac{N+2}{2})\frac{\Gamma(\frac{N}{2})\Gamma(\frac{N-2}{2}(q-1)-1)}{2\Gamma(\frac{N-2}{2}(q-1)+\frac{N}{2})}\\
&=(\frac{1}{q}-\frac{1}{2^*})q R(0)N(N-2)\omega_{N}b(\frac{N-2}{2}(q-1)-\frac{N+2}{2})\frac{2q\alpha_{N}^{\frac{2N-8}{N-2}}\omega_{N}(N-2)^{3}}{2N-(N-2)q}\\
&\quad \times\frac{\Gamma(\frac{N-2}{2}q)}{\Gamma(\frac{N}{2})\Gamma(\frac{N-2}{2}q-\frac{N}{2})}\frac{\Gamma(\frac{N}{2})\Gamma(\frac{N-2}{2}q-\frac{N}{2})}{(N-2)q\Gamma(\frac{N-2}{2}q)}\\
&=(\frac{1}{q}-\frac{1}{2^*})q R(0)N(N-2)\omega_{N}b(\frac{N-2}{2}(q-1)-\frac{N+2}{2})\frac{2q\alpha_{N}^{\frac{2N-8}{N-2}}\omega_{N}(N-2)^{3}}{2N-(N-2)q}\frac{1}{(N-2)q}\\
&=-b\frac{2N-(N-2)q}{2}(N-2)\frac{\omega_{N}^{2}\alpha_{N}^{2^*}}{N^{2}}R(0).
\end{aligned}
\end{equation}
Hence by (\ref{asymptotic uniqueness-1-proof-step-4-15}) and (\ref{asymptotic uniqueness-1-proof-step-4-18}), we conclude that
\begin{equation}
    q=\frac{4}{N-2},
\end{equation}
then we get a contradiction with the choice of $q$, which implies $b=0$, and so $w=0$. 

Finally, Let $x_{\varepsilon}\in\R^{N}$ such that 
\begin{equation}
    \|w_{\varepsilon}\|_{L^{\infty}(\Omega_{\varepsilon})}=w_{\varepsilon}(x_{\varepsilon})=1.
\end{equation}
since $w_{\varepsilon}(x)\to 0$, as $\varepsilon\to 0$, uniformly on compact set of $\R^{N}$, thus up to a subsequence, $x_{\varepsilon}\to 0$ as $\varepsilon\to 0$, which make a contraction with (\ref{asymptotic uniqueness-1-proof-step-3-1}), and we complete the proof.
\end{proof}

\subsection{Non-convex domain}
In this subsection, we consider the asymptotic uniqueness of the least energy solution under the assumption that the Robin function is nondegenerate. The proof mainly depends on the decomposition on the least energy solutions and local Poho\v{z}aev identity. The result in this subsection states as follows.
\begin{Assum}\label{Assum-2}
 $x_{0}$ is a nondegenerate critical point of Robin function $R(x)$.   
\end{Assum}
\begin{Thm}\label{local uniqueness-Robin function}
Let $\Omega$ be a smooth bounded star-shape domain of $\R^{N}$, $N\geq 4$, $q\in (2,2^*)$ and $q\geq \frac{N+2}{N-2}$. Suppose that $u_{\varepsilon}$ and $v_{\varepsilon}$ are two least energy solutions of (\ref{p-varepsion}) concentrate on the same point $x_{0}\in\Omega$
and $x_{0}$ satisfies Assumption $\ref{Assum-2}$, then there exists $\varepsilon_{0}>0$ such that for any $\varepsilon<\varepsilon_{0}$
\begin{equation}
    u_{\varepsilon}\equiv v_{\varepsilon}\quad \text{~in~}\Omega.
\end{equation}
\end{Thm}

In the following, we will prove this theorem by several Claims. First, from Lemma \ref{decomposition of u-varepsilon}, we know that $u_{\varepsilon}$ and $v_{\varepsilon}$ can be written as
\begin{equation}\label{local uniqueness proof 1}
    u_{\varepsilon}(x)=\alpha_{1,\varepsilon}PU_{\lambda_{1,\varepsilon},x_{1,\varepsilon}}+w_{1,\varepsilon}(x) \text{~~and~~}v_{\varepsilon}(x)=\alpha_{2,\varepsilon}PU_{\lambda_{2,\varepsilon},x_{2,\varepsilon}}+w_{2,\varepsilon}(x),
\end{equation}
with $w_{i,\varepsilon}\in E_{\lambda_{i,\varepsilon},x_{i,\varepsilon}}$, for $i=1,2$ and
\begin{equation}\label{local uniqueness proof 2}
  \alpha_{i,\varepsilon}\to\alpha_{N},\quad  x_{i,\varepsilon}\to x_{0},\quad \lambda_{i,\varepsilon}\to\infty, \quad\|w_{i,\varepsilon}\|_{H^{1}_{0}(\Omega)}\to 0. 
\end{equation}
Moreover, we have
\begin{equation}\label{local uniqueness proof 3}
    \varepsilon=O(\lambda_{i,\varepsilon}^{-\frac{N-2}{2}(q+2-2^*)}),
\end{equation}
and
\begin{equation}\label{local uniqueness proof 4}
    u_{\varepsilon}(x)\leq CU_{\lambda_{1,\varepsilon},x_{1,\varepsilon}}(x),\quad v_{\varepsilon}(x)\leq CU_{\lambda_{2,\varepsilon},x_{2,\varepsilon}}(x).
\end{equation}

The first claim give the refined estimate on $u_{\varepsilon}(x)$ and the results for $v_{\varepsilon}$ is similar.
\begin{Claim} For any $x\in\Omega\setminus B(x_{1,\varepsilon},\delta)$, we have
\begin{equation}\label{local uniqueness proof step 1-1}
u_{\varepsilon}(x)=A_{1,\varepsilon}\lambda_{1,\varepsilon}^{-\frac{N-2}{2}}G(x,x_{1,\varepsilon})+O\left(\frac{\ln \lambda_{1,\varepsilon}}{\lambda_{1,\varepsilon}^{\frac{N+2}{2}}}\right),
\end{equation}
and
\begin{equation}\label{local uniqueness proof step 1-2}
    \nabla u_{\varepsilon}(x)=A_{1,\varepsilon}\lambda_{1,\varepsilon}^{-\frac{N-2}{2}}\nabla G(x,x_{1,\varepsilon})+O\left(\frac{\ln \lambda_{1,\varepsilon}}{\lambda_{1,\varepsilon}^{\frac{N+2}{2}}}\right),
\end{equation}
where $A_{1,\varepsilon}=\int_{B(0,\frac{1}{2}\delta\lambda_{1,\varepsilon})}\tilde{u}_{\varepsilon}^{2^*-1}dx$ and  $\tilde{u}_{\varepsilon}(x)=\lambda_{1,\varepsilon}^{-\frac{N-2}{2}}u_{\varepsilon}(\lambda_{1,\varepsilon}^{-1}x+x_{1,\varepsilon})$.  
\end{Claim}
\begin{proof}
By Green representation formula (\ref{Green's representation formula}), we have
\begin{equation}\label{local uniqueness proof step 1-3}
    u_{\varepsilon}(x)=\int_{\Omega}G(x,y)(u_{\varepsilon}^{2^*-1}(y)+\varepsilon u_{\varepsilon}^{q-1}(y))dy,
\end{equation}
and it's easy to see that
\begin{equation}\label{local uniqueness proof step 1-4}
    \begin{aligned}
        \int_{\Omega}G(x,y)u_{\varepsilon}^{q-1}(y)dy\leq \int_{\Omega\setminus B(x_{1,\varepsilon},\frac{1}{2}\delta)}G(x,y)u_{\varepsilon}^{q-1}(y)dy +\int_{B(x_{1,\varepsilon},\frac{1}{2}\delta)}G(x,y)u_{\varepsilon}^{q-1}(y)dy.
    \end{aligned}
    \end{equation}
    
First, by (\ref{local uniqueness proof 4}) and Lemma \ref{estimate of Green function}, the first term in the right-hand side of the above equation can be estimate as
\begin{equation}\label{local uniqueness proof step 1-5}
\begin{aligned}
 \int_{\Omega\setminus B(x_{1,\varepsilon},\frac{1}{2}\delta)}G(x,y)u_{\varepsilon}^{q-1}(y)dy&\leq C \int_{\Omega\setminus B(x_{1,\varepsilon},\frac{1}{2}\delta)}G(x,y)U_{\lambda_{1,\varepsilon},x_{1,\varepsilon}}^{q-1}(y)dy\\
 &\leq C\int_{\Omega\setminus B(x_{1,\varepsilon},\frac{1}{2}\delta)}\frac{1}{|x-y|^{N-2}}\frac{1}{\lambda_{1,\varepsilon}^{\frac{N-2}{2}(q-1)}|y-x_{1,\varepsilon}|^{(N-2)(q-1)}}dy\\
  &\leq \frac{C}{\lambda_{1,\varepsilon}^{\frac{N-2}{2}(q-1)}}.  
\end{aligned}
\end{equation}   

Next, we consider the second term in the right-hand side of \eqref{local uniqueness proof step 1-4}. When $(q-1)<\frac{N}{N-2}$, we can obtain that
\begin{equation}\label{local uniqueness proof step 1-6}
\begin{aligned}
\int_{B(x_{1,\varepsilon},\frac{1}{2}\delta)}G(x,y)u_{\varepsilon}^{q-1}(y)dy&\leq C \int_{B(x_{1,\varepsilon},\frac{1}{2}\delta)}G(x,y)U_{\lambda_{1,\varepsilon},x_{1\varepsilon}}^{q-1}(y)dy\\
&\leq\int_{B(x_{1,\varepsilon},\frac{1}{2}\delta)}\frac{1}{|x-y|^{N-2}}\frac{1}{\lambda_{1,\varepsilon}^{\frac{N-2}{2}(q-1)}|y-x_{1,\varepsilon}|^{(N-2)(q-1)}}dy \\
&\leq  \frac{C}{\lambda_{1,\varepsilon}^{\frac{N-2}{2}(q-1)}}.  
\end{aligned}  
\end{equation}
To consider the other case, we will use the following Taylor expansion
\begin{equation}\label{local uniqueness proof step 1-7}
    G(x,y)=G(x,x_{1,\varepsilon})+\langle\nabla G(x,x_{1,\varepsilon}),y-x_{1,\varepsilon}\rangle+O(|y-x_{1,\varepsilon}|^{2}),
\end{equation}
and we have
\begin{equation}\label{local uniqueness proof step 1-8}
\begin{aligned}
&\int_{B(x_{1,\varepsilon},\frac{1}{2}\delta)}G(x,y)u^{q-1}_{\varepsilon}(y)dy \\
&=G(x,x_{1,\varepsilon})\int_{B(x_{1,\varepsilon},\frac{1}{2}\delta)}u^{q-1}_{\varepsilon}(y)dy+\int_{B(x_{1,\varepsilon},\frac{1}{2}\delta)}\langle\nabla G(x,x_{1,\varepsilon}),y-x_{1,\varepsilon}\rangle u^{q-1}_{\varepsilon}(y)dy\\
&\quad+ O\left(\int_{B(x_{1,\varepsilon},\frac{1}{2}\delta)}|y-x_{1,\varepsilon}|^{2}u^{q-1}_{\varepsilon}(y)dy\right).
\end{aligned}
\end{equation}
The first term in the right-hand side of the above equation can be estimated as
\begin{equation}\label{local uniqueness proof step 1-9}
\begin{aligned}
  \int_{B(x_{1,\varepsilon},\frac{1}{2}\delta)}u^{q-1}_{\varepsilon}(y)dy&=\lambda_{1,\varepsilon}^{-N+\frac{N-2}{2}(q-1)}\int_{B(0,\frac{1}{2}\delta\lambda_{1,\varepsilon})}\tilde{u}_{\varepsilon}^{q-1}(x)dx\\
  &\leq C \lambda_{1,\varepsilon}^{-N+\frac{N-2}{2}(q-1)}\int_{B(0,\frac{1}{2}\delta\lambda_{1,\varepsilon})}\frac{1}{(1+|x|^{2})^{\frac{N-2}{2}(q-1)}}dx\\
  &=\begin{cases}
      \frac{C}{\lambda_{1,\varepsilon}^{N-\frac{N-2}{2}(q-1)}}, &\text{~if~} (N-2)(q-1)>N,\\
      \frac{C\ln \lambda_{1,\varepsilon}}{\lambda_{1,\varepsilon}^{N-\frac{N-2}{2}(q-1)}}, &\text{~if~} (N-2)(q-1)=N,\\
      \frac{C}{\lambda_{1,\varepsilon}^{\frac{N-2}{2}(q-1)}}, &\text{~if~} (N-2)(q-1)<N,
  \end{cases}
\end{aligned}
\end{equation}
and the third term in \eqref{local uniqueness proof step 1-8} bound by
\begin{equation}\label{local uniqueness proof step 1-10}
    \begin{aligned}
      \int_{B(x_{1,\varepsilon},\frac{1}{2}\delta)}|y-x_{1,\varepsilon}|^{2}u^{q-1}_{\varepsilon}(y)dy&\leq C \int_{B(x_{1,\varepsilon},\frac{1}{2}\delta)}|y-x_{1,\varepsilon}|^{2}U_{\lambda_{1,\varepsilon},x_{1,\varepsilon}}^{q-1}(y)dy\\
      &=\frac{C}{\lambda_{1,\varepsilon}^{(2+N)-\frac{N-2}{2}(q-1)}}\int_{B(0,\frac{1}{2}\delta\lambda_{1,\varepsilon})}|y|^{2}U_{1,0}^{q-1}(y)dy\\
      &\leq \frac{C}{\lambda_{1,\varepsilon}^{\frac{N-2}{2}(q-1)}}.
    \end{aligned}
\end{equation}
Moreover, by symmetry, Lemma \ref{estimate of Green function} and Lemma \ref{perturbation estimate of w-varepsilon}, we have 
\begin{equation}\label{local uniqueness proof step 1-11}
    \begin{aligned}
       &\int_{B(x_{1,\varepsilon},\frac{1}{2}\delta)}\langle\nabla G(x,x_{1,\varepsilon}),y-x_{1,\varepsilon}\rangle u^{q-1}_{\varepsilon}(y)dy\\
       &\leq C  \int_{B(x_{1,\varepsilon},\frac{1}{2}\delta)}\langle\nabla G(x,x_{1,\varepsilon}),y-x_{1,\varepsilon}\rangle U^{q-1}_{\lambda_{1,\varepsilon},x_{1,\varepsilon}}(y)dy\\
       &\quad+O\left(\int_{B(x_{1,\varepsilon},\frac{1}{2}\delta)}|y-x_{1,\varepsilon}| U^{q-2}_{\lambda_{1,\varepsilon},x_{1,\varepsilon}}(y)|w_{1,\varepsilon}(y)|dy\right)\\
       &=O\left(\int_{B(x_{1,\varepsilon},\frac{1}{2}\delta)}|y-x_{1,\varepsilon}| U^{q-2}_{\lambda_{1,\varepsilon},x_{1,\varepsilon}}(y)|w_{1,\varepsilon}(y)|dy\right)\\
       &=O\left(\int_{B(x_{1,\varepsilon},\frac{1}{2}\delta)}\frac{1}{\lambda_{\varepsilon}^{\frac{N-2}{2}(q-2)}|y-x_{1,\varepsilon}|^{(N-2)(q-2)-1}}|w_{1,\varepsilon}|dy\right)\\
       &=O\left(\frac{1}{\lambda_{\varepsilon}^{\frac{N-2}{2}(q-2)}}\right)\|w_{1,\varepsilon}\|_{H^{1}_{0}(\Omega)}=O\left(\frac{1}{\lambda_{1,\varepsilon}^{\frac{N-2}{2}(q-1)}}\right).
    \end{aligned}
\end{equation}
Finally, by (\ref{local uniqueness proof step 1-4})-(\ref{local uniqueness proof step 1-11}), we conclude that
\begin{equation}\label{local uniqueness proof step 1-12}
    \int_{\Omega}G(x,y)u^{q-1}_{\lambda_{1,\varepsilon}}(y)dy\leq \begin{cases}
      \frac{C}{\lambda_{1,\varepsilon}^{N-\frac{N-2}{2}(q-1)}}, &\text{~if~} (N-2)(q-1)>N,\\
      \frac{C\ln \lambda_{1,\varepsilon}}{\lambda_{1,\varepsilon}^{N/2}}, &\text{~if~} (N-2)(q-1)=N,\\
      \frac{C}{\lambda_{1,\varepsilon}^{\frac{N-2}{2}(q-1)}}, &\text{~if~} (N-2)(q-1)<N.\\
      \end{cases}
\end{equation}
Similarly, we have
\begin{equation}\label{local uniqueness proof step 1-13}
     \int_{\Omega}G(x,y)u^{2^*-1}_{\lambda_{1,\varepsilon}}(y)dy=A_{1,\varepsilon}\lambda_{1,\varepsilon}^{-\frac{N-2}{2}}G(x,x_{1,\varepsilon})+O\left(\frac{\ln \lambda_{1,\varepsilon}}{\lambda_{1,\varepsilon}^{\frac{N+2}{2}}}\right). 
\end{equation}
Then by (\ref{local uniqueness proof step 1-3}), (\ref{local uniqueness proof 3}), (\ref{local uniqueness proof step 1-12}) and (\ref{local uniqueness proof step 1-13}), we obtain (\ref{local uniqueness proof step 1-1}).
Similarly, from Lemma \ref{estimate of Green function} and 
\begin{equation}
    \nabla u_{\varepsilon}(x)=\int_{\Omega}\nabla_{x}G(x,y)(u_{\varepsilon}^{2^*-1}(y)+\varepsilon u_{\varepsilon}^{q-1}(y))dy,
\end{equation}
we can prove (\ref{local uniqueness proof step 1-2}).
\end{proof}

Next, we give a exact estimate on location and height of bubble.
\begin{Claim} There exist a constant $c_{0}$ such that
\begin{equation}\label{local uniqueness proof step 2-1}
    \nabla R(x_{1,\varepsilon})=O\left(\frac{\ln\lambda_{1,\varepsilon}}{\lambda_{1,\varepsilon}^{2}}\right),\quad |x_{1,\varepsilon}-x_{0}|=O\left(\frac{\ln\lambda_{1,\varepsilon}}{\lambda_{1,\varepsilon}^{2}}\right),\quad \varepsilon=\lambda_{1,\varepsilon}^{-\frac{N-2}{2}(q+2-2^*)}(c_{0}+o(1)).
\end{equation}    
\end{Claim}
\begin{proof}
We recall the local Poho\v{z}aev identity for $u_{\varepsilon}$ (see Lemma \ref{Local Pohozaev identiy})
\begin{equation}
    \int_{\partial B(x_{1,\varepsilon},\delta)}\frac{\partial u_{\varepsilon}}{\partial x_{i}}\frac{\partial u_{\varepsilon}}{\partial n}-\frac{1}{2}|\nabla u_{\varepsilon}|^{2}n_{i}+\left(\frac{1}{2^*}|u_{\varepsilon}|^{2^*}+\frac{\varepsilon}{q}|u_{\varepsilon}|^{q}\right)n_{i}dS_{x}=0,
\end{equation}
for $i=1,\cdots,N$. Then by (\ref{local uniqueness proof step 1-1}), (\ref{local uniqueness proof step 1-2}) in Claim 1 and Lemma \ref{Pohozaev identity of Green functon }, we can prove that
\begin{equation}
    \nabla R(x_{1,\varepsilon})=O\left(\frac{\ln\lambda_{1,\varepsilon}}{\lambda_{1,\varepsilon}^{2}}\right).
\end{equation}
Moreover, by $x_{0}$ is a nondegenerate critical point of $R(x)$, we have
\begin{equation}
    |x_{1,\varepsilon}-x_{0}|=O\left(\frac{\ln\lambda_{1,\varepsilon}}{\lambda_{1,\varepsilon}^{2}}\right).
\end{equation}

Next, by the proof of Claim 1, we have that for $x\in\Omega\setminus B(x_{1,\varepsilon},\delta)$
\begin{equation}
u_{\varepsilon}(x)=A_{1,\varepsilon}\lambda_{1,\varepsilon}^{-\frac{N-2}{2}}G(x,x_{1,\varepsilon})+O\left(\frac{1}{\lambda_{1,\varepsilon}^{\frac{N}{2}}}+\frac{\varepsilon}{\lambda_{1,\varepsilon}^{\frac{N-2}{2}}}\right),
\end{equation}
and
\begin{equation}
    \nabla u_{\varepsilon}(x)=A_{1,\varepsilon}\lambda_{1,\varepsilon}^{-\frac{N-2}{2}}\nabla G(x,x_{1,\varepsilon})+O\left(\frac{1}{\lambda_{1,\varepsilon}^{\frac{N}{2}}}+\frac{\varepsilon}{\lambda_{1,\varepsilon}^{\frac{N-2}{2}}}\right).
\end{equation}
Then by Lemma \ref{Pohozaev identity of Green functon } and the local Poho\v{z}aev identity (see Lemma \ref{Pohozaev identity}),
\begin{equation}
  \begin{aligned}
      &\int_{\partial B(x_{1,\varepsilon},\delta)}\left((\nabla u_{\varepsilon}\cdot n)((x-x_{1,\varepsilon})\cdot\nabla u_{\varepsilon})-\frac{|\nabla u_{\varepsilon}|^{2}}{2}(x-x_{1,\varepsilon})\cdot n\right)+\frac{N-2}{2}\int_{\partial B(x_{1,\varepsilon},\delta)}u_{\varepsilon}\frac{\partial u_{\varepsilon}}{\partial n}\\
      &\quad+\int_{\partial B(x_{1,\varepsilon},\delta)} \left(\frac{1}{2^*}u_{\varepsilon}^{2^*}+\frac{\varepsilon}{q}u_{\varepsilon}^{q}\right)\left((x-x_{1,\varepsilon})\cdot n\right)\\
      &=\varepsilon\left(\frac{N}{q}-\frac{N-2}{2}\right)\int_{ B(x_{1,\varepsilon},\delta)}u_{\varepsilon}^{q}, 
  \end{aligned}   
 \end{equation}
we have
\begin{equation}
    \frac{A_{1,\varepsilon}^{2}(N-2)R(x_{1,\varepsilon})}{2\lambda_{1,\varepsilon}^{N-2}}=\varepsilon\left(\frac{N}{q}-\frac{N-2}{2}\right)\int_{B(x_{1,\varepsilon},\delta)}u_{\varepsilon}^{q}+O\left(\frac{\varepsilon}{\lambda_{1,\varepsilon}^{N-2}}+\frac{1}{\lambda_{1,\varepsilon}^{N-1}}\right).
\end{equation}
On the other hand, it's easy to see that
\begin{equation}
    A_{1,\varepsilon}=\alpha_{N}^{2^*-1}\int_{\R^{N}}U_{1,0}^{2^*-1}+o,
\end{equation}
and
\begin{equation}
\int_{B(x_{1,\varepsilon},\delta)}u_{\varepsilon}^{q}=\frac{1}{\lambda_{1,\varepsilon}^{N-\frac{N-2}{2}q}}\left(\alpha_{N}^{q}\int_{\R^{N}}U_{1,0}^{q}+o(1)\right).
\end{equation}
Thus we obtain that
\begin{equation}\label{local uniqueness proof step 2-11}
\begin{aligned}
\frac{(N-2)R(x_{1,\varepsilon})\alpha_{N}^{22^*-2}(\int_{\R^{N}}U_{1,0}^{2^*-1}+o(1))^{2}}{2\lambda_{1,\varepsilon}^{N-2}}&=\frac{\varepsilon}{\lambda_{1,\varepsilon}^{N-\frac{N-2}{2}q}}(\frac{N}{q}-\frac{N-2}{2})\alpha_{N}^{q}\left(\int_{\R^{N}}U_{1,0}^{q}+o(1)\right)\\
&\quad+O\left(\frac{\varepsilon}{\lambda_{1,\varepsilon}^{N-2}}+\frac{1}{\lambda_{1,\varepsilon}^{N-1}}\right),    
\end{aligned}  
\end{equation}
which gives that
\begin{equation}
 \varepsilon=\lambda_{1,\varepsilon}^{-\frac{N-2}{2}(q+2-2^*)}(c_{0}+o(1)).   
\end{equation}
where $c_{0}$ is a constant only depend on $N,q$ and $R(x_{0})$.
\end{proof}
Next, we assume that $u_{\varepsilon}\not\equiv v_{\varepsilon}$ and 
set
\begin{equation}\label{local uniqueness proof step 3-1}
    z_{\varepsilon}(x)=\frac{u_{\varepsilon}(x)-v_{\varepsilon}(x)}{\|u_{\varepsilon}-v_{\varepsilon}\|_{L^{\infty}(\Omega)}},
\end{equation}
then $\|z_{\varepsilon}(x)\|_{L^{\infty}(\Omega)}=1$ and $z_{\varepsilon}(x)$ satisfy that
\begin{equation}\label{local uniqueness proof step 3-2}
    -\Delta z_{\varepsilon}(x)=k_{\varepsilon}(x)z_{\varepsilon}(x),
\end{equation}
where
\begin{equation}
\begin{aligned}
   k_{\varepsilon}(x)&=(2^*-1)\int_{0}^{1}(tu_{\varepsilon}(x)+(1-t)v_{\varepsilon}(x))^{2^*-2}dt+\varepsilon(q-1) \int_{0}^{1}(tu_{\varepsilon}(x)+(1-t)v_{\varepsilon}(x))^{q-2}dt \\
   &:=k_{2^*,\varepsilon}(x)+\varepsilon k_{q,\varepsilon}(x).
\end{aligned}
\end{equation}
On the other hand, by Lemma \ref{Local Pohozaev identiy}, we have the following Poho\v{z}aev identity for $z_{\varepsilon}$
\begin{equation}\label{local uniqueness proof step 3-pohozaev identity-1}
\begin{aligned}
    &-\int_{\partial B(x_{1,\varepsilon},\delta)}\frac{\partial z_{\varepsilon}}{\partial n}\frac{\partial u_{\varepsilon}}{\partial x_{i}}-\int_{\partial B(x_{1,\varepsilon},\delta)}\frac{\partial z_{\varepsilon}}{\partial x_{i}}\frac{\partial v_{\varepsilon}}{\partial n}\\
    &+\frac{1}{2}\int_{\partial B(x_{1,\varepsilon},\delta)}\langle\nabla(u_{\varepsilon}+v_{\varepsilon}),\nabla z_{\varepsilon}\rangle n_{i}=\int_{\partial B(x_{1,\varepsilon},\delta)}C_{\varepsilon}(x)z_{\varepsilon}n_{i},
\end{aligned}
\end{equation}
and
\begin{equation}\label{local uniqueness proof step 3-pohozaev identity-2}
    \begin{aligned}
        &\int_{\partial B(x_{1,\varepsilon},\delta)}\frac{\partial z_{\varepsilon}}{\partial n}\langle x-x_{1,\varepsilon},\nabla u_{\varepsilon}\rangle+\int_{\partial B(x_{1,\varepsilon},\delta)}\frac{\partial v_{\varepsilon}}{\partial n}\langle x-x_{1,\varepsilon},\nabla z_{\varepsilon}\rangle\\
        &-\frac{1}{2}\int_{\partial B(x_{1,\varepsilon},\delta)}\langle\nabla(u_{\varepsilon}+v_{\varepsilon},\nabla z_{\varepsilon})\rangle\langle x-x_{1,\varepsilon},n\rangle\\
        &+\frac{N-2}{2}\int_{\partial B(x_{1,\varepsilon},\delta)}\frac{\partial z_{\varepsilon}}{\partial n}u_{\varepsilon}+\frac{N-2}{2}\int_{\partial B(x_{1,\varepsilon},\delta)}\frac{\partial v_{\varepsilon}}{\partial n}z_{\varepsilon}\\
        &=-\int_{\partial B(x_{1,\varepsilon},\delta)} C_{\varepsilon}(x)z_{\varepsilon}(x)\langle x-x_{1,\varepsilon},n\rangle+\varepsilon\int_{B(x_{1,\varepsilon},\delta)}D_{\varepsilon}(x)z_{\varepsilon}(x),
    \end{aligned}
\end{equation}
where
\begin{equation}
    C_{\varepsilon}(x)=\int_{0}^{1}(t u_{\varepsilon}(x)+(1-t)v_{\varepsilon}(x))^{2^*-1}dt+\varepsilon \int_{0}^{1}(t u_{\varepsilon}(x)+(1-t)v_{\varepsilon}(x))^{q-1}dt,
\end{equation}
and
\begin{equation}
    D_{\varepsilon}(x)=(\frac{N}{q}-\frac{N-2}{2})q\int_{0}^{1}(t u_{\varepsilon}(x)+(1-t)v_{\varepsilon}(x))^{q-1}dt.
\end{equation}
\begin{Claim}we have the following estimate   on $z_{\varepsilon}(x)$.
  \begin{equation}\label{local uniqueness proof step 3-4}
    |z_{\varepsilon}(x)|\leq \frac{C}{(1+\lambda_{1,\varepsilon}^{2}|x-x_{1,\varepsilon}|^{2})^{\frac{N-2}{2}}}.
\end{equation}  
\end{Claim}
\begin{proof}
Recall the Green representation formula (\ref{Green's representation formula}) 
\begin{equation}\label{local uniqueness proof step 3-5}
    z_{\varepsilon}(x)=\int_{\Omega}G(x,y)k_{\varepsilon}(y)z_{\varepsilon}(y)dy.
\end{equation}
Note that, by (\ref{local uniqueness proof step 2-1}), we have
\begin{equation}
    |x_{1,\varepsilon}-x_{2,\varepsilon}|=O\left(\frac{\ln\lambda_{1,\varepsilon}}{\lambda_{1,\varepsilon}^{2}}\right),\quad \frac{\lambda_{1,\varepsilon}}{\lambda_{2,\varepsilon}}\to 1,
\end{equation}
Thus
\begin{equation}
    |U_{\lambda_{1,\varepsilon},x_{1,\varepsilon}}-U_{\lambda_{2,\varepsilon},x_{2,\varepsilon}}|\leq CU_{\lambda_{1,\varepsilon},x_{1,\varepsilon}}(\lambda_{1,\varepsilon}|x_{1,\varepsilon}-x_{2,\varepsilon}|+\lambda_{1,\varepsilon}^{-1}|\lambda_{1,\varepsilon}-\lambda_{2,\varepsilon}|)=o(1)U_{\lambda_{1,\varepsilon},x_{1,\varepsilon}}.
\end{equation}
Then there exist constant $c_{1}$ and $c_{2}$ independent on $\varepsilon$ such that for $\varepsilon$ small enough
\begin{equation}
    U_{\lambda_{1,\varepsilon},x_{1,\varepsilon}}\leq c_{1}U_{\lambda_{1,\varepsilon},x_{1,\varepsilon}}\leq c_{2} U_{\lambda_{1,\varepsilon},x_{1,\varepsilon}}.
\end{equation}
Moreover, by $u_{\varepsilon}\leq CU_{\lambda_{1,\varepsilon},x_{1,\varepsilon}}$ and $v_{\varepsilon}\leq CU_{\lambda_{2,\varepsilon},x_{2,\varepsilon}}$, we obtain that
\begin{equation}
    k_{2^*,\varepsilon}(x)\leq C\frac{\lambda_{1,\varepsilon}^{2}}{(1+\lambda_{1,\varepsilon}^{2}|x-x_{1,\varepsilon}|^{2})^{2}},\quad k_{q,\varepsilon}(x)\leq C\frac{\lambda_{1,\varepsilon}^{\frac{(N-2)(q-2)}{2}}}{(1+\lambda_{1,\varepsilon}^{2}|x-x_{1,\varepsilon}|^{2})^{\frac{(N-2)(q-2)}{2}}}.
\end{equation}
Now, we estimate the right hand side of (\ref{local uniqueness proof step 3-5}), by lemma \ref{Green function propertity 1} and Lemma \ref{useful estimate}, we have
\begin{equation}
\begin{aligned}
 \int_{\Omega}G(x,y)k_{q,\varepsilon}(y)&\leq C\int_{\Omega}\frac{1}{|x-y|^{N-2}} \frac{\lambda_{1,\varepsilon}^{\frac{(N-2)(q-2)}{2}}}{(1+\lambda_{1,\varepsilon}^{2}|y-x_{1,\varepsilon}|^{2})^{\frac{(N-2)(q-2)}{2}}} dy\\
 &\leq C\lambda_{1,\varepsilon}^{\frac{(N-2)(q-2)}{2}-2}\int_{\lambda_{1,\varepsilon}(\Omega-x_{1,\varepsilon})}\frac{1}{|\lambda_{1,\varepsilon}(x-x_{1,\varepsilon})-y|^{N-2}}\frac{1}{(1+|y|^{2})^{\frac{(N-2)(q-2)}{2}}}dy\\
 &\leq  C\lambda_{1,\varepsilon}^{\frac{(N-2)(q-2)}{2}}\int_{\lambda_{1,\varepsilon}(\Omega-x_{1,\varepsilon})}\frac{1}{|\lambda_{1,\varepsilon}(x-x_{1,\varepsilon})-y|^{N-2}}\frac{1}{(1+|y|^{2})^{\frac{(N-2)(q-2)+2}{2}}}dy\\
 &\leq C\frac{\lambda_{1,\varepsilon}^{\frac{(N-2)(q-2)}{2}}}{(1+\lambda_{1,\varepsilon}|x-x_{1,\varepsilon}|)^{(N-2)(q-2)}}.
\end{aligned}
\end{equation}
Moreover, by (\ref{local uniqueness proof step 2-1}), we obtain that
\begin{equation}\label{local uniqueness proof step 3-11}
    \varepsilon\int_{\Omega}G(x,y)k_{q,\varepsilon}(y)\leq C\frac{1}{\lambda_{1,\varepsilon}^{N-4}(1+\lambda_{1,\varepsilon}|x-x_{1,\varepsilon}|)^{(N-2)(q-2)}}.
\end{equation}
Similarly, we have
   \begin{equation}\label{local uniqueness proof step 3-12}
    \int_{\Omega}G(x,y)k_{2^*,\varepsilon}(x)\leq C\frac{1}{(1+\lambda_{1,\varepsilon}|x-x_{1,\varepsilon}|)^{2}}.
\end{equation} 
Note that $\|z_{\varepsilon}\|_{L^{\infty}(\Omega)}=1$, then by (\ref{local uniqueness proof step 3-5}), (\ref{local uniqueness proof step 3-11}) and (\ref{local uniqueness proof step 3-12}), we conclude 
\begin{equation}
    z_{\varepsilon}(x)=O\left(\frac{1}{(1+\lambda_{1,\varepsilon}|x-x_{1,\varepsilon}|)^{2}}\right).
\end{equation}
Now, by repeating the above process, we can obtain that
the claim (\ref{local uniqueness proof step 3-4}) hold. 
\end{proof}
Furthermore, we claim that the following estimate of $z_{\varepsilon}(x)$ hold.
\begin{Claim}
\begin{equation}\label{local uniqueness proof step 3-13}
    z_{\varepsilon}(x)=B_{\varepsilon}G(x_{1,\varepsilon},x)+\frac{1}{\lambda_{1,\varepsilon}^{N-1}}\sum_{h=1}^{N}B_{\varepsilon,h}\partial_{h} G(x_{1,\varepsilon},x)+O\left(\frac{\ln \lambda_{1,\varepsilon}}{\lambda_{1,\varepsilon}^{N}}\right),\quad \text{~in~}C^{1}(\Omega\setminus B(x_{1,\varepsilon},2\theta)),
\end{equation}
where $\theta>0$ is any fixed small constant, $\partial_{h} G(y,x)=\frac{\partial G(y,x)}{\partial y_{h}}$,
\begin{equation}
    B_{\varepsilon}=\int_{B(x_{1,\varepsilon},\theta)}k_{\varepsilon}(y)z_{\varepsilon}(y),
\end{equation}
\begin{equation}
    B_{\varepsilon,h}=\frac{1}{\lambda_{1,\varepsilon}^{2}}\int_{B(0,\lambda_{1,\varepsilon}\theta)}y_{h}k_{\varepsilon}\left(\frac{1}{\lambda_{1,\varepsilon}}y+x_{1,\varepsilon}\right)z_{\varepsilon}\left(\frac{1}{\lambda_{1,\varepsilon}}y+x_{1,\varepsilon}\right).
\end{equation}
\end{Claim}
\begin{proof}
For any $x\in \Omega\setminus B(x_{1,\varepsilon},\theta)$, by Green representation formula, we obtain that
\begin{equation}
\begin{aligned}
 z_{\varepsilon}(x)&=\int_{\Omega}G(x,y)k_{\varepsilon}(y)z_{\varepsilon}(y)dy\\
&=\int_{\Omega\setminus B(x_{1,\varepsilon},\theta)}G(x,y)k_{\varepsilon}(y)z_{\varepsilon}(y)dy+\int_{B(x_{1,\varepsilon},\theta)}G(x,y)k_{\varepsilon}(y)z_{\varepsilon}(y)dy.    
\end{aligned}
\end{equation}
Then, it follows from Lemma \ref{estimate of Green function}, (\ref{local uniqueness proof step 1-1}) (\ref{local uniqueness proof step 3-4}) and (\ref{local uniqueness proof step 2-1}), the first term in the right-hand side of the above equation can be estimate as
\begin{equation}
\begin{aligned}
\int_{\Omega\setminus B(x_{1,\varepsilon},\theta)}G(x,y)k_{\varepsilon}(y)z_{\varepsilon}(y)dy&\leq C\frac{1}{\lambda_{1,\varepsilon}^{N-2}} \int_{\Omega\setminus B(x_{1,\varepsilon},\theta)}\frac{1}{|x-y|^{N-2}}k_{\varepsilon}(y)dy\\     
&\leq C\left(\frac{1}{\lambda_{1,\varepsilon}^{N}}+\frac{\varepsilon}{\lambda_{1,\varepsilon}^{\frac{N-2}{2}q}}\right)\leq  C\left(\frac{1}{\lambda_{1,\varepsilon}^{N}}+\frac{1}{\lambda_{1,\varepsilon}^{(N-2)q-2}}\right)\\
&\leq C\left(\frac{\ln\lambda_{1,\varepsilon}}{\lambda_{1,\varepsilon}^{N}}\right),
\end{aligned} 
\end{equation}
and the second term can be bound by
\begin{align}
&\int_{B(x_{1,\varepsilon},\theta)}G(x,y)k_{\varepsilon}(y)z_{\varepsilon}(y)dy\nonumber\\
&=G(x_{1,\varepsilon},x)\int_{B(x_{1,\varepsilon},\theta)}k_{\varepsilon}(y)z_{\varepsilon}(y)dy+\int_{B(x_{1,\varepsilon},\theta)}\left(G(y,x)-G(x_{1,\varepsilon},x)\right)k_{\varepsilon}(y)z_{\varepsilon}(y)dy\nonumber\\
&=G(x_{1,\varepsilon},x)\int_{B(x_{1,\varepsilon},\theta)}k_{\varepsilon}(y)z_{\varepsilon}(y)dy+\int_{B(x_{1,\varepsilon},\theta)}\langle \partial G(x_{1,\varepsilon},x),y-x_{1,\varepsilon}\rangle k_{\varepsilon}(y)z_{\varepsilon}(y)dy\nonumber\\
&\quad +O\left(\frac{1}{|x-x_{1,\varepsilon}|^{N}}\int_{B(x_{1,\varepsilon},\theta)}|y-x_{1,\varepsilon}|^{2}|k_{\varepsilon}(y)||z_{\varepsilon}(y)|dy\right)\\
&=G(x_{1,\varepsilon},x)\int_{B(x_{1,\varepsilon},\theta)}k_{\varepsilon}(y)z_{\varepsilon}(y)dy+\frac{1}{\lambda_{1,\varepsilon}^{N-1}}\sum_{h=1}^{N}B_{\varepsilon,h}\partial_{h}G(x_{1,\varepsilon},x)\nonumber\\
&\quad+O\left(\frac{\ln\lambda_{1,\varepsilon}}{\lambda_{1,\varepsilon}^{N}|x-x_{1,\varepsilon}|^{N}}+\frac{\varepsilon}{\lambda_{1,\varepsilon}^{\frac{N-2}{2}q}|x-x_{1,\varepsilon}|^{N}}\right)\nonumber\\
&=G(x_{1,\varepsilon},x)\int_{B(x_{1,\varepsilon},\theta)}k_{\varepsilon}(y)z_{\varepsilon}(y)dy+\frac{1}{\lambda_{1,\varepsilon}^{N-1}}\sum_{h=1}^{N}B_{\varepsilon,h}\partial_{h}G(x_{1,\varepsilon},x)+O\left(\frac{\ln\lambda_{1,\varepsilon}}{\lambda_{1,\varepsilon}^{N}}\right)\nonumber.
\end{align}
Similarly, we can prove that (\ref{local uniqueness proof step 3-13}) holds in $C^{1}(\Omega\setminus B(x_{1,\varepsilon},2\theta))$.
\end{proof}

Let $\tilde{z}_{\varepsilon}(y)=z_{\varepsilon}\left(\frac{1}{\lambda_{1,\varepsilon}}y+x_{1,\varepsilon}\right)$ , then by (\ref{local uniqueness proof step 3-1}), we have 
\begin{equation}\label{local uniqueness proof step 4-1}
    \|\tilde{z}_{\varepsilon}(x)\|_{L^{\infty}(\tilde{\Omega})}=1.
\end{equation}
where $\tilde{\Omega}:=\lambda_{1,\varepsilon}(\Omega-x_{1,\varepsilon})$, then by standard elliptic regularity theory in \cite{gilbarg1977elliptic} with Arzel\'a-Ascoli theorem, we know that there exist a function $\tilde{z}_{0}(x)$ such that
\begin{equation}
    \tilde{z}_{\varepsilon}(x)\to z_{0}(x)\quad \text{~in~}C^{1}(B_{R}(0)),
\end{equation}
for any $R>0$. Moreover, by (\ref{local uniqueness proof step 3-2}) we have
\begin{equation}
    -\Delta \tilde{z}_{0}(x)=N(N+2)U_{1,0}^{2^*-2}\tilde{z}_{0}(x),\quad\text{~in~}\R^{N},
\end{equation}
which together with Lemma \ref{Kernel of Emden-Fowler equation} we have
\begin{equation}
\begin{aligned}
\tilde{z}_{0}(x)
&=a_{0}\frac{1-|x|^{2}}{(1+|x|^{2})^{\frac{N}{2}}} +\sum_{i=1}^{N}\frac{a_{i}x_{i}}{(1+|x|^{2})^{\frac{N}{2}}}\\
&:=\sum_{i=0}^{N}a_{i}\psi_{i}(x),\\
\end{aligned} 
\end{equation}
for some $a_{i}\in\R$. Then we have
\begin{equation}\label{local uniqueness proof step 4-5}
    B_{\varepsilon}=\frac{N(N+2)}{\lambda_{1,\varepsilon}^{N-2}}\left(\int_{\R^{N}}U_{1,0}^{2^*-2}a_{0}\psi_{0}+o(1)\right)
\end{equation}
and
\begin{equation}\label{local uniqueness proof step 4-6}
\begin{aligned}
      B_{\varepsilon,h}&\to N(N+2)\int_{\R^{N}}y_{h}U_{1,0}^{2^*-2}(\sum_{k=0}^{N}a_{k}\psi_{k}(y))=N(N+2)a_{h}\int_{\R^{N}}y_{h}U_{1,0}^{2^*-2}\psi_{h}(y)\\
      &=Na_{h}\int_{\R^{N}}U_{1,0}^{2^*-1},
\end{aligned}
\end{equation}
for $h=1,\cdots,N$. 

\begin{proof}[\text{Proof of Theorem} $\ref{local uniqueness-Robin function}$] 
We will obtain a contradiction by proving $a_{i}=0$ for any $i=0,\cdots,N$. First, we define that
\begin{equation}
    \begin{aligned}
        Q(z,u,\delta)=-\int_{\partial B(x_{1,\varepsilon},\delta)}\frac{\partial z}{\partial n}\frac{\partial u}{\partial x_{i}}-\int_{\partial B(x_{1,\varepsilon},\delta)}\frac{\partial u}{\partial n}\frac{\partial z}{\partial x_{i}}+\int_{\partial B(x_{1,\varepsilon},\delta)}\langle\nabla u,\nabla z\rangle n_{i}.
    \end{aligned}
\end{equation}
Inserting (\ref{local uniqueness proof step 1-1}), (\ref{local uniqueness proof step 1-2}), (\ref{local uniqueness proof step 3-4}) and (\ref{local uniqueness proof step 3-13}) to (\ref{local uniqueness proof step 3-pohozaev identity-1}), we have
\begin{equation}\label{local uniqueness proof step 4-8}
    \begin{aligned}
        &Q\left(B_{\varepsilon}G(x_{1,\varepsilon},x)+\sum_{h=1}^{N}\frac{B_{\varepsilon,h}\partial h G(x_{1,\varepsilon},x)}{\lambda_{1,\varepsilon}^{N-1}},\frac{A_{1,\varepsilon}G(x,x_{1,\varepsilon})}{\lambda_{1,\varepsilon}^{\frac{N-2}{2}}},\delta\right)\\
        &=\frac{1}{\lambda_{1,\varepsilon}^{N-2}}O\left(\frac{\ln \lambda_{1,\varepsilon}}{\lambda_{1,\varepsilon}^{\frac{N+2}{2}}}+\frac{\varepsilon}{\lambda_{1,\varepsilon}^{\frac{(N-2)(q-1)}{2}}}\right).
    \end{aligned}
\end{equation}
Since 
\begin{equation}\label{local uniqueness proof step 4-A-varepsilon}
   A_{1,\varepsilon}=\alpha_{N}^{2^*-1}\int_{\R^{N}}U_{1,0}^{2^*-1}+o(1)>0, 
\end{equation}
Then, it follows from (\ref{local uniqueness proof step 4-8}) and  (\ref{local uniqueness proof step 2-1}), we obtain
\begin{equation}\label{local uniqueness proof step 4-9}
    \begin{aligned}
      &Q\left(B_{\varepsilon}G(x_{1,\varepsilon},x)+\sum_{h=1}^{N}\frac{B_{\varepsilon,h}\partial_{h} G(x_{1,\varepsilon},x)}{\lambda_{1,\varepsilon}^{N-1}},G(x,x_{1,\varepsilon}),\delta\right)\\
      &=\frac{1}{\lambda_{1,\varepsilon}^{N-2}}O\left(\frac{\ln \lambda_{1,\varepsilon}}{\lambda_{1,\varepsilon}^{2}}+\frac{\varepsilon}{\lambda_{1,\varepsilon}^{\frac{(N-2)(q-2)}{2}}}\right)=O\left(\frac{\ln \lambda_{1,\varepsilon}}{\lambda_{1,\varepsilon}^{N}}\right).  
    \end{aligned}
\end{equation}
On the other hand, by Lemma \ref{Pohozaev identity of Green functon }, (\ref{local uniqueness proof step 4-5}) and (\ref{local uniqueness proof step 2-1}), we have
\begin{equation}\label{local uniqueness proof step 4-10}
    Q(B_{\varepsilon}G(x_{1,\varepsilon},x),G(x,x_{1,\varepsilon}),\delta)=B_{\varepsilon}Q(G(x_{1,\varepsilon},x),G(x,x_{1,\varepsilon}),\delta)=B_{\varepsilon}\frac{\partial R(x_{1,\varepsilon})}{\partial x_{i}}=O\left(\frac{\ln\lambda_{1,\varepsilon}}{\lambda_{1,\varepsilon}^{N}}\right).
\end{equation}
Then by (\ref{local uniqueness proof step 4-9}) and (\ref{local uniqueness proof step 4-10}), we deduce that
\begin{equation}
 \sum_{h=1}^{N}B_{\varepsilon,h}Q(\partial_{h}G(x_{1,\varepsilon},x),G(x,x_{1,\varepsilon}),\delta)=O\left(\frac{\ln\lambda_{1,\varepsilon}}{\lambda_{1,\varepsilon}}\right),
\end{equation}
which together with Lemma \ref{Pohozaev identity of Green functon } and the nondegeneracy of the critical point $x_{0}$, we get
\begin{equation}
    B_{\varepsilon,h}=O\left(\frac{\ln\lambda_{1,\varepsilon}}{\lambda_{1,\varepsilon}}\right).
\end{equation}
Therefore, by (\ref{local uniqueness proof step 4-6}), we have $a_{i}=0$, for $i=1,\cdots,N$. 

Next, we insert (\ref{local uniqueness proof step 1-1}), (\ref{local uniqueness proof step 1-2}), (\ref{local uniqueness proof step 3-4}) and (\ref{local uniqueness proof step 3-13}) to (\ref{local uniqueness proof step 3-pohozaev identity-2}), we have
\begin{equation}\label{local uniqueness proof step 4-13}
    \begin{aligned}
    &2A_{1,\varepsilon}B_{\varepsilon}\int_{B(x_{1,\varepsilon},\delta)}\frac{\partial G(x_{1,\varepsilon},x)}{\partial n}\langle x-x_{1,\varepsilon},\nabla G(x_{1,\varepsilon},x)\rangle\\
    &-A_{1,\varepsilon}B_{\varepsilon}\int_{B(x_{1,\varepsilon},\delta)}\langle\nabla G(x_{1,\varepsilon},x),\nabla G(x_{1,\varepsilon},x)\rangle\langle x-x_{1,\varepsilon},n\rangle\\
    &+(N-2)A_{1,\varepsilon}B_{\varepsilon}\int_{B(x_{1,\varepsilon},\delta)}\frac{\partial G(x_{1,\varepsilon},x)}{\partial n}G(x_{1,\varepsilon},x)\\
    &=\frac{1}{\lambda_{1,\varepsilon}^{N-2}}O\left(\frac{\ln\lambda_{1,\varepsilon}}{\lambda_{1,\varepsilon}^{2}}+\frac{\varepsilon}{\lambda_{1,\varepsilon}^{\frac{(N-2)(q-2)}{2}}}\right)+\varepsilon\lambda_{1,\varepsilon}^{\frac{N-2}{2}}\int_{B(x_{1,\varepsilon},\delta)}D_{\varepsilon}(x)z_{\varepsilon}(x).
    \end{aligned}
\end{equation}
Furthermore, we obtain
\begin{equation}
\varepsilon\lambda_{1,\varepsilon}^{\frac{N-2}{2}}\int_{B(x_{1,\varepsilon},\delta)}D_{\varepsilon}(x)z_{\varepsilon}(x)=  \frac{\varepsilon}{\lambda_{1,\varepsilon}^{(N-\frac{N-2}{2}q)}}(N-\frac{N-2}{2}q)\alpha_{N}^{q-1}\left(\int_{\R^{N}}U_{1,0}^{q-1}a_{0}\psi_{0}+o(1)\right), 
\end{equation}
which together with (\ref{local uniqueness proof step 4-13}), (\ref{local uniqueness proof step 4-5}), (\ref{local uniqueness proof step 4-A-varepsilon}) and Lemma \ref{Pohozaev identity of Green functon }, we have
\begin{equation}
\begin{aligned}
&\frac{(N-2)(2^*-1)R(x_{1,\varepsilon})}{\lambda_{1,\varepsilon}^{N-2}}\alpha_{N}^{22^*-3}\int_{\R^{N}}U_{1,0}^{2^*-1}\left(\int_{\R^{N}}U_{1,0}^{2^*-2}a_{0}\psi_{0}+o(1)\right)\\
&=\frac{\varepsilon}{\lambda_{1,\varepsilon}^{(N-\frac{N-2}{2}q)}}(N-\frac{N-2}{2}q)\alpha_{N}^{q-1}\left(\int_{\R^{N}}U_{1,0}^{q-1}a_{0}\psi_{0}+o(1)\right).    
\end{aligned}
\end{equation}
Note that
\begin{equation}
    \int_{\R^{N}}U_{1,0}^{q-1}\psi_{0}=\left(1-\frac{2^*}{q}\right)\int_{\R^{N}}U_{1,0}^{q},\quad \int_{\R^{N}}U_{1,0}^{2^*-2}\psi_{0}=\left(1-\frac{2^*}{2^*-1}\right)\int_{\R^{N}}U_{1,0}^{2^*-1},
\end{equation}
Thus
\begin{equation}
\begin{aligned}
&\frac{(N-2)^{2}}{2}\frac{R(x_{1,\varepsilon})}{\lambda_{1,\varepsilon}^{N-2}}\alpha_{N}^{22^*-2}\left(\int_{\R^{N}}U_{1,0}^{2^*-1}\right)^{2}a_{0}\\
&=\frac{\varepsilon}{\lambda_{1,\varepsilon}^{(N-\frac{N-2}{2}q)}}a_{0}\frac{(N-\frac{N-2}{2}q)^{2}}{q}\alpha_{N}^{q}\int_{B(x_{1,\varepsilon},\delta)}U_{1,0}^{q}+o\left(\frac{1}{\lambda_{1,\varepsilon}^{N-2}}\right).
\end{aligned}  
\end{equation}
Finally, by (\ref{local uniqueness proof step 2-11}), we can conclude that
\begin{equation}
    (N-2)a_{0}=(N-\frac{N-2}{2}q)a_{0},
\end{equation}
which together with the choice of $q$, we have $a_{0}=0$. Hence $\tilde{z}_{\varepsilon}\to 0$ uniformly in $B(0,R)$ for any $R>0$. 

Let $x_{\varepsilon}\in\Omega$ such that $|z_{\varepsilon}(x_{\varepsilon})|=\|z_{\varepsilon}\|_{L^{\infty}(\Omega)}=1$, then $|\tilde{z}_{\varepsilon}(\lambda_{1,\varepsilon}(x_{\varepsilon}-x_{1,\varepsilon}))|=1$ and $\lambda_{1,\varepsilon}(x_{\varepsilon}-x_{1,\varepsilon})\to\infty$ as $\varepsilon\to \infty$. But by (\ref{local uniqueness proof step 3-4}), we have
\begin{equation}
    |\tilde{z}_{\varepsilon}(x)|\leq \frac{C}{(1+|x|^{2})^{\frac{N-2}{2}}},
\end{equation}
then $|\tilde{z}_{\varepsilon}(\lambda_{1,\varepsilon}(x_{\varepsilon}-x_{1,\varepsilon}))|\to 0$ as $\varepsilon\to 0$, this is a contraction, thus $u_{\varepsilon}\equiv v_{\varepsilon}$.
\end{proof}

\section{Asymptotic nondegeneracy}\label{Asymptotic nondegeneracy}
In this section, we use an unified method to prove the asymptotic nondegeneracy of the least energy solution under two different assumptions.

\begin{proof}[Proof of Theorem $\ref{asymptotic nondegeneracy-domain and Robin function}$]
By contradiction, we suppose that there exists $v_{\varepsilon}\not\equiv 0$ such that $\|v_{\varepsilon}\|_{L^{\infty}}=\|u_{\varepsilon}\|_{L^{\infty}}$ and $v_{\varepsilon}$ solves (\ref{nondegenerate-1}). We define
\begin{equation}
     \tilde{u}_{\varepsilon}(x)=\frac{1}{\|u_{\varepsilon}\|_{L^{\infty}}}u_{\varepsilon}\left(\frac{x}{\|u_{\varepsilon}\|_{L^{\infty}}^{\frac{2}{N-2}}}+x_{\varepsilon}\right),\tilde{v}_{\varepsilon}(x)=\frac{1}{\|u_{\varepsilon}\|_{L^{\infty}}}v_{\varepsilon}\left(\frac{x}{\|u_{\varepsilon}\|_{L^{\infty}}^{\frac{2}{N-2}}}+x_{\varepsilon}\right),\quad x\in{\Omega_{\varepsilon}},
\end{equation}
where $x_{\varepsilon}$ is the maximum point of $u_{\varepsilon}$ and $\Omega_{\varepsilon}:=\|u_{\varepsilon}\|_{L^{\infty}}^{\frac{2}{N-2}}(\Omega-x_{\varepsilon})$. It is easy to verify that $\tilde{v}_{\varepsilon}$ satisfies that
\begin{equation}\label{nondegeneracy-1-proof-3}
    \begin{cases}
        -\Delta \tilde{v}_{\varepsilon}=(2^*-1)\tilde{u}^{2^*-2}_{\varepsilon}\tilde{v}_{\varepsilon}+\frac{\varepsilon(q-1)}{\|u_{\varepsilon}\|_{L^{\infty}}^{2^*-q}}\tilde{u}_{\varepsilon}^{q-2}\tilde{v}_{\varepsilon},&\quad\text{~~in~~}\Omega_{\varepsilon},\\
        \|\tilde{v}_{\varepsilon}\|_{L^{\infty}(\Omega_{\varepsilon})}=1,\\
        \tilde{v}_{\varepsilon}(x)=0,&\quad\text{~~on~~}\partial\Omega_{\varepsilon}.
    \end{cases}
\end{equation}
Since $\|\tilde{v}_{\varepsilon}\|_{L^{\infty}(\Omega_{\varepsilon})}=1$, by standard elliptic regularity theory \cite{gilbarg1977elliptic}, there exist a function $v_{0}$ such  that $\tilde{v}_{\varepsilon}\to v_{0}$ in $C^{1}_{loc}(\R^{N})$. Passing to the limit and by (2) in Theorem \ref{main theorem-1} we have
\begin{equation}\label{nondegeneracy-1-proof-4}
\begin{cases}
   -\Delta v_{0}=(2^*-1)U^{2^*-2}v_{0},\quad \text{~~in~~}\R^{N},\\
   \|v_{0}\|_{L^{\infty}(\Omega_{\varepsilon})}\leq 1.
\end{cases}
\end{equation}
Next, By (\ref{nondegeneracy-1-proof-3}) together with a similar proof of Step 2 in Theorem \ref{asymptotic uniqueness-1}, we have that $v_{0}\in \mathcal{D}^{1,2}(\R^{N})$, hence by Lemma \ref{Kernel of Emden-Fowler equation}, we obtain that
\begin{equation}\label{nondegeneracy-1-proof-4-1}
    v_{0}(x)=\sum_{i=1}^{N}\frac{a_{i}x_{i}}{(N(N-2)+|x|^{2})^{\frac{N}{2}}}+b\frac{N(N-2)-|x|^{2}}{(N(N-2)+|x|^{2})^{\frac{N}{2}}},
\end{equation}
for some $a_{i},b\in\R$. Moreover, for any $w_{0}$ be a neighbourhood of the origin, by the same proof of  step 3 in Theorem \ref{asymptotic uniqueness-1}, we have that there exist a positive constant C independent on $\varepsilon$ such that
\begin{equation}\label{nondegeneracy-1-proof-5}
    |\tilde{v}_{\varepsilon}(x)|\leq \frac{C}{|x|^{N-2}},\text{~~for any ~~} x\in\Omega_{\varepsilon}\setminus w_{0}.
\end{equation}

Next, we shall show that $a_{i}$ and $b=0$ by several steps.

\textbf{Step} 1. we claim $b=0$. By contraction, we suppose that $b\neq 0$. Then by a similar proof of Step 4 in Theorem \ref{asymptotic uniqueness-1}, we have that
\begin{equation}\label{nondegeneracy-1-proof-step-1-1}
    \|u_{\varepsilon}\|_{L^{\infty}}v_{\varepsilon}\to -b(N-2)\omega_{N}G(x,x_{0}),\quad \text{~~in~~}C^{1,\alpha}(\omega),
\end{equation}
where $\omega$ be a neighborhood of the boundary $\partial\Omega$, not containing the point $x_{0}$. And by (\ref{p-varepsion}) and (\ref{nondegenerate-1}) together with Green formula, we have
\begin{equation}\label{nondegeneracy-1-proof-step-1-2}
    \int_{\partial \Omega}\frac{\partial v_{\varepsilon}}{\partial n}\frac{\partial u_{\varepsilon}}{\partial n}(x-x_{0},n)dS=\varepsilon(N-\frac{N-2}{2}q)\int_{\Omega}u_{\varepsilon}^{q-1}v_{\varepsilon}dx.
\end{equation}
Thus 
\begin{equation}\label{nondegeneracy-1-proof-step-1-3}
\begin{aligned}
&\int_{\partial \Omega}\frac{\partial}{\partial n}(\|u_{\varepsilon}\|_{L^{\infty}}^{2}v_{\varepsilon})\frac{\partial }{\partial n}(u_{\varepsilon}\|u_{\varepsilon}\|_{L^{\infty}})(x-x_{0},n)dS\\
&=\varepsilon\|u_{\varepsilon}\|_{L^{\infty}}^{q+2-2^*}(N-\frac{N-2}{2}q)\int_{\Omega_{\varepsilon}}\tilde{u}_{\varepsilon}^{q-1}w_{\varepsilon}dx .    
\end{aligned}
\end{equation}
Passing to the limit of (\ref{nondegeneracy-1-proof-step-1-3}) and using (\ref{nondegeneracy-1-proof-step-1-1}), (\ref{nondegeneracy-1-proof-4}) and Theorem \ref{main theorem-1}, we have the left hand side of (\ref{nondegeneracy-1-proof-step-1-3})
\begin{equation}\label{nondegeneracy-1-proof-step-1-4}
    LHS\to -b\frac{(N-2)}{N}\alpha_{N}^{2^*}\omega_{N}^{2}\int_{\partial\Omega}|\frac{\partial G(x,x_{0})}{\partial n}|^{2}(x-x_{0},n)dS=-b\frac{(N-2)^{2}}{N}\alpha_{N}^{2^*}\omega_{N}^{2}R(x_{0}),
\end{equation}
and the right hand side of (\ref{nondegeneracy-1-proof-step-1-3})
\begin{equation}\label{nondegeneracy-1-proof-step-1-5}
\begin{aligned}
  RHS&\to (N-\frac{N-2}{2}q)\alpha_{N,q}R(x_{0})b\int_{\R^{N}}(\frac{N(N-2)}{(N(N-2)+|x|^{2})})^{\frac{(N-2)(q-1)}{2}}\frac{N(N-2)-|x|^{2}}{(N(N-2)+|x|^{2})^{\frac{N}{2}}}\\
  &=-b\frac{2N-(N-2)q}{2}(N-2)\frac{\omega_{N}^{2}\alpha_{N}^{2^*}}{N}R(x_{0}).
\end{aligned}
\end{equation}
Thus by (\ref{nondegeneracy-1-proof-step-1-4}) and (\ref{nondegeneracy-1-proof-step-1-5}), we conclude that
\begin{equation}
   q=\frac{4}{N-2},
\end{equation}
which make a contraction with the choice of $q$, thus $b=0$.

\textbf{Step} 2. $a_{i}=0$, for $i=1,\cdots,N$. If domain $\Omega$ satisfy assumption \ref{Assum-1}, it is easy to see that $a_{i}=0$, for $i=1,\cdots,N$. Next, we shall consider the case that $x_{0}$ is a nondegenerate critical point of Robin function $R(x)$. Firstly, from (\ref{p-varepsion}) and (\ref{nondegenerate-1}) together with Green formula (see \cite[Lemma 2.7]{Grossi2005ANR} for details), we have that 
\begin{equation}\label{nondegenerate-1-proof-step-2-1}
    \int_{\partial\Omega}\frac{\partial u_{\varepsilon}}{\partial x_{i}}\frac{\partial v_{\varepsilon}}{\partial n}dS_{x}=0.
\end{equation}
Next, we claim that 
\begin{equation}\label{nondegenerate-1-proof-step-2-2}
    \|u_{\varepsilon}\|_{L^{\infty}}^{\frac{N}{N-2}}v_{\varepsilon}(x)\to \omega_{N}\sum_{j=1}^{N}a_{j}\left(\frac{\partial G}{\partial z_{j} }(x,z)\right)\bigg|_{z=x_{0}},\quad\text{~~in~~}C^{1}_{loc}(\bar{\Omega}\setminus\{x_{0}\}).
\end{equation}
Indeed, for any $x\in C^{1}_{loc}(\bar{\Omega}\setminus\{x_{0}\})$ by Green representation formula \ref{Green's representation formula}, we obtain that
\begin{equation}
\begin{aligned}
v_{\varepsilon}(x)&=(2^*-1)\int_{\Omega}G(x,z)u_{\varepsilon}^{2^*-2}(z)v_{\varepsilon}(z)dz+\varepsilon (q-1)\int_{\Omega}G(x,z)u_{\varepsilon}^{q-2}(z)v_{\varepsilon}(z)dz\\&:=I_{1,\varepsilon}(x)+I_{2,\varepsilon}(x).    
\end{aligned} 
\end{equation}

Firstly, we consider the $I_{1,\varepsilon}(x)$. By a change of variables, we get
\begin{equation}
    I_{1,\varepsilon}(x)=\frac{(2^*-1)}{\|u_{\varepsilon}\|_{L^{\infty}}}\int_{\Omega_{\varepsilon}}G_{\varepsilon}(x,y)\tilde{u}_{\varepsilon}^{2^*-2}(y)\tilde{v}_{\varepsilon}(y)dy,
\end{equation}
and $G_{\varepsilon}(x,y)=G(x,\frac{y}{\|u_{\varepsilon}\|_{L^{\infty}}^{\frac{2}{N-2}}}+x_{\varepsilon})$.
Note that by Theorem 1, (\ref{nondegeneracy-1-proof-4-1}) and step 1, we have
\begin{equation}
    \tilde{u}_{\varepsilon}^{2^*-2}(y)\to U^{2^*-2}(y),
\end{equation}
and
\begin{equation}
    \tilde{v}_{\varepsilon}(y)\to v_{0}(y)=\sum_{j=1}^{N}a_{j}\frac{y_{j}}{(N(N-2)+|y|^{2})^{\frac{N}{2}}}=-\frac{1}{N^{\frac{N-2}{2}}(N-2)^{\frac{N}{2}}}\sum_{j=1}^{N}a_{j}\frac{\partial U}{\partial y_{j}},
\end{equation}
uniformly on the compact subsets on $\R^{N}$. Thus 
\begin{equation}
\tilde{u}_{\varepsilon}^{2^*-2}(y)\tilde{v}_{\varepsilon}(y)\to \sum_{j=1}^{N}a_{j}\left(\frac{\partial}{\partial y_{j}}\frac{-1}{(N(N-2))^{\frac{N-2}{2}}(N+2)}U^{2^*-1}\right).    
\end{equation}
Now, we consider the linear first order PDE
\begin{equation}
    \sum_{j=1}^{N}a_{j}\frac{\partial w}{\partial y_{j}}=\tilde{u}_{\varepsilon}^{2^*-2}\tilde{v}_{\varepsilon},\quad y\in\R^{N},
\end{equation}
with the initial condition $w|_{\Gamma_{a}}=\frac{-1}{(N(N-2))^{\frac{N-2}{2}}(N+2)}U^{2^*-1}(y)$, where $\Gamma_{a}=\{x\in\R^{N}|x\cdot a=0\}$. By  \cite[Lemma 2.4]{Takahashi2008nondegeneracy}, we have a solution $w_{\varepsilon}$ with
\begin{equation}
    w_{\varepsilon}\to \frac{-1}{(N(N-2))^{\frac{N-2}{2}}(N+2)}U^{2^*-1},
\end{equation}
uniformly on compact subsets on $\R^{N}$, and
\begin{equation}
    \int_{\R^{N}}w_{\varepsilon}(y)dy\to\frac{-1}{(N(N-2))^{\frac{N-2}{2}}(N+2)}\int_{\R^{N}}U^{2^*-1}dy =-\frac{1}{2^*-1}\omega_{N}.
\end{equation}
Thus
\begin{equation}
\begin{aligned}
I_{1,\varepsilon}(x)&=\frac{(2^*-1)}{\|u_{\varepsilon}\|_{L^{\infty}}}\int_{\Omega_{\varepsilon}}G_{\varepsilon}(x,y)\tilde{u}_{\varepsilon}^{2^*-2}(y)\tilde{v}_{\varepsilon}(y)dy\\
&=\frac{(2^*-1)}{\|u_{\varepsilon}\|_{L^{\infty}}}\int_{\Omega_{\varepsilon}}G_{\varepsilon}(x,y)\sum_{j=1}^{N}a_{j}\frac{\partial w_{\varepsilon}}{\partial y_{j}}dy\\
&=-\frac{(2^*-1)}{\|u_{\varepsilon}\|_{L^{\infty}}}\sum_{j=1}^{N}a_{j}\int_{\Omega_{\varepsilon}}\frac{\partial}{\partial y_{j}}G_{\varepsilon}(x,y)w_{\varepsilon}(y)dy\\
&=-\frac{(2^*-1)}{\|u_{\varepsilon}\|_{L^{\infty}}^{\frac{N}{N-2}}}\sum_{j=1}^{N}a_{j}\int_{\Omega_{\varepsilon}}\left(\frac{\partial G}{\partial y_{j}}\right)(x,\frac{y}{\|u_{\varepsilon}\|_{L^{\infty}}^{\frac{2}{N-2}}}+x_{\varepsilon})w_{\varepsilon}(y)dy,
\end{aligned}
\end{equation}
and
\begin{equation}
\|u_{\varepsilon}\|_{L^{\infty}}^{\frac{N}{N-2}}I_{1,\varepsilon}(x)\to \omega_{N}\sum_{j=1}^{N}a_{j}\left(\frac{\partial G}{\partial z_{j} }(x,z)\right)\bigg|_{z=x_{0}},
\end{equation}
for any $x\in C^{1}_{loc}(\Omega\setminus\{x_{0}\})$.

Next, we consider $I_{2,\varepsilon}(x)$ 
By a change of variables, we obtain
\begin{equation}
    I_{2,\varepsilon}(x)=\frac{\varepsilon(q-1)}{\|u_{\varepsilon}\|_{L^{\infty}}^{2^*+1-q}}\int_{\Omega_{\varepsilon}}G_{\varepsilon}(x,y)\tilde{u}_{\varepsilon}^{q-2}(y)\tilde{v}_{\varepsilon}(y)dy.
\end{equation}
Note that by Theorem 1, (\ref{nondegeneracy-1-proof-4-1}) and step 1, we have
\begin{equation}
    \tilde{u}_{\varepsilon}^{q-2}(y)\to U^{q-2}(y),
\end{equation}
and
\begin{equation}
    \tilde{v}_{\varepsilon}(y)\to v_{0}(y)=\sum_{j=1}^{N}a_{j}\frac{y_{j}}{(N(N-2)+|y|^{2})^{\frac{N}{2}}}=-\frac{1}{N^{\frac{N-2}{2}}(N-2)^{\frac{N}{2}}}\sum_{j=1}^{N}a_{j}\frac{\partial U}{\partial y_{j}},
\end{equation}
uniformly on the compact subsets on $\R^{N}$. Thus 
\begin{equation}
\tilde{u}_{\varepsilon}^{q-2}(y)\tilde{v}_{\varepsilon}(y)\to \sum_{j=1}^{N}a_{j}\left(\frac{\partial}{\partial y_{j}}\frac{-1}{N^{\frac{N-2}{2}}(N-2)^{\frac{N}{2}}(q-1)}U^{q-1}\right).  
\end{equation}
Now we consider the linear first order PDE
\begin{equation}
    \sum_{j=1}^{N}a_{j}\frac{\partial w}{\partial y_{j}}=\tilde{u}_{\varepsilon}^{q-2}\tilde{v}_{\varepsilon},\quad y\in\R^{N},
\end{equation}
with the initial condition $w|_{\Gamma_{a}}=\frac{-1}{N^{\frac{N-2}{2}}(N-2)^{\frac{N}{2}}(q-1)}U^{q-1}$, where $\Gamma_{a}=\{x\in\R^{N}|x\cdot a=0\}$. By \cite[Lemma 2.4]{Takahashi2008nondegeneracy}, there exist a solution $w_{\varepsilon}$ with
\begin{equation}\label{nondegeneracy-1-proof-step-2-18}
    w_{\varepsilon}\to \frac{-1}{N^{\frac{N-2}{2}}(N-2)^{\frac{N}{2}}(q-1)}U^{q-1},
\end{equation}
uniformly on compact subsets on $\R^{N}$ and by Theorem \ref{main theorem-1} and (\ref{nondegeneracy-1-proof-5}), we have
\begin{equation}\label{nondegeneracy-1-proof-step-2-19}
    |w_{\varepsilon}(y)|\leq C\frac{1}{|y|^{(N-2)(q-1)-1}},\quad \text{~as~}|y|\to\infty.
\end{equation}
Thus
\begin{equation}\label{nondegeneracy-1-proof-step-2-20}
\begin{aligned}
I_{2,\varepsilon}(x)&=\frac{\varepsilon(q-1)}{\|u_{\varepsilon}\|_{L^{\infty}}^{2^*+1-q}}\int_{\Omega_{\varepsilon}}G_{\varepsilon}(x,y)\tilde{u}_{\varepsilon}^{q-2}(y)\tilde{v}_{\varepsilon}(y)dy\\
&=\frac{\varepsilon(q-1)}{\|u_{\varepsilon}\|_{L^{\infty}}^{2^*+1-q}}\int_{\Omega_{\varepsilon}}G_{\varepsilon}(x,y)\sum_{j=1}^{N}a_{j}\frac{\partial w_{\varepsilon}}{\partial y_{j}}dy\\
&=-\frac{\varepsilon(q-1)}{\|u_{\varepsilon}\|_{L^{\infty}}^{2^*-q+\frac{N}{N-2}}}\sum_{j=1}^{N}a_{j}\int_{\Omega_{\varepsilon}}\left(\frac{\partial G}{\partial y_{j}}\right)(x,\frac{y}{\|u_{\varepsilon}\|_{L^{\infty}}^{\frac{2}{N-2}}}+x_{\varepsilon})w_{\varepsilon}(y)dy.
\end{aligned}
\end{equation}
Note that, when $(N-2)(q-1)>N$, by (\ref{nondegeneracy-1-proof-step-2-18}) we deduce that
\begin{equation}\label{nondegeneracy-1-proof-step-2-21}
    \int_{\Omega_{\varepsilon}}w_{\varepsilon}(y)dy\to  \frac{-1}{N^{\frac{N-2}{2}}(N-2)^{\frac{N}{2}}(q-1)}\int_{\R^{N}}U^{q-1}(y)=C<\infty
\end{equation}
and for $(N-2)(q-1)\leq N$, by (\ref{nondegeneracy-1-proof-step-2-19}) we obtain
\begin{equation}\label{nondegeneracy-1-proof-step-2-22}
\begin{aligned}
\int_{\Omega_{\varepsilon}}w_{\varepsilon}(y)dy&\leq C+C\int_{B(0,\|u_{\varepsilon}\|_{L^{\infty}}^{\frac{2}{N-2}}\gamma)\setminus\{B(0,1)\}}\frac{1}{|y|^{(N-2)(q-1)-1}}dy\\
&\leq C+C\int_{1}^{\|u_{\varepsilon}\|_{L^{\infty}}^{\frac{2}{N-2}}\gamma}\frac{1}{r^{(N-2)(q-1)-N}}dr\\
&\leq C\|u_{\varepsilon}\|_{L^{\infty}}^{\frac{2N+2}{N-2}-2(q-1)},
\end{aligned}
\end{equation}
for some constant $C$ and $\gamma$. Then by (\ref{nondegeneracy-1-proof-step-2-20}), (\ref{nondegeneracy-1-proof-step-2-21}), (\ref{nondegeneracy-1-proof-step-2-22}) and Theorem \ref{main theorem-1} we have
\begin{equation}
\begin{aligned}
\|u_{\varepsilon}\|_{L^{\infty}}^{\frac{N}{N-2}}I_{2,\varepsilon}(x)&\leq \frac{C}{\|u_{\varepsilon}\|_{L^{\infty}}^{2}}|\int_{\Omega_{\varepsilon}}w_{\varepsilon}(y)dy|\\
&\leq \begin{cases}
    \frac{C}{\|u_{\varepsilon}\|_{L^{\infty}}^{2}},&\quad\text{~if~} (N-2)(q-1)>N,\\
    \frac{C}{\|u_{\varepsilon}\|_{L^{\infty}}^{2q-\frac{2N+2}{N-2}}},&\quad\text{~if~}(N-2)(q-1)\leq N,\\
\end{cases}\\
&\to 0.
\end{aligned}
\end{equation}
Hence claim (\ref{nondegenerate-1-proof-step-2-2}) hold. Moreover by (\ref{nondegenerate-1-proof-step-2-2}), Theorem \ref{main theorem-1} and Lemma \ref{Green function propertity 1}, we conclude that
\begin{equation}
\begin{aligned}
\int_{\partial\Omega}\frac{\partial \|u_{\varepsilon}\|_{L^{\infty}}u_{\varepsilon}}{\partial x_{i}}\frac{\partial \|u_{\varepsilon}\|_{L^{\infty}}^{\frac{N}{N-2}}v_{\varepsilon}}{\partial n}dS_{x}&\to \frac{\alpha_{N}^{2^*}\omega_{N}^{2}}{N}\int_{\partial\Omega}\frac{\partial G}{\partial x_{i}}(x,x_{0})\frac{\partial}{\partial n}\left(\sum_{j=1}^{N}a_{j}\frac{\partial G}{\partial z_{j}}(x,x_{0})\right)dS\\
    &=\frac{\alpha_{N}^{2^*}\omega_{N}^{2}}{2N}\sum_{j=1}^{N}a_{j}\frac{\partial^{2}R}{\partial x_{i}\partial x_{j}}(x_{0}),   
\end{aligned}    
\end{equation}
which together with (\ref{nondegenerate-1-proof-step-2-1}) and the nondegeneracy assumption of $x_{0}$, we have $a_{i}=0$ for $i=1,\cdots,N$.

Now, by the proof of Step 1 and Step 2, we know that $\tilde{v}_{\varepsilon}\to 0$ uniformly on compact subsets of $\R^{N}$. Let $\tilde{x}_{\varepsilon}$ be the point such $|\tilde{v}_{\varepsilon}(\tilde{x}_{\varepsilon})|=\|\tilde{v}_{\varepsilon}\|_{L^{\infty}(\Omega_{\varepsilon})}=1$, then $\tilde{x}_{\varepsilon}\to \infty$ as $\varepsilon \to 0$, but this make a contraction with (\ref{nondegeneracy-1-proof-5}), thus we complete the proof.
\end{proof}

%%%%%%%%%%%%%%%%%%%%%%%%%%%%%%%%%%%%%%%%%%%%%%%%%%%%%%%%%%%%%%%%%%%%%%%%%%%%%%%%%%%%%%%%%%%%%%%%%%%%%%%%%%%%%%%%%%

\section{appendix}

%\appendix

\subsection{Some elementary estimates}
%\section{Some Elementary estimates}
 From \cite[Lemma 6.1.1]{Cao_Peng_Yan_2021} and \cite[Lemma B.2]{Wei2010InfinitelyMS}, we have the following elementary estimates.
 
 \begin{Lem}\label{Elementary estimate}
  For $a,b>0$, we have the following estimate
  \begin{equation}
 \begin{aligned}
     &(a+b)^{p}=a^{p}+pa^{p-1}b+O(b^{p}),\text{~if~} p\in(1,2],\\
    &(a+b)^{p}=a^{p}+pa^{p-1}b+\frac{p(p-1)}{2}a^{p-2}b^{2}+O(b^{p}),\text{~if~} p>2.
 \end{aligned}
 \end{equation}
\end{Lem}

\begin{Lem}\label{useful estimate}
    For any constant $0<\theta\leq N-2$, there exist a constant $C>0$ such that
    \begin{equation}
        \int_{\R^{N}}\frac{1}{|y-z|^{N-2}}\frac{1}{(1+|z|)^{2+\theta}}dz\leq \begin{cases}
            C(1+|y|)^{-\theta},\quad \theta<N-2,\\
            C|\log|y||(1+|y|)^{-\theta}, \quad\theta=N-2.\\
        \end{cases}
    \end{equation}
\end{Lem}

\subsection{Green function and Robin function}

From  \cite[Proposition 6.7.1]{Cao_Peng_Yan_2021}, we have the following estimate for Robin function $R(x)$.
\begin{Lem}\label{estimate of robin function up to doundary}
Let $d=d(x,\partial\Omega)$ for $x\in \Omega$. Then as $d\to 0$, we have
\begin{equation}
R(x)=\frac{1}{(N-2)\omega_{N}}\frac{1}{(2d)^{N-2}}(1+O(d))    
\end{equation}
and 
\begin{equation}
    \nabla R(x)=\frac{2}{\omega_{N}}\frac{1}{(2d)^{N-1}}\frac{x^{\prime}-x}{d}+O\left(\frac{1}{d^{N-2}}\right),
    \end{equation}
where $x'\in\partial\Omega$ is the unique point, satisfying $d(x,\partial\Omega)=|x-x'|$.
\end{Lem}

% In particular, when $\Omega$ is a convex domain, by Theorem 3.1 in \cite{Caffarelli1985ConvexityOS} and Corollary 3.2 \cite{Cardaliaguet2002OnTS} we have the following lemma.
% \begin{Lem}
% If $N\geq 2$ and $\Omega$ is a bounded convex domain of $\R^{N}$, then the Robin function $R$ is strictly convex and then $R$ has a unique minimum point.   
% \end{Lem}

% We also introduce the notion of harmonic radius and harmonic center \cite{Caffarelli1985ConvexityOS}
% \begin{Def}\label{definition of harmonic center and harmonic radius}
% The harmonic radius $r(x):\Omega\to\R^{+}$ is defined by
% \begin{equation}
%     r(x)=F^{-1}(R(x)) \text{~~with~~} F(t)=t^{2-N},
% \end{equation}
% and a harmonic center of $\Omega$ is a point of $\Omega$ such that $r(x)$ attains its maximum (or, equivalently, a point of $\Omega$ such that $R(x)$ attains its minimum).
% \end{Def}
\begin{Rem}
In particularly, when $\Omega=B(0,R)$ and $N\geq 3$
    \begin{equation}
        R(x)=\frac{1}{(N-2)\omega_{N}}\left(\frac{R^2-|x|^2}{R}\right)^{2-N}.
    \end{equation}
%then $0$ is harmonic center of $R$.     
 \end{Rem}
 
Next, from \cite[Theorem 4.3--4.4]{Brezispletier1989} and  \cite[Lemma 2.1]{Takahashi2008nondegeneracy}
\begin{Lem}\label{Green function propertity 1}
    For any $y\in \Omega$, we have
    \begin{equation}
        \int_{\partial\Omega}(x-y,n)\left(\frac{\partial G(x,y)}{\partial n}\right)^{2}dS_{x}=(N-2)R(y),
    \end{equation}
    \begin{equation}
        \int_{\partial\Omega}\left(\frac{\partial G(x,y)}{\partial n}\right)^{2}n(x)dS_{x}=\nabla R(y), 
    \end{equation}
    and
    \begin{equation}
        \int_{\partial\Omega}\left(\frac{\partial G}{\partial x_{i}}\right)\frac{\partial}{\partial n}\left(\frac{\partial G}{\partial y_{j}}\right)(x,y)dS_{x}=\frac{1}{2}\frac{\partial^{2}R}{\partial x_{i}\partial x_{j}}(y),
    \end{equation}
where $n=n(x)$ denote the unit outward normal to $\partial\Omega$ at $x$.
\end{Lem}
By strong maximum principle and Lemma A.1 in \cite{ACKERMANN20134168}
\begin{Lem}\label{estimate of Green function}
 For any $x\neq y\in\Omega$, there exists a constant $C>0$ such that
 \begin{equation}
     0<G(x,y)<\frac{C}{|x-y|^{N-2}} \text{~~and~~} |\nabla_{x}G(x,y)|\leq \frac{C}{|x-y|^{N-1}}. 
 \end{equation}
\end{Lem}
 By  \cite[Proposition 6.23--6.2.5]{Cao_Peng_Yan_2021}, we have.
\begin{Lem}\label{Pohozaev identity of Green functon }
Let $x\in\Omega$ and $B(x,d)\subset\subset \Omega$, then we have
\begin{equation}
\begin{aligned}
&-\int_{\partial B(x,d)}\frac{\partial G(y,x)}{\partial n}(y-x)\cdot\nabla G(y,x)+\frac{1}{2}\int_{\partial B(x,d)}(y-x)\cdot n|\nabla G(y,x)|^{2}\\ 
&-\frac{N-2}{2}\int_{\partial B(x,d)}G(y,x)\frac{\partial G(y,x)}{\partial n}\\
=&-\frac{N-2}{2}H(x,x),
\end{aligned}
\end{equation}
and
\begin{equation}
\int_{\partial B(x,d)}\frac{\partial G(y,x)}{\partial n} \frac{\partial G(y,x)}{\partial y_{i}}-\frac{1}{2}\int_{\partial B(x,d)} |\nabla G(y,x)|^{2}n_{i}=\frac{\partial H(y,x)}{\partial y_{i}}\bigg|_{y=x}.   
\end{equation}
Moreover, we define
\begin{equation}
    Q(u,v,B(x,d))=-\int_{\partial B(x,d)}\frac{\partial u}{\partial x_{i}} \frac{\partial v}{\partial n} -\int_{\partial B(x,d)}\frac{\partial v}{\partial x_{i}} \frac{\partial u}{\partial n}+\int_{\partial B(x,d)}\langle \nabla u,\nabla v\rangle n_{i},
\end{equation}
then we have
\begin{equation}
\begin{aligned}
Q(G(x,y),\partial_{h}G(x,y),B(x,d))=-\frac{1}{2}\frac{\partial^{2}R(x)}{\partial x_{i}\partial x_{h}},
\end{aligned}
\end{equation}
where $n(x)$ denote the unit outward normal of $\partial B(x,d)$ at $x$ and $\partial_{h}G(x,y)=\frac{\partial G(x,y)}{\partial x_{h}}$.
\end{Lem}

Finally, We recall the following Green's representation formula \cite{evans2022partial}
\begin{equation}\label{Green's representation formula}
    u(x)=-\int_{\Omega}G(x,y)\Delta u(y)dy-\int_{\partial\Omega}u(y)\frac{\partial G}{\partial n}(x,y)d S_{y}.
\end{equation}

\subsection{Projection of bubbles}\label{Projection of bubbles}
For any $a\in\Omega$ and $\lambda\in \R^{+}$, we define $PU_{\lambda,a}$ is the projection of $U_{\lambda,a}$ onto $H^{1}_{0}(\Omega)$, i.e. 
$PU_{\lambda,a}:=U_{\lambda,a}-\psi_{\lambda,a}\in H^{1}_{0}(\Omega)$, where $\psi_{\lambda,a}$ is the harmonic extension of $U_{\lambda,a}|_{\partial\Omega}$ to $\Omega$
\begin{equation}
  \begin{cases}
-\Delta \psi_{\lambda,a}=0, \quad{\text{~in~}\Omega},\\
\psi_{\lambda,a}|_{\partial\Omega}=U_{\lambda,a}|_{\partial\Omega}.  
\end{cases}
\end{equation}
Hence, by strong maximum principle, we have
 \begin{equation}
  \psi_{\lambda,a}\leq U_{\lambda,a}(x):=\left(\frac{\lambda}{1+\lambda^{2}|x-a|^{2}}\right)^{\frac{N-2}{2}}.
\end{equation}

Now, we recall some useful estimates obtained in \cite{bahri1988,Rey1990}.

\begin{Lem}\label{estimate of two different bubbles}
For any $\lambda_{1}, \lambda_{2}\in \R^{+}$ and $x_{1},x_{2}\in\R$, we have
\begin{equation}
    \int_{\R^{N}}U_{\lambda_{1},x_{1}}^{2^*-1}U_{\lambda_{2},x_{2}}=O\left(\frac{\lambda_{1}}{\lambda_{2}}+\frac{\lambda_{2}}{\lambda_{1}}+\lambda_{1}\lambda_{2}|x_{1}-x_{2}|\right)^{-\frac{N-2}{2}}.
\end{equation}
\end{Lem}

\begin{Lem}\label{estimate of U-lambda-a and psi-lambda-a 1}
Assume that $a\in\Omega$ and $\lambda\in \R^{+}$, then we have the following properties    \begin{equation*}\begin{aligned}
&\frac{\partial U_{\lambda,a}(x)}{\partial a_j} =(N-2)\lambda^{\frac{N+2}{2}}\frac{x_{j}-a_{j}}{\left(1+\lambda^{2}|x-a|^{2}\right)^{\frac{N}{2}}}=O\big(\lambda U_{a,\lambda}\big), \\
&\frac{\partial U_{\lambda,a}(x)}{\partial\lambda} =\frac{N-2}2\lambda^{\frac{N-4}2}\frac{1-\lambda^2|x-a|^2}{(1+\lambda^2|x-a|^2)^{\frac N2}}=O\Big(\frac{U_{a,\lambda}}\lambda\Big),
\end{aligned}\end{equation*}
where $d=\text{dist}(a,\partial\Omega)$ is the distance between $a$ and boundary $\partial\Omega$.
\end{Lem}

\begin{Lem}\label{estimate of U-lambda-a and psi-lambda-a 2}
Assume that $a\in\Omega$ and $\lambda\in \R^{+}$, then we have the following properties
\begin{equation*}
 \begin{aligned}
PU_{\lambda,a}& =U_{\lambda,a}-\psi_{\lambda,a}\geq 0, \\
\psi_{\lambda,a}(x)& =\frac{(N-2)\omega_N}{\lambda^{\frac{N-2}2}}H(a,x)+O\Big(\frac1{\lambda^{\frac{N+2}2}d^N}\Big), \\
\frac{\partial\psi_{\lambda,a}(x)}{\partial\lambda}& =-\frac{(N-2)\omega_N}{\lambda^{\frac N2}}H(a,x)+O\Big(\frac1{\lambda^{\frac{N+4}2}d^N}\Big), \\
\frac{\partial\psi_{\lambda,a}(x)}{\partial a_j}& =\frac{(N-2)\omega_N}{\lambda^{\frac{N-2}2}}\frac{\partial H(a,x)}{\partial a_j}+O\Big(\frac{1}{\lambda^{\frac{N+2}2}d^{N+1}}\Big),
\end{aligned}      
\end{equation*}
 where $d=\text{dist}(a,\partial\Omega)$ is the distance between $a$ and boundary $\partial\Omega$.
\end{Lem}

\begin{Lem}\label{estimate of U-lambda-a and psi-lambda-a 3}
Assume that $a\in\Omega$ and $\lambda\in \R^{+}$, then we have the following properties
    \begin{equation*}
        \begin{aligned}
\int_{\R^{N}\setminus\Omega}|\nabla U_{\lambda,a}|^{2}=O\left(\frac{1}{(\lambda d)^{N-2}}\right)&,\int_{\R^{N}\setminus\Omega}|U_{\lambda,a}|^{2^*}=O\left(\frac{1}{(\lambda d)^{N}}\right),\\
\int_{\Omega}|\nabla U_{\lambda,a}|^{2}=\alpha_{N}^{-2}S^{\frac{N}{2}}+O\left(\frac{1}{(\lambda d)^{N-2}}\right),&\int_{\Omega}|\nabla PU_{\lambda,a}|^{2}=\alpha_{N}^{-2}S^{\frac{N}{2}}+O\left(\frac{1}{(\lambda d)^{N-2}}\right),\\
\|\psi_{\lambda,a}\|_{L^{\infty}} =O\left(\frac1{\lambda^{\frac{N-2}2}d^{N-2}}\right), \parallel\psi_{\lambda,a}\parallel_{L^{2^{*}}}&=O\Big(\frac{1}{(\lambda d)^{\frac{N-2}{2}}}\Big), \|\nabla \psi_{\lambda,a}\|_{L^{2}}=O\left(\frac{1}{(\lambda d)^{N-2}}\right),\\
        \end{aligned}
    \end{equation*}
where $d=\text{dist}(a,\partial\Omega)$ is the distance between $a$ and boundary $\partial\Omega$.
\end{Lem}

\subsection{Some elliptic estimates}
In this subsection, we present some elliptic estimates that will be used frequently in our
 proof.
\begin{Lem}\label{gradient estimate}\cite{Brezispletier1989,Han1991}
    Let $u$ be a solution of the following problem
    \begin{equation}
        \begin{cases}
            -\Delta u=f, &\text{~in~} \Omega,\\
           \quad \ \ u=0, &\text{~on~} \partial \Omega.
        \end{cases}
    \end{equation}
    Then, for any $q<\frac{N}{N-1}$, $\alpha\in (0,1)$ and $\omega'\subset\subset \omega$ be two  neighborhoods of $\partial\Omega$ , there exists $C>0$ not depending on $u$ and $f$ such that
    \begin{equation}
        \|u\|_{W^{1,q}(\Omega)}\leq C \|f\|_{L^{1}(\Omega)},
    \end{equation}
and
\begin{equation}
\|\nabla u\|_{C^{0,\alpha}(\omega^{\prime})}\leq C(\|f\|_{L^{1}(\Omega)}+\|f\|_{L^{\infty}(\omega)}). 
\end{equation}
\end{Lem}

\begin{Rem}
Since the embedding $W^{1,q}(\Omega) \to L^{1}(\Omega)$ is compact if $q<\frac{N}{N-1}$, we obtain that if $\{f_{n}\}$ is a sequence of functions bounded both in $L^{1}(\Omega)$ and in $L^{\infty}(\omega)$, the corresponding solutions $\{u_{n}\}$ converge, up to a subsequence, a.e. on $\Omega$, to a function $u$ and, by the Arzel\`a-Ascoli theorem, the sequence $\{\nabla u_{n}\}$ converges to $\nabla u$ uniformly on $\omega$.
\end{Rem}

\begin{Lem} \cite[Appendix B]{Akahori2019CVPDE} \label{moser iteration}
Assume $N\geq 3$. Let $a(x)$ and $b(x)$ be functions on $B_{4}$ and let $u\in H^{1}(B_{4})$ be a weak solution to 
\begin{equation}
    -\Delta u+a(x)u=b(x)u, \text{~in~} B_{4}.
\end{equation}
Suppose that $a(x)$ and $u$ satisfy that
\begin{equation}
 a(x)\geq0\quad\text{for a.a. }x\in B_4,\quad\int_{B_4}a(x)|u(x)v(x)|dx<\infty\quad\text{for each }v\in H_0^1(B_4).   
\end{equation}
\end{Lem}
\begin{enumerate}
    \item Assume that for any $\varepsilon\in(0,1)$, there exists $t_{\varepsilon}>0$ such that 
    \begin{equation}
        \|\chi_{[|b|>t_{\varepsilon}]}b\|_{L^{d/2}(B_{4})}\leq \varepsilon,
    \end{equation}
    where $[|b|>t_{\varepsilon}]:=\{x\in B_{4}:|b(x)|>t\}$ and $\chi_{A}$ denotes the characteristic function of $A\subset \R^{N}$. Then, for any $q\in (0,\infty)$, there exists a constant $C(N,q,t_{\varepsilon})$ such that
    \begin{equation}
        \||u|^{q+1}\|_{H^{1}(B_{1})}\leq C(N,q,t_{\varepsilon})\|u\|_{L^{2^*}(B_{4})}.
    \end{equation}
\item 
Let $s>\frac{N}{2}$ and assume that $b(x)\in L^s(B_4)$. Then, there exists a constant $C\left(N,s,\right.$ $\left.\|b\|_{L^s(B_4)}\right)$ such that
\begin{equation}
 \|u\|_{L^\infty(B_1)}\leq C\left(N,s,\|b\|_{L^s(B_4)}\right)\|u\|_{L^{2^*}(B_4)}. 
\end{equation}
\end{enumerate}
Here, the constant $C(N,q,t_{\varepsilon})$ and $C(N,s,\|b\|_{L^s(B_4)})$ remain bounded as long as $q$, $t_{\varepsilon}$ and $\|b\|_{L^s(B_4)}$ are bounded.

\subsection{Local Poho\v{z}aev identity}

We recall the following local Poho\v{z}aev identities, see \cite[Appendix 6.2]{Cao_Peng_Yan_2021}.
\begin{Lem}\label{Local Pohozaev identiy}
 Suppose that $\Omega$ is a smooth domain in $\R^{N}$ and $f(x,t)\in C(\bar{\Omega}\times\R,\R)$, $u\in C^{2}(\bar{\Omega})$ is a solution of 
 \begin{equation}
     -\Delta u=f(x,u), \text{~in~}\Omega.
 \end{equation}
 Then, for any bounded domain $D \subset \Omega$, one has the following identity
 \begin{equation}
  \begin{aligned}
      &\int_{\partial D}\left((\nabla u\cdot n)((x-y)\cdot\nabla u)-\frac{|\nabla u|^{2}}{2}(x-y)\cdot n\right)+\frac{N-2}{2}\int_{\partial D}u\frac{\partial u}{\partial n} d S_{x}\\
      &\quad+\int_{\partial D} F(x,u)((x-y)\cdot n)d S_{x}\\
      &=\int_{D}N F(x,u)+\frac{2-N}{2}f(x,u)u+(x-y)\cdot F_{x}(x,u)dx,
  \end{aligned}   
 \end{equation}
and for $i=1,\cdots,N $, one has
\begin{equation}
    \int_{\partial D}\frac{\partial u}{\partial x_{i}}\frac{\partial u}{\partial n}-\frac{1}{2}|\nabla u|^{2}n_{i}+F(x,u)n_{i}d S_{x}=\int_{D}F_{x_{i}}(x,u)dx.
\end{equation}
where $y\in\R^{N}$, $F(x,u)=\int_{0}^{u}f(x,t)dt$, $F_{x}$ is the gradient of $F$ with respect to $x$, $dS_{x}$ is the volume element of $\partial\Omega$ and $n$ is the unit outward normal of $\partial\Omega$.
\end{Lem}
\begin{Cor}\label{Pohozaev identity}
Let $\Omega$ be a bounded smooth domain in $\R^{N}$ and $f(x,t)\in C(\bar{\Omega}\times\R,\R)$, $u\in C^{2}(\bar{\Omega})$ is a solution of
\begin{equation}
    \begin{cases}
        -\Delta u=f(x,u),&\quad\text{~in~}\Omega,\\
       \quad \ \ u=0,&\quad\text{~on~}\partial\Omega.       
    \end{cases}
\end{equation}
Then the following identity holds
\begin{equation}
    \begin{aligned}
\int_{\Omega} N F(x,u)-\frac{N-2}{2}u f(x,u)+(x-y)\cdot F_{x}(x,u)dx=\frac{1}{2}\int_{\partial\Omega}|\nabla u|^{2}(x-y)\cdot n dS_{x},
    \end{aligned}
\end{equation}
where $y\in\R^{N}$, $F(x,u)=\int_{0}^{u}f(x,t)dt$, $F_{x}$ is the gradient of $F$ with respect to $x$, $dS_{x}$ is the volume element of $\partial\Omega$ and $n$ is the unit outward normal of $\partial\Omega$.
\end{Cor}

\subsection{Kernel of Linear operator} The following lemma is proved in \cite{Bianchi1991ANO}. 
\begin{Lem}\cite[Lemma A.1]{Bianchi1991ANO}\label{Kernel of Emden-Fowler equation}
If $u$ is a solution of the following equation
\begin{equation}
\begin{cases}
    -\Delta u=(2^*-1)U^{2^*-2}u, \textrm{~~in~~}\R^{N},\\
    u\in\mathcal{D}^{1,2}(\R^{N}),
\end{cases}    
\end{equation}   
with $U=\left(\frac{N(N-2)}{N(N-2)+|x|^{2}}\right)^{\frac{N-2}{2}}$, then
\begin{equation}
u(x)=\sum_{i=1}^{N}\frac{a_{i}x_{i}}{(N(N-2)+|x|^{2})^{\frac{N}{2}}}+b\frac{N(N-2)-|x|^{2}}{(N(N-2)+|x|^{2})^{\frac{N}{2}}},
\end{equation}
for some $a_{i},b\in\R$.
\end{Lem}

%%%%%%%%%%%%%%%%%%%%%%%%%%%%%%%%%%%%%%%%%%%%%%%%%%%%%%%%%%%%%%%%%%%%%%%%%%%%%%%%%%%%%%%%%%%%%%%%%%%%%%%%%%%%%%%%%

% \medskip
% \subsection*{Acknowledgments}
% The research has been supported by the Open Research Fund of Key Laboratory of Nonlinear Analysis $\&$ Applications (Central China Normal University), Ministry of Education, P.R. China.

%%%%%%%%%%%%%%%%%%%%%%%%%%%%%%%%%%%%%%%%%%%%%%%%%%%%%%%%%%%%%%%%%%%%%%%%%%%%%%%%%%%%%%%%%%%%%%%%%%%%%%%%%%%%%%%%%%%%

%\newpage
%\bibliographystyle{siam}
%\bibliography{ref}

\smallskip

\end{document}